\documentclass[12pt,twoside]{article}
\usepackage{graphicx}
\usepackage{epstopdf}
\usepackage{amsthm,amsfonts,amssymb,amsmath}
\usepackage{lineno}
\usepackage{color}
\newtheorem{theorem}{Theorem}[section]

\newtheorem{lemma}[theorem]{Lemma}
\newtheorem{proposition}[theorem]{Proposition}

\newtheorem{remark}[theorem]{Remark}
\title{\Large\bf{Time adaptive numerical solution of a highly degenerate diffusion-reaction biofilm model based on regularisation}}
\author{Maryam Ghasemi and Hermann J. Eberl\\ \\
Department of Mathematics and Statistics \\
University of Guelph, Guelph, ON, N1G 2W1, Canada\\
}
\date{}
\begin{document}
\baselineskip=24pt
\maketitle

\begin{abstract}
We consider a quasilinear degenerate diffusion-reaction system that describes biofilm formation. The model exhibits two non-linear effects: a power law degeneracy as one of the dependent variables vanishes and a super diffusion singularity as it approaches unity. Biologically relevant solutions are characterized by a moving interface and gradient blow-up there.

Discretisation of  the PDE in space by a standard Finite Volume scheme leads to a singular system of ordinary differential equations. We show that regularisation of this system allows the application of error controlled adaptive integration techniques to solve the underlying PDE. This overcomes the major limitation of existing methods for this type of problem which work with fixed time-steps.

We apply the resulting numerical method to study the effect of signal diffusion in the aqueous phase on quorum sensing induction in a biofilm colony.


\end{abstract}

\section{Introduction}
\label{intro}

Bacterial biofilms are microbial communities on immersed surfaces, embedded in layers of self-produced extracellular polymeric substances (EPS), which protect the sessile cells against mechanical or chemical washout \cite{DFF:1996,HCS:2004,IS:2014,LL:2012,MYT:2008,WK:2000}. Biofilms are prevalent in natural, industrial and hospital settings. Depending on the context, they can be harmful or beneficial. For example biofilms can lead to corrosion problems in fresh water pipes, and oil pipelines \cite{SC:2001}; foremost cause of failure of medical implants are bacterial infections caused by biofilm formation \cite{AFS:2000,SLA:2008}; biofilms can lead to crop disease in plants \cite{DC:2002}; dental plaque is an oral biofilm on teeth that causes dental disease \cite{ABO:2007}. On the other hand, the adsorption and absorption properties and enhanced mechanical stability of biofilms make them advantageous to environmental engineering technologies, e.g. for waste water treatment, remediation of contaminated soil, and elimination of petroleum oil from contaminated oceans or marine systems \cite{WEV:2006}.

Although the term biofilm suggests a homogeneous film-like layer, biofilms on the meso-scale ($10\mu m \sim 1 mm$, the~actual~biofilm~scale) in reality are spatially heterogeneous assemblages of colonies which may merge as they grow and expand \cite{LHEKP:2002}. The architecture of a biofilm depends on some biological factors such as maximum cell density, specific growth rate, and local nutrient availability. 
In  environments with unlimited amount of substrate, biofilms tend to grow quickly and form homogeneous, compact, thick layer. On the other hand, biofilms in nutrient limited regimes tend to form in patchy heterogeneous structures \cite{EPV:2001}.

Several mathematical models have been proposed in the literature to describe the growth of spatially heterogeneous biofilms, ranging from stochastic individual based models to cellular automata models to deterministic continuum models, cf \cite{KD:2010,LHEKP:2002,WZ:2010,WEV:2006} and the references that they cite. The underlying mathematical principles of these models are quite different and, therefore, the mathematical and computational challenges vary from model to model. Nevertheless, they all show the same qualitative behaviour in predicting the development of biofilm morphologies in response to substrate limitation \cite{WEV:2006}.

An early prototype biofilm growth model has been introduced in \cite{EPV:2001}, and was extended later to account for further biofilm aspects and processes, such as quorum sensing, response to antibiotics, internally triggered dispersal, EPS production, and multispecies systems, e.g. in \cite{EC:2009,ES:2008,EHK:2015,FKH:2010,FKH:2011,KEE:2009,MD:2014,RSE:2015}. The model has been derived both from the view point of biofilms as spatially structured populations \cite{KhHE:2009} and as a fluid \cite{FKH:2010,NS:2016}, thus reflecting the ecological-mechanical duality of biofilms. In its basic form as a biofilm growth model, it consists of a density-dependent diffusion-reaction equation for biomass that is coupled with a semi-linear diffusion-reaction equation for growth limiting nutrients. The biomass equation shows two interacting non-linear diffusion effects: (i) the diffusion coefficient vanishes where the biomass vanishes as in the porous medium equation, and  (ii) a super-diffusion singularity as the dependent variable in the diffusion coefficient approaches the known maximum density. 
The interplay of both effects assures that the solution of the biomass equation is bounded by the maximum cell density \cite{EZE:2009}, that spatial expansion of the biofilm only takes place if locally no space is available to produce new biomass (i.e. a volume filling effect), and that initial data with compact support have solutions with compact support \cite{EPV:2001}. The latter property is a characteristic property of the porous medium equation and related problems. In our application this expresses itself as a sharp biofilm/water interface, at which the biomass gradient blows up. Since  biofilm colonies grow, their interfaces with the aqueous phase are not stationary but change over time. Eventually neighbouring colonies can merge into a bigger colony, in which case their interfaces merge and dissolve.

Interface problems of this type pose difficulties for numerical and analytical methods.
The steep biomass gradients at the interface can lead to interface smearing or spurious oscillations in numerical solutions. The singularity in the density-dependent diffusion coefficient leads to a blow up of model coefficients and forces the numerical method to work with very small time-steps.  
Several methods have been proposed in the literature for the numerical treatment of this highly non-linear degenerate model.

In \cite{EPV:2001} a hybrid time-integration strategy was applied. The slower biomass processes were solved explicitly with a Runge-Kutta-Fehlberg method, whereas the faster nutrient processes were solved implicitly. This required inefficiently small time steps for the biomass equation when the biomass density approaches maximum density somewhere. 
A different idea was explored in \cite{DE:2006} in which the degenerate equation is transformed into new dependent variables such that the spatial operator converted into the Laplacian operator and all non-linear effects appeared in the time-derivative. 
A fully implicit, fixed time-step method for the new equation could handle the diffusion singularity effect very well, but emphasized interface oscillations. 
In \cite{KE:2006} a fully adaptive in space and time method was proposed in the 1D case, based on a weak moving frame formulation.
This method explicitly tracked the biofilm/water interface and was able to deal well with the porous medium degeneracy but had problems with the super-diffusion singularity. In \cite{KE:2008} the transformation idea of \cite{DE:2006} was combined with the moving frame approach of \cite{KE:2006}. This method worked very well in 1D but in the 2D case becomes cumbersome; moreover, it was not obvious how it could be adapted for multi-species biofilm systems. 

A semi-implicit method that is based on a non-local (in time) representation of the non-linear density dependent  biomass flux is described in \cite{ED:2007,SH:2012}. This method can be efficiently parallelised for shared memory architectures \cite{ME:2010} and has been used for several extensions and applications of the underlying biofilm model \cite{EC:2009,ES:2008,EHK:2015,FKH:2010,FKH:2011,KEE:2009,RE:2014}. This low order method has shown to be robust and to handle the steep biomass gradients at the biofilm/water interface reasonably well.  
This method has also been studied and used for a variant of the biofilm model with a slightly different biomass diffusion coefficient in \cite{NS:2016}.
Minor adaptations and extensions of this approach were suggested in \cite{MD:2014,MH:2012}, however without major documented improvement. 
A related fully-implicit first order time integration method was proposed in \cite{E:2016} for a cellulosic biofilm system in which nutrients are stationary and do not diffuse. It uses a fixed-point iteration at each time-step to solve the resulting non-linear system. This method reduces to the semi-implicit approach of \cite{ED:2007} as a special case if only one iteration step is carried out. In \cite{BLV} the implicit trapezoidal rule was used for time integration in the spatially one-dimensional case and explicit variants were proposed in \cite{MR:2013,S:2015}, however, without demonstrated gain.
 All these methods use fixed time steps and because of the low regularity of the solution they do not possess error control properties for time integration. Choosing the time step such that an acceptable trade-off between accuracy and compute time is achieved, and that a breakdown of the method near the super-diffusion singularity is avoided, requires some user experience. 

In this paper we want to overcome the constraints around time integration. 
Our problem at hand has two sources of stiffness. One stems from the disparity of characteristic time scales of substrate diffusion and uptake on one hand, and biomass growth on the other hand \cite{EPV:2001}. The other one comes into play when the biomass density approaches the super-diffusion singularity. For stiff ordinary differential equations with sufficiently smooth solutions, a host of such methods exists. Due to the singularity in the biomass diffusion coefficient and the resulting missing regularity of solutions, it is not {\it a priori} clear whether these can be applied for the time integration of the biofilm equation. Furthermore, it will be important that the resulting numerical scheme is able to deal with sharp biomass gradients at the biofilm/water interface without introducing spurious oscillation or extensive interface smearing. To address these questions, and to propose a fully adaptive method for time integration of the biofilm model is our main objective. 

We will demonstrate the utility of this method by studying a question that arises in the context of quorum sensing in biofilms.
Bacterial cells produce and secrete chemical signals to communicate with each other, and measure the concentration of the signal molecules in their environment. This mechanism is known as quorum sensing. 
Once a critical signal concentration is reached, changes in gene expressions are induced and bacteria collectively synchronize their behaviour \cite{FKH:2011,HKM:2007}. Since the producer cells respond to their own signal these are sometimes referred to as autoinducers (e.g., acyl-homoserine lactones (AHL) in Gram-negative bacteria).
In most bacterial autoinducer systems, the autoinducer synthase gene is upregulated, i.e a positive feedback stimulus is induced, leading to production of AHL molecules at an increased rate \cite{FKH:2011}.

The name quorum sensing stems from an interpretation of autoinduction as a mechanism to estimate cell density. However, in environments that are not completely mixed, such as spatially structured biofilm systems, the autoinducer concentration is affected by transport of chemical signals in the aqueous environment. In fact they might use the autoinducer concentration as an estimate of any of the factors that can affect the concentration of this molecule, such as diffusion limitation \cite{HKM:2007}. This has lead to the interpretation of quorum sensing as diffusion sensing  and a mechanism to survey the environment \cite{R:2002}.
An important question to ask is then how the environment affects the time to autoinduction in a biofilm colony, and the size of the colony that is required for autoinduction, i.e. for the signal concentration to exceed induction threshold. 

In a recent study  \cite{TSZ:2014} it was investigated experimentally using a synthetic biofilm matrix and mathematically using a simple diffusion model, how much autoinducer needs to be added instantaneously in a colony center to achieve concentrations in the biofilm above induction threshold for different colony sizes.  Our goal is to build on this using a model that accounts also for biofilm growth and dynamic autoinducer production, and that also considers diffusion of signal molecules into the environment. Related, older studies are \cite{CH:2003,CKMP:2002}, where the relationship between biofilm size and induction was studied in a one-dimensional setting, using a Wanner-Gujer type biofilm growth model. The 1D setting in that study, however, is limiting in that it does not allow to investigate effects of signal diffusion parallel to the substratum, and thus effects of neighboring colonies or confined domains. How convective transport of signal molecules due to bulk flow hydrodynamics affects washout of autoinducers and the onset of induction in narrow flow channels, on the other hand, was simulated in \cite{FKH:2010,VSC:2010}; the presence of background flow added another complexity that does not allow to hone in on aspects of diffusion sensing.

Our paper is organized as follows: In Section \ref{mathmodel} we introduce the prototype biofilm model and obtain its semi-discrete version by applying a Finite Volume Method to discretize the partial differential equations in space. In Section \ref{sec3} we analyze the semi-discrete model; in particular we show, using regularization and the method of super and sub solutions, that the solution of the resulting ODE never attains the singularity and is thus sufficiently smooth to apply standard numerical methods for time integration, such as embedded Rosenbrock-Wanner methods, which we will use.
In order to demonstrate that the resulting numerical method is able to handle gradient blow-up at the biofilm/water interface, we apply in section \ref{sec4} the numerical scheme to a simpler problem (without super-diffusion singularity), for which an exact solution is known, namely the porous medium equation with linear growth. 
In section \ref{sec5} we investigate the behaviour of our numerical method for the full prototype biofilm growth model. In particular we give a grid refinement study and explore quantitatively the dependence of the numerical method on the regularization parameter that was introduced for analytical treatment.
In section \ref{DS}, finally we present an illustrative application of the method in the context of quorum sensing. Concluding remarks are provided in Section \ref{conclusion}.

%%%%%%%%%%%%%%%%%%%%%%%%%%%%%%%%%%%%%%%%%%%%%%%%%%%%%%%%%%%%%%%

\section{Mathematical Model}\label{mathmodel}

\subsection{Governing equations}

The mathematical model studied in this paper was originally introduced in \cite{EPV:2001}. It is formulated as a density-dependent degenerate diffusion-reaction equation over domain $\Omega\subset \mathbb{R}^2$. The dependent variables are the volume fraction occupied by biomass, $u$, and the concentration of a growth limiting nutrient, $c$.  Domain $\Omega$ is divided into region $\Omega_1(t)=\left\lbrace (x,y)\in \Omega \subset \mathbb{R}^2: u(t,x,y)=0\right\rbrace$ that describes the aquatic phase (bulk liquid, channels and pores of a biofilm) without biomass, and region $\Omega_2(t)=\left\lbrace (x,y)\in \Omega \subset \mathbb{R}^2: u(t,x,y)>0\right\rbrace$, which is the actual biofilm with positive density, cf. Figure \ref{figbio2}. 

\begin{figure}[h!]
\centering
\includegraphics[scale=0.4]{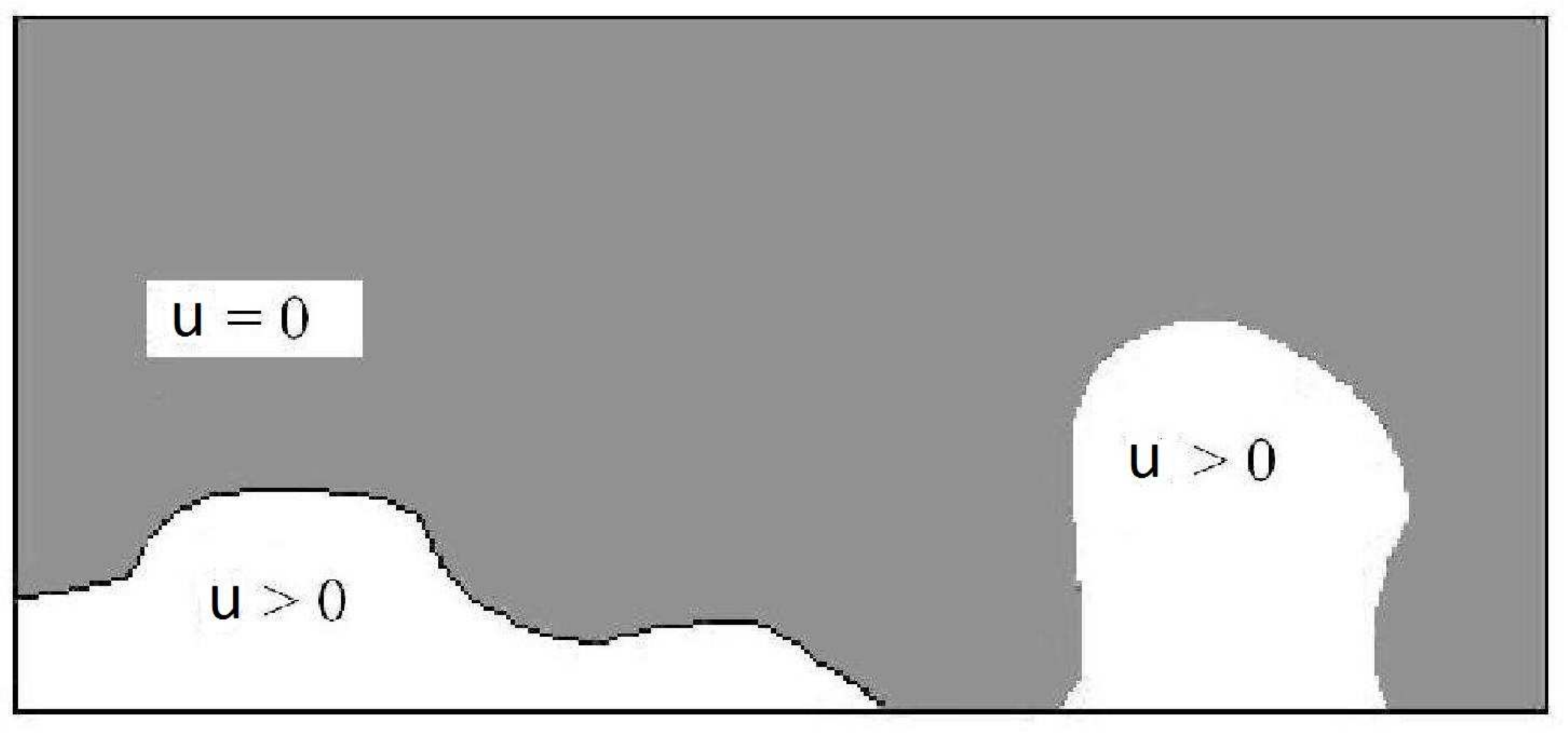}
\caption{The domain $\Omega\subset \mathbb{R}^2$ with liquid region $\Omega_1(t)=\left\lbrace (x,y)\in \Omega \subset \mathbb{R}^2: u(t,x,y)=0\right\rbrace$ and biofilm region $\Omega_2(t)=\left\lbrace (x,y)\in \Omega \subset \mathbb{R}^2: u(t,x,y)>0\right\rbrace$.}
\label{figbio2}
\end{figure}

The independent variables $t\geq0$ and $(x,y)\in \Omega$ denote here time and spatial location, respectively. Both regions are separated by the biofilm/water interface $\Gamma(t) =\partial\Omega_{1}\cap\partial\Omega_{2}$ which changes as the biofilm grows, i.e. as $u$ changes.

As is usual in mathematical models of biofilms, the EPS that is produced by the bacteria is subsumed in the biofilm volume fraction in this prototype biofilm model. An extension of this model that accounts for EPS explicitly has been proposed, for example, in \cite{FKH:2011}.

The model in non-dimensionalised form reads \cite{EPV:2001}:
\begin{equation}\left\{%
\begin{array}{lllll}
\frac{\partial u}{\partial t}=\nabla(D(u)\nabla u)+\frac{c}{K_{U}+c}u-ku,\\
\frac{\partial c}{\partial t}=d_c\Delta c-\frac{\nu_U c}{K_{U}+c}u,\label{21}
\end{array}%
\right.
\end{equation}
$D(u)$ is a density-dependent diffusion coefficient defined as:
\begin{equation}\label{22}
D(u)=\delta\frac{u^{\alpha}}{(1-u)^{\beta}},~\alpha,\beta\geq1,~d_c\gg\delta>0.
\end{equation}
Diffusion coefficient $D(u)$ shows two non-linear effects: (i) a power law degeneracy as in the porous medium equation, i.e. $D(u)$ vanishes as $u$ vanishes and (ii) a super diffusion singularity as $u$ approaches to unity. The porous medium degeneracy, $u^{\alpha}$, guarantees that the biofilm does not spread notably if the biomass density is small, $u\ll 1$, and it is also responsible for the formation of a sharp interface between biofilm and surrounding liquid, i.e. initial data with compact support lead to solutions with
compact support. The second effect (ii) enforces the solution to be bounded by unity \cite{ED:2007,WEV:2006}. This is counteracted by the degeneracy as $u = 0$ at the interface. Consequently, $u$ squeezes in the biofilm region and approaches its maximum value 1. Hence, the interaction of both non-linear diffusion effects with the growth term is needed to describe spatial biomass spreading \cite{ED:2007}.
Diffusion of the dissolved nutrient is assumed to be Fickian, i.e. $d_c$ is a constant. Biomass spreading is much slower than diffusion of dissolved substrate, thus the biomass motility coefficient, $\delta$, is several orders of magnitude smaller than the substrate diffusion coefficient. 

In (\ref{21}), parameter $k$ is the cell lysis rate, $K_U$ the Monod half saturation concentration, and $\nu_U$ the maximum substrate uptake rate.

To study the mathematical model (\ref{21}) we restrict ourselves to the rectangular domain $\Omega=[0,L]\times[0,H]$. The substratum, on which biofilm colonies form is the bottom boundary, $y=0$. We assume the substratum is impermeable to biomass and dissolved substrate. At the lateral boundaries, $x=0$ and $x=L$, we assume a symmetry boundary condition for both dependent variables, which allows us to view the domain as a part of a continuously repeating larger domain. At the top boundary, $y=H$, we pose a homogeneous Dirichlet condition for the biomass and a Robin condition for the nutrient which reflects that the substrate is added to the system through this segment of the domain boundary. Thus the imposed boundary conditions on domain $\Omega=[0,L]\times[0,H]$ are defined as:
 \begin{equation} \label{23}
\left\{%
\begin{array}{lllll}
 &&\partial_nu=\partial_nc=0~ \ \mbox{at} \ x=0, x=L~~~\mbox{and}~~~y=0,\\
 &&u=0,~~~~c+\lambda\partial_n c=1~\ \mbox{at} \ y=H,
 \end{array}%
\right.
\end{equation}
Here $\lambda$ can be understood as the external concentration boundary layer thickness, and $\partial_n$ denotes the outward normal derivative. The concentration boundary layer is introduced to mimic the convective contribution of external bulk flow to substrate supply, i.e. it is
related to the bulk flow velocity, in the sense that a small bulk flow velocity implies a thick concentration boundary layer, while a thin concentration boundary layer represents fast bulk flow \cite{EC:2009}.

For the reader's convenience we recall the main result of \cite{EZE:2009} that studies the longtime behaviour of PDE model (\ref{21}), without proof:
\begin{theorem}[Existence]\label{PDEexist:thm}
Let the initial data $(u_0,c_0)$ satisfy the conditions
\begin{equation}\left\{%
\begin{array}{lllll}
1.~c_0\in L^{\infty}(\Omega)\cap H^1(\Omega),~~0\leq c_0 \leq 1,~c_0\mid_{\partial\Omega}=1,\\
2.~u_0\in L^{\infty}(\Omega),~F(u_0)\in H_0^1(\Omega),\\
3.~u_0\geq0,~\|u_0\|_{\small{L^{\infty}}}<1,\label{26}
\end{array}%
\right.
\end{equation}
where
\begin{equation}
F(u)=F_{\infty}(u):=\int_{0}^{u} \frac{v^{\alpha}}{(1-v)^{\beta}}dv,~0\leq u<1.
\end{equation}\label{27}
Then, there exists a solution $(u(t),c(t))$ of degenerate problem (\ref{21}) (in the sense of distributions) belonging to the following class:
\begin{equation}\left\{%
\begin{array}{lllll}
1.~u,c\in L^{\infty}(\mathbb{R_+}\times\Omega)\cap C([0,\infty),L^2(\Omega)),\\ 
2.~c,F(u)\in L^{\infty}(\mathbb{R_+},H^1(\Omega))\cap C([0,\infty),L^2(\Omega)),\\
3.~0\leq u(t,x,y),c(t,x,y)\leq 1,~\|u\|_{\small{L^{\infty}(R_+\times\Omega})}<1.
\end{array}%
\right.\label{28}
\end{equation}
\end{theorem}

\begin{theorem}[{Uniqueness}]\label{PDEunique:thm}
Let $(u_1(t),c_1(t))$ and $(u_2(t),c_2(t))$ be two solutions of (\ref{21}) belonging to the class (\ref{28}). Then, the following estimate is valid:
\begin{eqnarray}
&&\| c_1(t)-c_2(t)\|_{\small{L^1(\Omega)}}+\| u_1(t)-u_2(t)\|_{\small{L^1(\Omega)}} \nonumber\\
&&\leq e^{K}t \left( \| c_1(0)-c_2(0)\|_{\small{L^1(\Omega)}}+\| u_1(0)-u_2(0)\|_{\small{L^1(\Omega)}}\right),\label{29}
\end{eqnarray}
for some positive constant $K$ that depends on reaction parameters.
In particular, the solution of (\ref{21}) is unique in the class (\ref{28}).
\end{theorem}

An extension of these results to more general boundary conditions is described in \cite{EZE:2009}.

\subsection{Spatial discretization}\label{sec2.2}
We introduce a uniform grid of size $N \times M$ for the rectangular domain $[0,L]\times [0,H]$. Integrating the first equations of (\ref{21}) over each grid cell and using the Divergence Theorem yields
\begin{eqnarray}\label{FVM:eq}
&&\frac{d}{dt}\int_{v_{i,j}}udxdy=\int_{\partial v_{i,j}}J_nds+\int_{v_{i,j}}R(c)udxdy, \quad i=1,...,N,  j=1,...,M\nonumber\\
&&
\end{eqnarray}
where $v_{i,j}$ denotes the domain of the cell with grid index $(i,j)$, $J_n = D(u)\partial_nu$ denotes the outward normal flux across the grid cell boundary, and $R(c)=c/(K_U+c)-k$ stands for the reaction terms.

To evaluate the area integrals in (\ref{FVM:eq}), we evaluate the dependent variables at the center of the grid cells, 
\begin{equation}
U_{i,j}(t):=u(t,x_i,y_j)\approx u\left(t,\left(i-\frac{1}{2}\right)\Delta x,\left(j-\frac{1}{2}\right)\Delta x\right)
\end{equation}
and similarly for the nutrient concentration
\begin{equation}
C_{i,j}(t):=c(t,x_i,y_j)\approx c\left((t,\left(i-\frac{1}{2}\right)\Delta x,\left(j-\frac{1}{2}\right)\Delta x\right).
\end{equation}
for $i=1,...,N$ and $j=1,...,M$ with $\Delta x=L/N=H/M$, and we approximate the integrals
by the midpoint rule.
 Similarly, the line integral in (\ref{FVM:eq}) is evaluated by considering every edge of the grid cell separately, using the midpoint rule.   To this end, the diffusion coefficient $D(u)$ in the midpoint of the cell edge is approximated by arithmetic averaging from the neighbouring grid cell center points, and the derivative of $u$ across the cell edge by a central finite difference.
We get then for the biomass density in the grid cell center the ordinary differential equation
\begin{equation}\label{211}
\frac{d}{dt}U_{i,j}=\frac{1}{\Delta x} \left(J_{i+\frac{1}{2},j}+J_{i-\frac{1}{2},j}+J_{i,j+\frac{1}{2}}+J_{i,j-\frac{1}{2}}\right)+ R_{i,j}U_{i,j},
\end{equation}
where $R_{ij}=C_{i,j}/(K_U+C_{i,j})-k$ and for the fluxes we have, accounting for the boundary conditions,
\begin{equation}\label{210a}
J_{i+\frac{1}{2},j}=\left\{\begin{array}{ll}
                             \frac{1}{2\Delta x}\big(D(U_{i+1,j})+D(U_{i,j})\big)(U_{i+1,j}-U_{i,j})  &  \mbox{for}\quad i<N, \\
                             -\frac{2}{\Delta x} D(0) U_{N,j}  & \mbox{for}\quad i=N,
                           \end{array}\right. 
\end{equation}
\begin{equation}\label{Uflux2:eq}
J_{i-\frac{1}{2},j}=\left\{\begin{array}{ll}
                             0 & \mbox{for}\quad i=1,\\
                             \frac{1}{2\Delta x}\big(D(U_{i,j})+D(U_{i-1,j})\big)(U_{i-1,j}-U_{i,j})  &  \mbox{for}\quad i>1,
                           \end{array}\right.
\end{equation}
\begin{equation}\label{Uflux3:eq}  
J_{i,j+\frac{1}{2}}=\left\{\begin{array}{ll}
                           \frac{1}{2\Delta x}\big(D(U_{i,j+1})+D(U_{i,j})\big)(U_{i,j+1}-U_{i,j}) & \mbox{for}\quad j<M,\\
                           0 & \mbox{for}\quad j=M,
                          \end{array}\right.     
\end{equation}
\begin{equation}\label{Uflux4:eq}
J_{i,j-\frac{1}{2}}=\left\{\begin{array}{ll}
                           0 & \mbox{for}\quad j=1,\\
                          \frac{1}{2\Delta x}\big(D(U_{i,j-1})+D(U_{i,j})\big)(U_{i,j-1}-U_{i,j}),& \mbox{for}\quad j>1.
                          \end{array}\right.     
\end{equation}

The spatial discretization of the equation for the nutrient follows the same principle, with the major difference being that at the top of the domain we have a Robin boundary condition instead of homogeneous Dirichlet condition. That $d_c$ is constant simplifies the flux expressions. We have then
\begin{equation}
\frac{d}{dt}C_{i,j}=\frac{1}{\Delta x} \left(\hat J_{i+\frac{1}{2},j}+ \hat J_{i-\frac{1}{2},j}+\hat J_{i,j+\frac{1}{2}}+ \hat J_{i,j-\frac{1}{2}}\right)- \hat R_{i,j}U_{i,j}
\end{equation}
with
\begin{equation}
\hat R_{i,j} = \frac{\nu_U C_{i,j}}{K_U+C_{i,j}}
\end{equation}
and
\begin{equation}
\hat J_{i+\frac{1}{2},j}=\left\{\begin{array}{ll}
                             \frac{d_c}{\Delta x}(C_{i+1,j}-C_{i,j})  &  \mbox{for}\quad i<N, \\
                             \frac{d_c}{\Delta x}\left(\frac{2 \Delta x}{2\lambda+\Delta x} - C_{i,j} \left(1+\frac{\triangle x-2\lambda}{\Delta x+2\lambda}\right)\right)   & \mbox{for}\quad i=N,
                           \end{array}\right. \label{BCtopC:eq}
\end{equation}
\begin{equation}
\hat J_{i-\frac{1}{2},j}=\left\{\begin{array}{ll}
                             0 & \mbox{for}\quad i=1,\\
                             \frac{d_c}{\Delta x}(C_{i-1,j}-C_{i,j})  &  \mbox{for}\quad i>1,
                           \end{array}\right.
\end{equation}
\begin{equation}      
\hat J_{i,j+\frac{1}{2}}=\left\{\begin{array}{ll}
                           \frac{d_c}{\Delta x}(C_{i,j+1}-C_{i,j}) & \mbox{for}\quad j<M,\\
                           0 & \mbox{for}\quad j=M,
                          \end{array}\right.     
\end{equation}
\begin{equation}
\hat J_{i,j-\frac{1}{2}}=\left\{\begin{array}{ll}
                           0 & \mbox{for}\quad j=1,\\
                          \frac{d_c}{\Delta x}(C_{i,j-1}-C_{i,j}),& \mbox{for}\quad j>1.
                          \end{array}\right.     
\end{equation}

By introducing the lexicographical grid ordering 
\begin{equation}\label{gridordering:eq}
\pi :\left\lbrace 1,...,N\right\rbrace\times \left\lbrace 1,...,M\right\rbrace\rightarrow \left\lbrace1,...,NM\right\rbrace,~(i,j)\mapsto p=(i-1)N+j
\end{equation}
and the vector notation $\textbf{U}=(U_1,...,U_{NM})$, $\textbf{C}=(C_1,...,C_{NM})$
with $U_{p}:=U_{\pi(i,j)}=U_{i,j}$, $C_{p}:=C_{\pi(i,j)}=C_{i,j}$ for $i=1,...,N$, $j=1,...,M$,
we arrive at the coupled system of $2 \cdot N \cdot M$ ordinary differential equations
\begin{equation}\label{220}
\left\{%
\begin{array}{lllll}
\frac{d\textbf{U}}{dt}=\mathfrak{D}(\textbf{U})\textbf{U}+\mathfrak{R}_{U}(\textbf{C})\textbf{U}\\%+\textbf{b}_1\\
\frac{d\textbf{C}}{dt}=\mathfrak{L}\textbf{C}-\mathfrak{R}_{C}(\textbf{C})\textbf{U}+\textbf{b}.
\end{array}%
\right.
\end{equation}

\begin{remark}\label{Rem1}
By construction the $NM \times NM$ matrices $\mathfrak{D}(U), \mathfrak{L}$ are symmetric, and weakly diagonally dominant with non-positive main diagonals and non-negative off-diagonals. They contain the spatial derivative terms.
They can efficiently be stored in sparse diagonal format with 4 off-diagonals, with offsets $\pm 1, \pm M$. The matrices $\mathfrak{R}_{U}(\textbf{C}), \mathfrak{R}_{C}(\textbf{C})$ are diagonal matrices which contain the reaction terms; in particular the $p$th entry in $\mathfrak{R}_{C}(\textbf{C})$ vanishes if $C_p=0$. The vector $\textbf{b}$ contains contributions from the Robin boundary conditions in (\ref{BCtopC:eq}); its entries are zero for all grid cells $(i,j)$ with $i<N$, and $b_{\pi(i,j)}=\frac{d_c}{\Delta x}\frac{2}{2\lambda+\Delta x}  >0$ for grid points with $i=N$.
\end{remark}
\begin{remark}
The discretization can easily be extended to more general boundary conditions, such as Robin conditions everywhere, in analogy to (\ref{BCtopC:eq}).
\end{remark}
%%%%%%%%%%%%%%%%%%%%%%%%%%%%%%%%%%%%%%%%%%%%%%%%%%%%%%%

\section{The regularized semi-discrete system}\label{sec3}

Due to the singularity in the biomass diffusion coefficient (\ref{22}) standard arguments from the theory for ordinary differential equations, such as the Picard-Lindel\"of theorem and the tangent condition cannot be readily applied to study the well-posedness of the ODE system (\ref{220}) and to prove that the biomass volume fraction $\textbf{U}$ is indeed bounded by unity, as is required for physical reasons and for agreement with the underlying PDE (\ref{21}), as per Theorem \ref{PDEexist:thm}. To overcome this limitation we regularize the semi-discretised system.
Following \cite{EZE:2009}, we introduce the regularised density-dependent biomass diffusion coefficient
\begin{equation}\label{25}
D_{\epsilon}(u)=\left\{%
\begin{array}{ll}
\delta\epsilon^{\alpha}&u<0\\
\delta\frac{(u+\epsilon)^{\alpha}}{(1-u)^{\beta}}&0\leq u\leq1-\epsilon\\
\delta\epsilon^{-\beta}&u\geq1-\epsilon.
\end{array}%
\right.,
\end{equation}
where the continuous extension of $D_\epsilon$ for negative $u$ is for technical reasons. We show lateron, that indeed $u\geq 0$ for the solutions of our model.
We define the following regularized version of ODE system (\ref{220}),  
\begin{equation}\label{31}
\left\{%
\begin{array}{lll}
\frac{d\textbf{U}^{\epsilon}}{dt}&=&\mathfrak{D_{\epsilon}}(\textbf{U}^{\epsilon})\textbf{U}^{\epsilon}+\mathfrak{R}_{U}(\textbf{C}^{\epsilon})\textbf{U}^{\epsilon}\\
\frac{d\textbf{C}^{\epsilon}}{dt}&=&\mathfrak{L}\textbf{C}^\epsilon-\mathfrak{R}_{C}(\textbf{C}^{\epsilon})\textbf{U}^{\epsilon}+\textbf{b},
\end{array}%
\right.
\end{equation}
where matrix $\mathfrak{D_{\epsilon}}(\cdot)$ is defined like $\mathfrak{D}(\cdot)$, only with $D(u)$ in the flux terms (\ref{210a})-(\ref{Uflux4:eq}) replaced by $D_\epsilon(u)$ as defined in (\ref{25}).

\begin{remark} The regularised semi-discretised system (\ref{31}) is a standard, 2nd order (in space) convergent, finite difference approximation of the regularised partial differential equation system
\begin{equation}\label{regPDE:eq}
\left\{%
\begin{array}{lllll}
\frac{\partial u_\epsilon}{\partial t}=\nabla(D_\epsilon(u_\epsilon)\nabla u_\epsilon)+\frac{c_\epsilon}{K_{U}+c_\epsilon} u_\epsilon-ku_\epsilon,\\
\frac{\partial c_\epsilon}{\partial t}=d_c\Delta c_\epsilon-\frac{\nu_U c_\epsilon}{K_{U}+c_\epsilon}u_\epsilon,
\end{array}%
\right.
\end{equation}
the solutions of which are smooth. It was shown in \cite{EZE:2009} that the solutions of the initial value problem of (\ref{regPDE:eq}) converge to the solution of the corresponding initial value problem of (\ref{21}) as $\epsilon \rightarrow 0$, and that $u_\epsilon\leq 1-\xi$ for some $\xi>0$ and sufficiently small $\epsilon$.
\end{remark}

Going forward, we define vector inequalities component wise, i.e for two vectors $\textbf{U}=(U_1,...,U_n)^T$  and $\textbf{V}=(V_1,...,V_n)^T$ the inequality $\textbf{U}\leq \textbf{V}$ means $U_p \leq V_p$ for all $p=1,...,n$, and accordingly for strict inequalities. Furthermore, by $\mathbf{1}\in \mathbb{R}^n$ we denote the unity vector $(1,...,1)^T$, and by $\mathbf{0} \in \mathbb{R}^n$ we denote the vector $(0,...,0)^T$.

We show that the initial value problem of (\ref{31}) with initial data for $\textbf{U}^\epsilon$ and $\textbf{C}^\epsilon$ such that $0\leq \textbf{U}^\epsilon(0) < \textbf{1}$ and  $0\leq \textbf{C}^\epsilon (0) \leq \textbf{1}$ posseses a unique solution that is non-negative and bounded.

\begin{proposition}\label{pro1}
Let $0<\epsilon\ll 1$. Suppose the initial values $\textbf{U}^{\epsilon} (0)$ and $\textbf{C}^{\epsilon}(0)$ of the regularized ODE system (\ref{31}) are non-negative and satisfy $\textbf{U}^{\epsilon} (0)\leq (1-\rho) \mathbf{1}$ with $\rho\in(0,1)$ and $\textbf{C}^{\epsilon}(0)\leq \mathbf{1}$. Then there exists a unique solution $(\textbf{U}^{\epsilon},\textbf{C}^{\epsilon})$ of the regularized problem (\ref{31}), which is non-negative and bounded by a constant for sufficiently small values of $\epsilon$. Also, an upper bound on the solution exists that does not depend on the regularization parameter $\epsilon$. 
\end{proposition}
\begin{proof}
The semi-discrete system (\ref{31}) satisfies the Lipschitz condition in the non-negative cone. Therefore, the initial value problem has a unique solution. Non-negativity of this solution follows with standard arguments, such as the tangent condition (cf \cite{walter}), from the properties of matrices  $\mathfrak{R}_{C}, \mathfrak{R}_{U}, \mathfrak{L}, \mathfrak{D}_\epsilon$ as per Remark \ref{Rem1}. That $\textbf{C}^{\epsilon} \leq \mathbf{1}$ follows with the non-negativity of $\mathbf{U}^\epsilon$ from the definitions of  $\mathfrak{R}_{C}, \mathfrak{L}, \textbf{b}$ again from the tangent condition.

To show the boundedness of $\textbf{U}^{\epsilon}$ we introduce the barrier function $\textbf{U}^{\theta}=\mathbf{1}+\vec{\theta}$, where the vector  $\vec{\theta}$ is the solution of the linear system $\mathcal{A}\vec{\theta}=\mathbf{1}$ with $\mathcal{A}= -\frac{1}{\delta \epsilon^{\alpha}}\mathfrak{D}_\epsilon(0)$. Note that this matrix $\mathcal{A}$ does not depend on $\epsilon$, whence also $\vec{\theta}$ is independent of $\epsilon$.

By construction, $\mathcal{A}$ is invertible and weakly diagonally dominant with positive main diagonal entries and negative off-diagonal entries. Therefore, it is an M-matrix \cite{hackbusch}. Hence its inverse exists and is non-negative, which implies the non-negativity of $\vec{\theta}.$ Thus,
\begin{equation}
\mathbf{1} \leq \textbf{U}^{\theta} \leq \mathbf{1}+\textbf{q}\label{34}
\end{equation}
for some positive vector $\textbf{q}$. Since $\vec{\theta}$ does not depend on $\epsilon$, neither does $\textbf{q}$.  The initial data obviously satisfy 
\begin{equation}
\textbf{U}^{\epsilon}(0)\leq \textbf{U}^{\theta},
\end{equation}
%~~\textbf{U}^{\epsilon}\mid_{\small{\partial\Omega}}\leq \textbf{U}^{\theta}\mid_{\small{\partial\Omega}}
and for sufficiently small values of $\epsilon$ we have
\begin{eqnarray}
\frac{d\textbf{U}^{\theta}}{dt}&-&\mathfrak{D_{\epsilon}}(\textbf{U}^{\theta})\textbf{U}^{\theta}-\mathfrak{R}_{U}(\textbf{C}^{\epsilon})\textbf{U}^{\theta}\nonumber\\
&=&\delta\epsilon^{-\beta}\left(\mathcal{A}(\mathbf{1}+\vec{\theta})\right)-\mathfrak{R}_{U}(\textbf{C}^{\epsilon})\textbf{U}^{\theta}\nonumber\\
%&= &\delta\epsilon^{-\beta}(\mathcal{A}\mathbf{1}+\mathcal{A}\vec{\theta})-\mathfrak{R}_{U}(C^{\epsilon})\textbf{U}^{\theta}\nonumber\\
&= &\delta\epsilon^{-\beta}(\mathcal{A}\mathbf{1}+\mathbf{1})-\mathfrak{R}_{U}(\textbf{C}^{\epsilon})\textbf{U}^{\theta}\nonumber\\
&\geq & 0\nonumber\\
&=&\frac{d\textbf{U}^{\epsilon}}{dT}-\mathfrak{D_{\epsilon}}(\textbf{U}^{\epsilon})\textbf{U}^{\epsilon}-\mathfrak{R}_{U}(\textbf{C}^{\epsilon})\textbf{U}^{\epsilon}.\label{34b}
\end{eqnarray}
Note that $\mathcal{A}$ is a weakly diagonally dominant matrix such that $\mathcal{A}\mathbf{1}\geq0$.
Consequently,  $\textbf{U}^\theta$ is an upper solution of (\ref{31}) for sufficiently small $\epsilon$, wherefore then $\textbf{U}^\theta \geq \textbf{U}^\epsilon$, by the comparison theorem for quasimonotonic systems \cite{walter}.
Thus, inequality (\ref{34b}) implies that there exists $\hat \epsilon$ such that the solutions $\textbf{U}^{\epsilon}$ are uniformly bounded for all $0<\epsilon\leq\hat \epsilon.$ 
\end{proof}

We will use this result to improve the upper bound on $\mathbf{U}^\epsilon$. In particular we will show that $\mathbf{U}^\epsilon<\mathbf{1}$ for small enough $\epsilon$. 

According to Proposition \ref{pro1}, there is a constant $Q$ independent of $\epsilon$, such that $Q\geq \|\mathfrak{R}_U(\mathbf{C}^\epsilon)\mathbf{U}^\epsilon \|_\infty$. As before, we assume that there is a $0<\rho<1$ such that $\mathbf{U}^\epsilon(0)\leq \rho\mathbf{1}$. 
Let us denote by $\textbf{V}^\epsilon$ the solution of 
\begin{equation}\label{veps}
{d  \mathbf{V}^\epsilon \over dt}= \mathfrak{D}_\epsilon(\mathbf{V}^\epsilon)\mathbf{V}^\epsilon + \mathbf{Q} +\mathbf{b}, \quad \mathbf{V}^\epsilon(0) =\rho\mathbf{1}
\end{equation}
where vector $\mathbf{Q} = Q\mathbf{1} $ and vector $\mathbf{b}$ is a boundary correction that one obtains if imposing the inhomogeneous Dirichlet boundary condition $u=\rho$ at $y=H$ instead of the homogeneous one, defined in analogy with equation (\ref{220}) and Remark \ref{Rem1}. This ensures that at the boundary the $\textbf{V}^\epsilon$ bounds $\textbf{U}^\epsilon$ from above.

We have the following relationship between $\mathbf{U}^\epsilon$ and $\mathbf{V}^\epsilon$:

\begin{lemma}\label{1:lem} Let $\mathbf{V}^\epsilon$ be the solution of (\ref{veps}) and $\mathbf{U}^\epsilon$ be the solution of (\ref{31}) with initial data $\mathbf{U}^\epsilon(0)\leq \rho\mathbf{1}$. Then  $\mathbf{V}^\epsilon(t) > \mathbf{U}^\epsilon(t) \geq \mathbf{0}$ for all $t>0$.
\end{lemma}
\begin{proof} This follows with standard comparison theorems for essentially positive ordinary differential equation systems, cf \cite{walter}, applied to the first equation in (\ref{31}). By introducing the defect $\mathfrak{P}:=\frac{d}{dt}-\mathfrak{D}_\epsilon(.) - \mathfrak{R}_U(\cdot)$, we have $\mathfrak{P}\mathbf{V}_\epsilon(t)> \mathfrak{P}\mathbf{U}_\epsilon(t)\equiv 0 = \mathfrak{P}\textbf{0}$,  which gives $\mathbf{V}^\epsilon(t) > \mathbf{U}^\epsilon(t) \geq \mathbf{0}$.
\end{proof}

For each vector index $p \in \left\lbrace1,...,NM\right\rbrace$ we denote by $(i_p, j_p)$ the corresponding location on the $N\times M$ grid, i.e. 
\begin{equation}\label{gridorderinginverse:eq}
(i_p,j_p)=\pi^{-1}(p), \quad p\in \{1,..., NM\},
\end{equation}
where $\pi^{-1}$ is the inverse  of the grid ordering map (\ref{gridordering:eq}).

The following observation allows us to reduce the analysis of the 2D system of $NM$ ordinary differential equations to a 1D problem consisting of $M$ ordinary differential equations, which will simplify notation.

\begin{lemma} \label{2:lem} The coefficients $V_{\epsilon,p}$ of the solution $\mathbf{V}^\epsilon$ of (\ref{veps}) satisfy
$V_{\epsilon,p}(t) = V_{\epsilon,j_p}(t)$ for $t>0$ and $p=1,...,NM$.
\end{lemma}
\begin{proof} This follows from the symmetry of boundary conditions of (\ref{veps}), the spatially homogeneous initial data and that the source term is a constant in (\ref{veps}), and from the uniqueness of solutions of (\ref{veps}). 
\end{proof}

With this observation and the definition of $\pi$ in (\ref{gridordering:eq}), we have $V_{\epsilon,\pi(i,j)}=V_{\epsilon,\pi(\hat i,j)} = V_{\epsilon,j}$ for all $\hat i =1,..., i-1,i+1,...,N$. Thus the ordinary differential equations determining the  functions $V_{\epsilon,j}(t)$ reduce to 
\begin{equation}\label{1D:sys}
\left\{\begin{array}{lcrl}
\frac{dV_{\epsilon,1}}{dt} &=& {1 \over 2\Delta x^2}&(D_\epsilon(V_{\epsilon,2})+D_\epsilon(V_{\epsilon,1}))  (V_{\epsilon,2}-V_{\epsilon,1}) + Q,\\
\\
\frac{dV_{\epsilon,j}}{dt} &=& {1 \over 2\Delta x^2}&[(D_\epsilon(V_{\epsilon,j+1})+D_\epsilon(V_{\epsilon,j}))  (V_{\epsilon,j+1}-V_{\epsilon,j}) \\
  &&  &-(D_\epsilon(V_{\epsilon,j-1}) +D_\epsilon(V_{\epsilon,j}))  (V_{\epsilon,j}-V_{\epsilon,j-1})] + Q, ~~ j=2,...,M-1\\
\\
\frac{dV_{\epsilon,M}}{dt} &=& {1 \over 2\Delta x^2}&[4D_\epsilon(\rho) (\rho-V_{\epsilon,M})  \\&&&
- (D_\epsilon(V_{\epsilon,M-1}) +D_\epsilon(V_{\epsilon,M}))  (V_{\epsilon,M}-V_{\epsilon,M-1})] +Q
\end{array}\right.\\
\end{equation}

For given $\epsilon, 0<\epsilon<1$, we furthermore define $\mathbf{Z}^\epsilon :=(Z_{\epsilon,1},...,Z_{\epsilon,M})^T\in \mathbb{R}^M$ as the solution of system of linear algebraic equations
\begin{equation}\label{1DStSt:sys}
\left\{\begin{array}{lcrl}
0 &=& {\delta \epsilon^{-\beta} \over \Delta x^2}&  (Z_{\epsilon,2}-Z_{\epsilon,1}) + Q,\\ \\
0 &=& {\delta \epsilon^{-\beta}  \over \Delta x^2}&[  Z_{\epsilon,j+1}-2Z_{\epsilon,j} +Z_{\epsilon,j-1}] + Q, \quad j=2,...,M-1\\ \\
0 &=& {\delta \epsilon^{-\beta}   \over \Delta x^2}&[2 (1-\epsilon) - 3Z_{\epsilon,M}+Z_{\epsilon,M-1}] + Q,\\
\end{array}\right.
\end{equation}

\begin{lemma} \label{3:lem} 
Let $\mathbf{V}^\epsilon(t)$ be the solution of (\ref{veps}) and $\mathbf{Z}^\epsilon$ be the solution of the linear system (\ref{1DStSt:sys}). 
If $\rho \leq 1-\epsilon$, then $\mathbf{V}^\epsilon(t) \leq \mathbf{Z}^\epsilon$ for all $t>0$. 
\end{lemma}
\begin{proof} We use the usual invariance theorem, cf. \cite{walter}, to show that the set $\mathcal{Z_\epsilon}= \prod\limits_{i=1}^M[\rho,Z_{\epsilon,i}]$ is positively invariant under the differential equation (\ref{1D:sys}):  In the equation for $V_{\epsilon,j}$ substitute $V_{\epsilon,j}=\rho$ and assume $\rho \leq V_{\epsilon,j\pm1}\leq Z_{\epsilon,j\pm1}$. Then $\frac{dV_{\epsilon,j}}{dt}\geq 0$, wherefore $V_{\epsilon,j}(t)\geq \rho$.
Similarly, in the equation for $V_{\epsilon,j}$ substitute $V_{\epsilon,j}=Z_{\epsilon,j}$  
and assume $\rho \leq V_{\epsilon,j\pm1}\leq Z_{\epsilon,j\pm1}$. Using $D_\epsilon(V_{\epsilon,j\pm1})\leq \delta \epsilon^{-\beta}$ and the definition of $Z_{\epsilon,j}$ in (\ref{1DStSt:sys}) gives $\frac{dV_{\epsilon,j}}{dt}\leq 0$, wherefore $V_{\epsilon,j}(t)\leq Z_{\epsilon,j}$.
Thus, the solution of the initial value problem (\ref{1D:sys}) remains in $\mathcal{Z}_\epsilon$ which proves the assertion.
\end{proof}

\begin{lemma} \label{4:lem} Let $\mathbf{Z}^\epsilon$ be the solution of the linear system (\ref{1DStSt:sys}). For small enough $\epsilon$ and $\Delta x$, the inequality $\mathbf{Z}^\epsilon<\mathbf{1}$ holds.
\end{lemma}
\begin{proof} 
By straightforward calculation we have
\begin{eqnarray*}
Z_{\epsilon,2}&:=&  Z_{\epsilon,1} - \frac{\Delta x^2\epsilon^b}{\delta} Q\\
Z_{\epsilon,3} &:=& 2Z_{\epsilon,2}-Z_{\epsilon,1}  -\frac{\Delta x^2\epsilon^b}{\delta} Q = Z_{\epsilon,1} -3\frac{\Delta x^2\epsilon^b}{\delta} Q\\
%z_{\epsilon,4} &:=& 2z_{\epsilon,3}-z_{\epsilon,2}  -\frac{\Delta x^2\epsilon^b}{\delta} R =  z_{\epsilon,1} -6\frac{\Delta x^2\epsilon^b}{\delta} R \\
%z_{\epsilon,5} &:=& 2z_{\epsilon,4}-z_{\epsilon,3}  -\frac{\Delta x^2\epsilon^b}{\delta} R=  z_{\epsilon,1} -10\frac{\Delta x^2\epsilon^b}{\delta} R \\
%z_{\epsilon,6} &:=& 2z_{\epsilon,5}-z_{\epsilon,4}  -\frac{\Delta x^2\epsilon^b}{\delta} R= z_{\epsilon,1} -15\frac{\Delta x^2\epsilon^b}{\delta} R \\
& \vdots&\\
Z_{\epsilon,j+1} &:=& 2Z_{\epsilon,j}-Z_{\epsilon,j-1}  -\frac{\Delta x^2\epsilon^b}{\delta} Q =
Z_{\epsilon,1}-k_{j+1}\frac{\Delta x^2\epsilon^b}{\delta} Q\\
\end{eqnarray*}
for $j=1,...,M-1$, where $k_1=0$ and $k_{j+1}:=k_j+j$.  Therefore,  $Z_{\epsilon,1} > ...> Z_{\epsilon,M}$. 
Finally we have from the last equation of (\ref{1DStSt:sys})
%\begin{equation*}
%2(1-\epsilon) = 3z_{\epsilon,M}-z_{\epsilon,M-1}-\frac{\Delta x^2\epsilon^b}{\delta} R = 2z_%{\epsilon,1}  +\left(-3k_M+k_{M-1}-1\right) \frac{\Delta x^2\epsilon^b}{\delta} R
%\end{equation*}
%or
\begin{equation*}
Z_{\epsilon,1} = (1-\epsilon) +\left(3k_M-k_{M-1}+1\right) \frac{\Delta x^2\epsilon^b}{2\delta} Q. 
\end{equation*}
For small enough $\epsilon$ and $\Delta x$, the second term on the right hand side is dominated by the first term, i.e. $1> Z_{\epsilon,1}> ... > Z_{\epsilon,M}$.
\end{proof}

\begin{theorem} Let $\mathbf{U}^\epsilon(t)$ be a solution of (\ref{31}) with $\mathbf{U}^\epsilon(0)<\rho \mathbf{1}$ for some $\rho$ with $0<\rho<1$. Then for sufficiently small $\epsilon$ and $\Delta x$ it holds $\mathbf{U}^\epsilon(t)<\mathbf{1}$.
\end{theorem}
\begin{proof} Proposition \ref{pro1} and Lemmas \ref{1:lem}-\ref{4:lem} imply directly that for small enough $\epsilon$ and $\Delta x$ the following inequalities hold
\begin{equation*}
0\leq \mathbf{U}^\epsilon(t) < \mathbf{V}^\epsilon(t)<\hat{\mathbf{Z}}^\epsilon <\mathbf{1},
\end{equation*}
where $\hat{\mathbf{Z}}^\epsilon=(\hat Z_{\epsilon,1},...,\hat Z_{\epsilon,NM})^T$ with $\hat Z_{\epsilon,\pi(i,j)}:=Z_{\epsilon,j}$ for $1\leq i\leq N$, $1\leq j\leq M$.
\end{proof}

\begin{theorem} Let $\mathbf{U}(0)\leq \rho\mathbf{1}$ for some $0<\rho<1$,  and let $T>0$. For $\epsilon>0$, let $(\mathbf{U}^{\epsilon}(t), \mathbf{C}^\epsilon(t))$ be the solution of (\ref{31}) with $\mathbf{U}^{\epsilon}(0)=\mathbf{U}(0), \mathbf{C}^{\epsilon}(0)=\mathbf{C}(0)$. Then $(\mathbf{U}^{\epsilon}(t),\mathbf{C}^\epsilon(t))$, $0\leq t\leq T$ converges to the solution of (\ref{220}) as $\epsilon \rightarrow 0$.
\end{theorem}
\begin{proof} 
%For $k=1,2,....$ define $\epsilon_k=1/k$ and consider the sequence of functions
%$\mathbf{U}^{\epsilon_k}(t), 0<t<T$.  Since $\mathbf{U}^{\epsilon_k}(t)$ is bounded and differentiable, the assertion follows from the Arzela-Ascoli theorem for vector valued real functions.
We have shown that $\textbf{U}^{\epsilon}(t)$ and $\textbf{C}^{\epsilon}(t)$ are bounded for small values of $\epsilon$ and $\Delta x$. Boundedness of $\textbf{U}^{\epsilon}(t)$ and $\textbf{C}^{\epsilon}(t)$ implies that there is a sequence $\epsilon_n$ that goes to zero when $n \rightarrow \infty$ such that $\textbf{U}^{\epsilon_n}(t)$ and $\textbf{C}^{\epsilon_n}(t)$ converge to some $\textbf{U}(t)$ and $\textbf{C}(t)$ component wise. We can now pass to the limit $\epsilon \rightarrow 0$ in the regular semi-discrete system (\ref{31}) and claim that $(\textbf{U},\textbf{C})$ are the desired solutions of degenerate semi-discrete problem (\ref{220}). In order to prove this hypothesis for $\textbf{U}$ it is sufficient to verify that $\mathfrak{D}_{\epsilon}(\textbf{U}^{\epsilon})\textbf{U}^{\epsilon} \rightarrow \mathfrak{D}(\textbf{U})\textbf{U}$ if $\epsilon \rightarrow 0$. We have
\begin{equation}\label{pro}
\parallel\mathfrak{D}_{\epsilon}(\textbf{U}^{\epsilon})\textbf{U}^{\epsilon}-\mathfrak{D}(\textbf{U})\textbf{U}\parallel_{\infty}\leq
\parallel\mathfrak{D}_{\epsilon}(\textbf{U}^{\epsilon})(\textbf{U}^{\epsilon}-\textbf{U})\parallel_{\infty}
+\parallel(\mathfrak{D}_{\epsilon}(\textbf{U})-\mathfrak{D}(\textbf{U}))\textbf{U}\parallel_{\infty}
\end{equation}
Boundedness of $\textbf{U}^{\epsilon}(t)$ and continuity of $D_{\epsilon}(u)$ imply that each term in the right hand side of inequality (\ref{pro}) tends to zero if $\epsilon \rightarrow 0$. This proves that $\textbf{U}(t)$ is indeed the solution of the biomass equation of  the degenerate semi-discrete system (\ref{220}).
In the same manner we can also show that $\textbf{C}(t)$ is the solution of the substrate equation of  (\ref{220}).

\end{proof}

\section{Numerical method}\label{sec4}
\subsection{Time integration}
%We have seen above that the solutions of the ordinary differential equation (\ref{220}) are bounded 
%by unity and separated from the singularity, 
%and that (\ref{220}) satisfies a Lipschitz condition. Therefore, the solutions have sufficient regularity and no special numerical methods should be required. 

In order to obtain a numerical approximation of (\ref{21}) we will solve the regularized semi-discrete problem, (\ref{31}) to approximate (\ref{regPDE:eq}), and pass $\epsilon$ to 0.  The regularized system (\ref{31}) satisfies a Lipschitz condition and has bounded solutions. This suggests the application of standard ODE solvers. 

In typical biofilm applications, different time scales for the substrate equation and the biomass equation induce stiffness, which can be exacerbated by non-linear diffusion effects if the biomass approaches unity somewhere.
We use a time adaptive, error controlled, embedded  Rosenbrock-Wanner method (ROW), more specifically ROS3PL, a third order method with 4 stages \cite{R:2013}. Rosenbrock-Wanner methods require the solutions of linear systems. In our case the coefficient matrix of these systems are sparse and non-symmetric. We use the stabilised bi-conjugate gradient method \cite{V:1992} to solve them. More specifically we use a routine from the SPARSKIT library \cite{S:1994}, that was prepared for parallel execution using OpenMP  in \cite{ME:2010}.

\subsection{Spatial discretisation and gradient blow-up at the interface}

The ordinary differential equation (\ref{220}) was derived by applying a space discretisation to the underlying partial differential equation (\ref{21}). Since at the biofilm/water interface the solutions of (\ref{21}) are known to have blow-up of the gradient, the question arises to which extent our spatial discretisation introduces smearing around the interface. We investigate this, applying our method to a related test problem with known exact solution, that has the same phenomenon.
Our test problem is the Porous Medium Equation with linear source term \cite{PZ:2011},
\begin{equation}
\frac{\partial u}{\partial t}=\nabla(u^{m}\nabla u)+ku,\quad,m>1\ \label{41}
\end{equation}
which admits the self-similar Barenblatt solution \cite{PZ:2011}
\begin{equation}
u(t,x,y)=e^{kt}\tau^{-(m+1)^{-1}}\left[\frac{m}{4(m+1)}\left(k_{0}^{2}-\frac{(x^2+y^2)}{\tau^{1/(m+1)}}\right)\right]_+^{\frac{1}{m}},\label{42}
\end{equation}
with $\tau=\frac{1}{km}e^{kmt}$ and $k_0^2=r_0^{2}(\frac{1}{km}e^{kmt_0})^{-(m+1)^{-1}}$, where $r_0$ is the radius of the initial spherical colony and $t_0$ is the initial time. Notation $u_+$ is defined as $u_+:=\max(u,0)$. This solution is induced by a Dirac delta function as initial data. For our numerical test we chose as initial data function (\ref{42}) evaluated at time  $t_0=0.1$ and $r_0=0.1$. The values of $m$ and $k$ used in our simulations are chosen as $m=4,~k=3$.

For our test simulations, we pose the homogeneous Neumann boundary condition on the boundaries and initially one semi-spherical colony is placed in the center of the domain. Note that (\ref{42}) and the solution of the biofilm model satisfy this boundary condition as well as the homogeneous Dirichlet condition up until the moving interface reaches the boundary somewhere. We terminated the simulations before this occurs.

After regularisation, (\ref{41}) becomes
\begin{equation}
\frac{\partial u_{\epsilon}}{\partial t}=\nabla\big((u+\epsilon)^m\nabla u_{\epsilon}\big)+ku_{\epsilon} \label{41eps}
\end{equation}
in analogy to the regularization in (\ref{25}).

\begin{table} 
\centering
\caption{ Least square error between $\mathbf{U}^\epsilon$ and the the non-regularised numerical solution with $\epsilon=0$, $E_0^{\epsilon}:=\frac{1}{256\times 256}\parallel \textbf{U}^{\epsilon}-\textbf{U}^{0} \parallel_{2}$, and the error between regularised and Barenblatt solutions, $E^{\epsilon}:=\frac{1}{256\times 256}\parallel \mathbf{U}-\mathbf{U}^{\epsilon}\parallel_{2}$, at $t=1$. The cell resolution is $N=M=256$ and the tolerance of the ROW method is set at $TOL=1e-7$.}\label{table2eps}
 \begin{tabular}{llllll}
  %\hline
 \hline
  $\epsilon$ &\qquad\qquad$E_0^{\epsilon}$&\qquad\qquad$E^{\epsilon}$\\
\hline\\
  $10^{-3}$&~~\quad$0.5456777671\times10^{-5}$&~~\quad$0.2670409320\times10^{-6}$\\
  $10^{-4}$&~~\quad$0.5458369969\times10^{-6}$&~~\quad$0.2332513375\times10^{-6}$\\
  $10^{-5}$&~~\quad$0.5483556720\times10^{-7}$&~~\quad$0.2302420952\times10^{-6}$\\
  $10^{-6}$&~~\quad$0.5489546560\times10^{-8}$&~~\quad$0.2299452153\times10^{-6}$\\
  $10^{-7}$&~~\quad$0.5511348599\times10^{-9}$&~~\quad$0.2299155176\times10^{-6}$\\
  \hline
  \end{tabular}
\end{table}

We solve (\ref{41eps}) on a uniform grid  of size $256\times 256$ for several choices of $\epsilon$ and compare the results against the exact solution (\ref{42}) and the numerical solution obtained with the choice $\epsilon=0$, i.e. the solution of (\ref{41}) in Table \ref{table2eps}. Reported are the 2-norm differences between the regularised numerical solutions and the Barenblatt solution,  $E^{\epsilon}:=\frac{1}{256\times 256}\parallel \mathbf{U}-\mathbf{U}^{\epsilon}\parallel_{2}$, and between $\mathbf{U}^\epsilon$ and  $\mathbf{U}^0=\mathbf{U}^\epsilon\mid_{\epsilon=0}$, i.e  $E_0^{\epsilon}:=\frac{1}{256\times 256}\parallel \mathbf{U}^\epsilon-\mathbf{U}^0\parallel_{2}$. We observe that with decreasing $\epsilon$ the error $E_0^\epsilon$  decreases to zero, i.e. convergence of the regularisation is confirmed. The error $E^\epsilon$ converges as well, but to a positive value. This suggests that error due to spatial discretisation is larger than the regularisation error.

\begin{table}
\centering
\caption{Error between the Barenblatt solution (\ref{42}) and the numerical solution with $\epsilon=0$  of PME (\ref{41eps}), $E_0^{N}=\frac{1}{N^2}\parallel \mathbf{U}-\mathbf{U}^{0} \parallel_{2}$, for different number of grid cells at $t=1$.}\label{table2}
\begin{tabular}{ll}
  %\hline
 \hline
  $N=M$ &\quad\quad\quad$E_0^{N}$\\
\hline\\
  $32$&~~\quad$1.0314434\times10^{-3}$\\
  $64$&~~\quad$1.9763426\times10^{-4}$\\
  $128$&~~\quad$8.7552260\times10^{-5}$\\
  $256$&~~\quad$2.2991221\times10^{-5}$\\
  $512$&~~\quad$6.8993723\times10^{-6}$\\
  $1024$&~~\quad$2.1955107\times10^{-6}$\\
  \hline
  \end{tabular}
\end{table}

\begin{figure*}
\centering
\begin{tabular}{ c c c}
\includegraphics[width=0.45\textwidth]{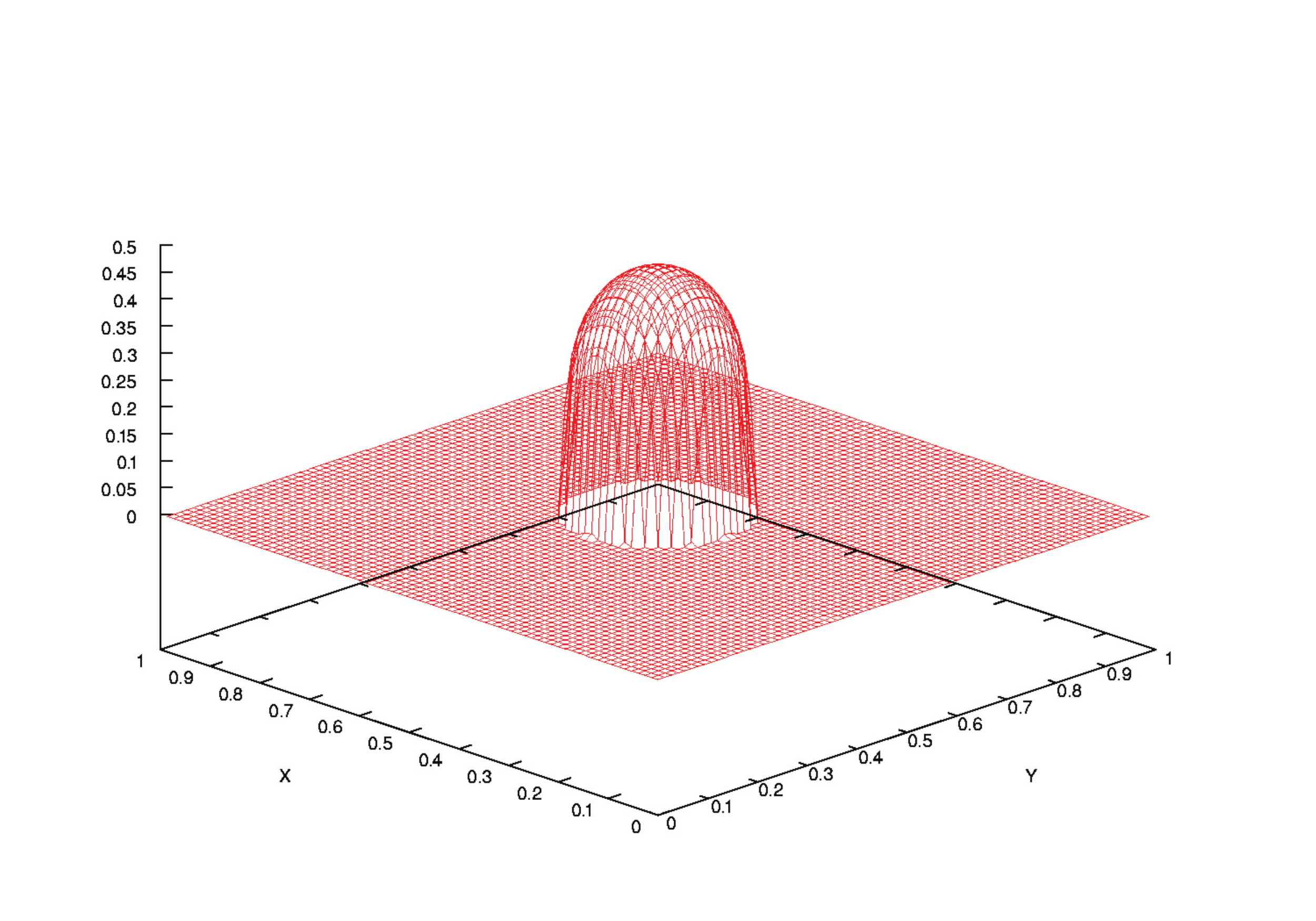} & \includegraphics[width=0.45\textwidth]{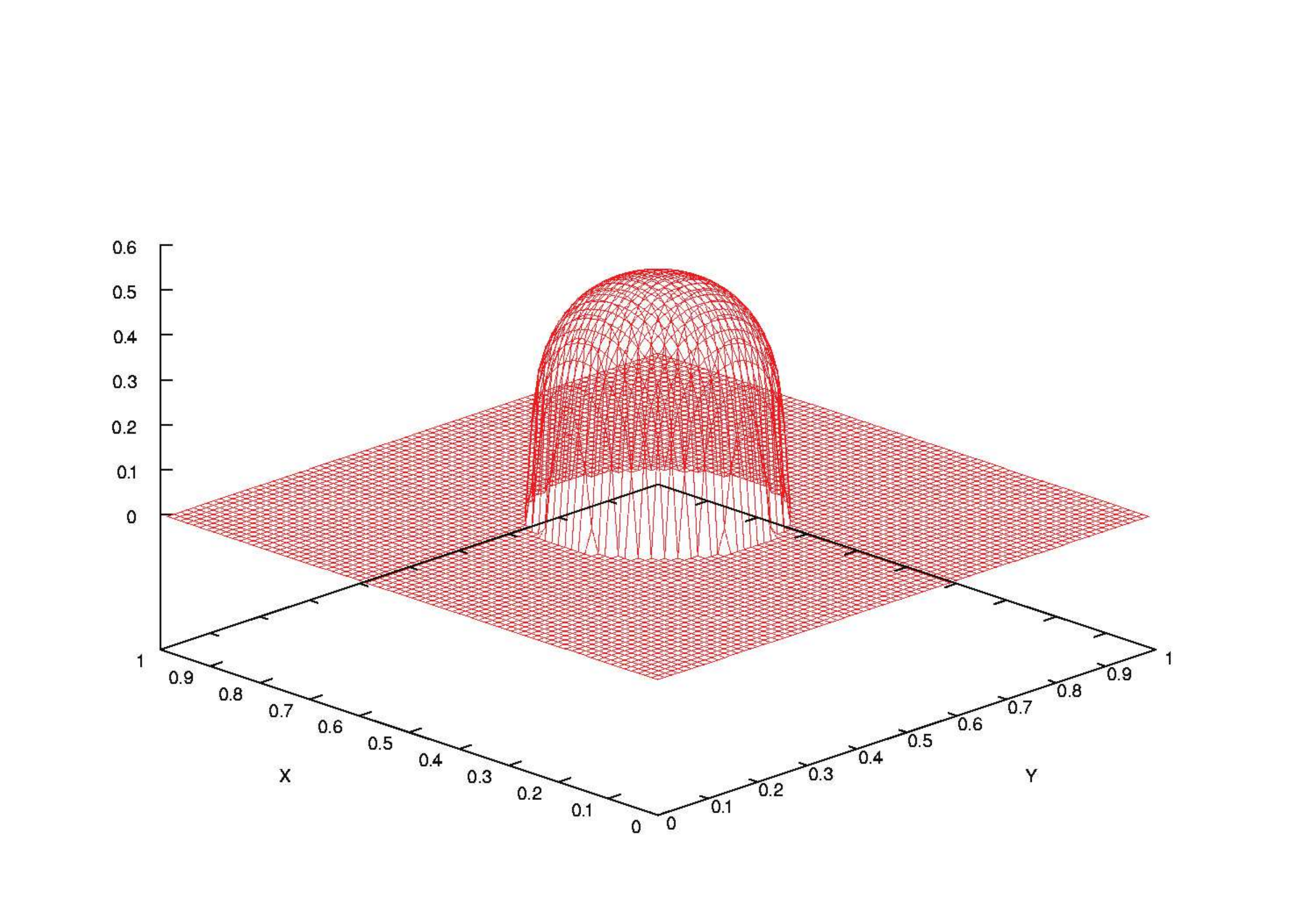}&\\
 \includegraphics[width=0.45\textwidth]{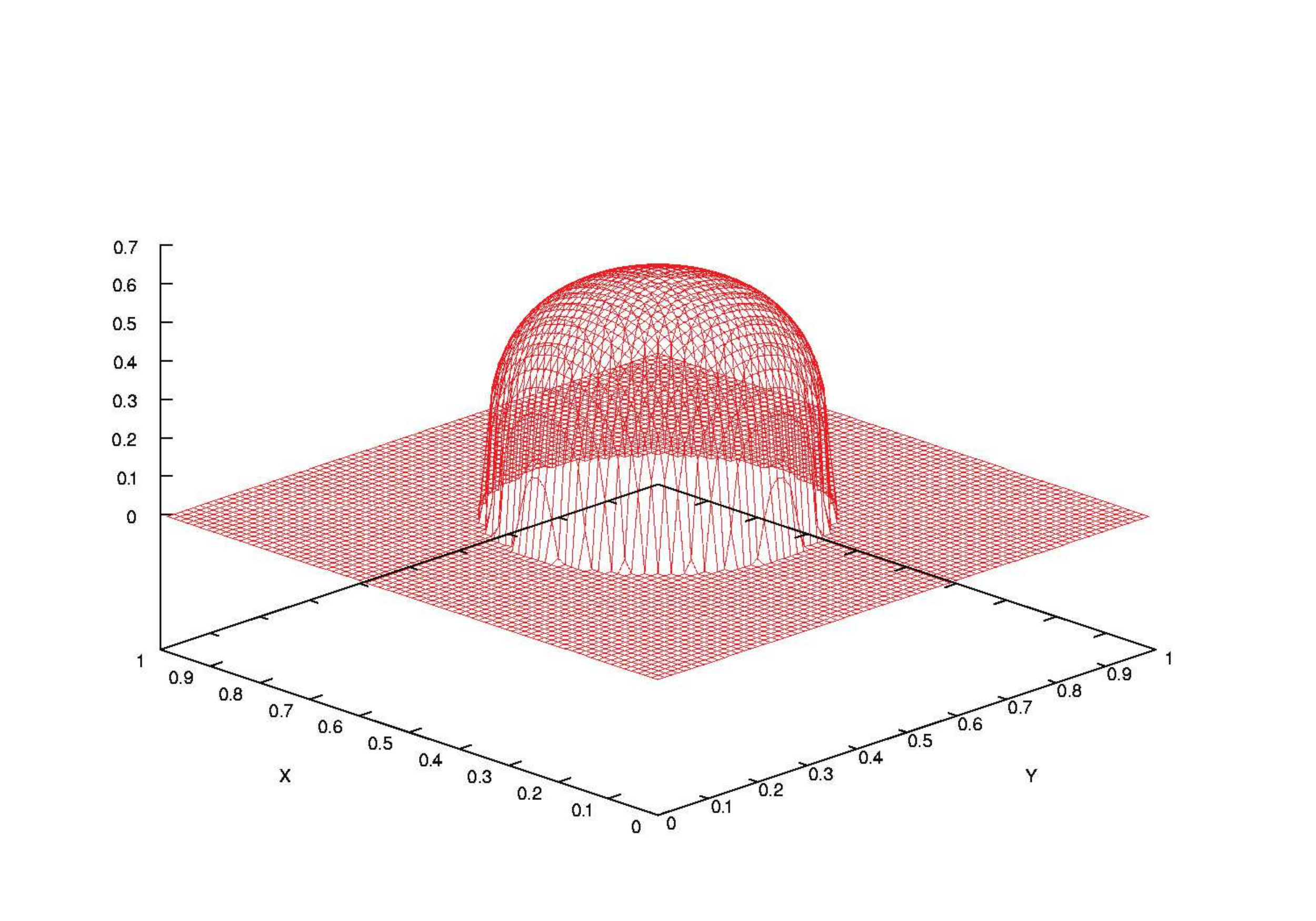} & \includegraphics[width=0.45\textwidth]{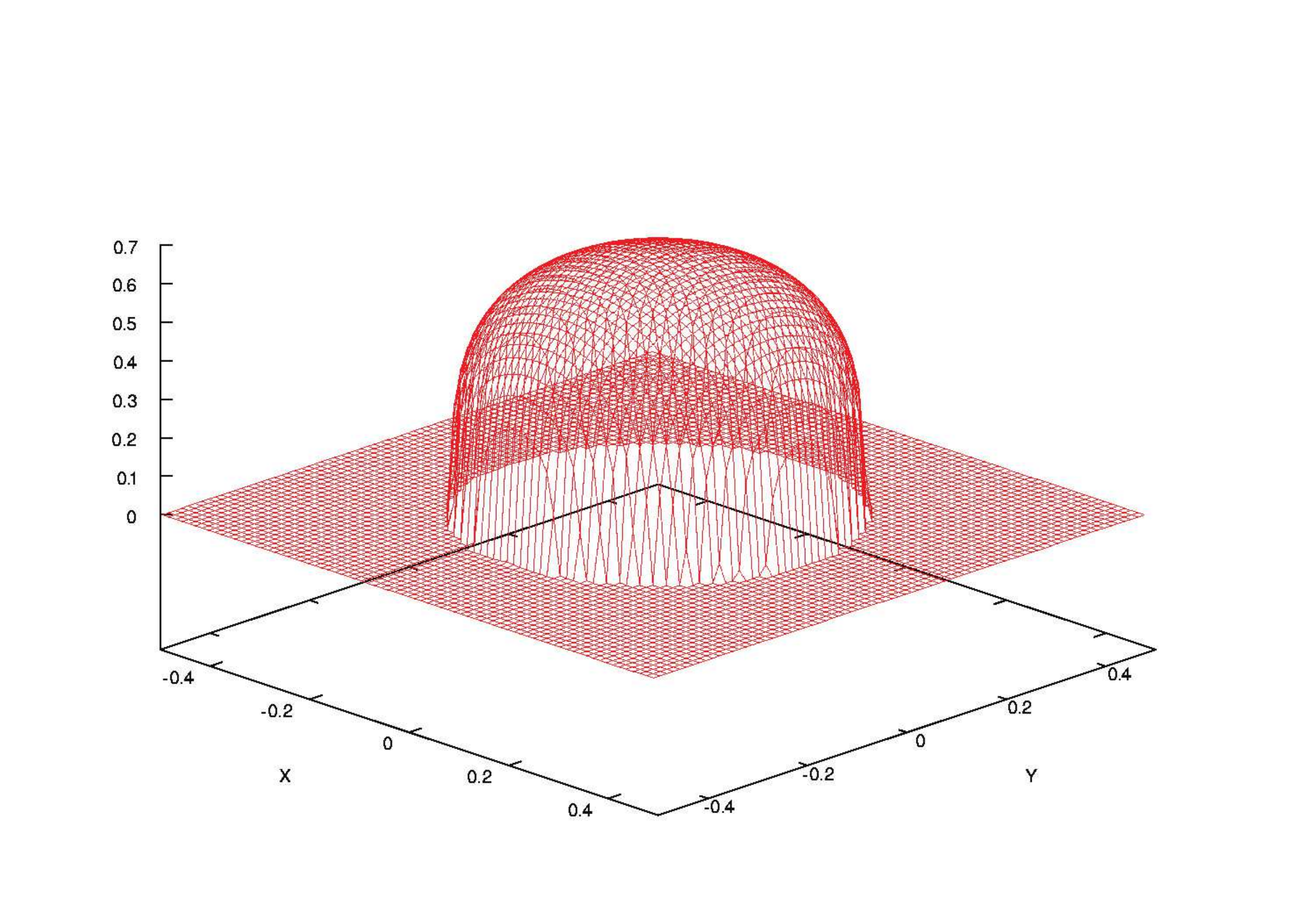}&\\
   \end{tabular}
    \caption{Snapshots of numerical solution of PME (\ref{41})  at $t=0.25,0.5,0.75,1$ from top left to the bottom right. The cell resolution is $N=M=256$ and the tolerance of the ROW method is $TOL=1e-7$ and the regularization parameter is $\epsilon=0$.}
   \label{fig1}
\end{figure*}

\begin{figure*}
%\centering
\begin{tabular}{ c c c}
  \includegraphics[width=0.46\textwidth]{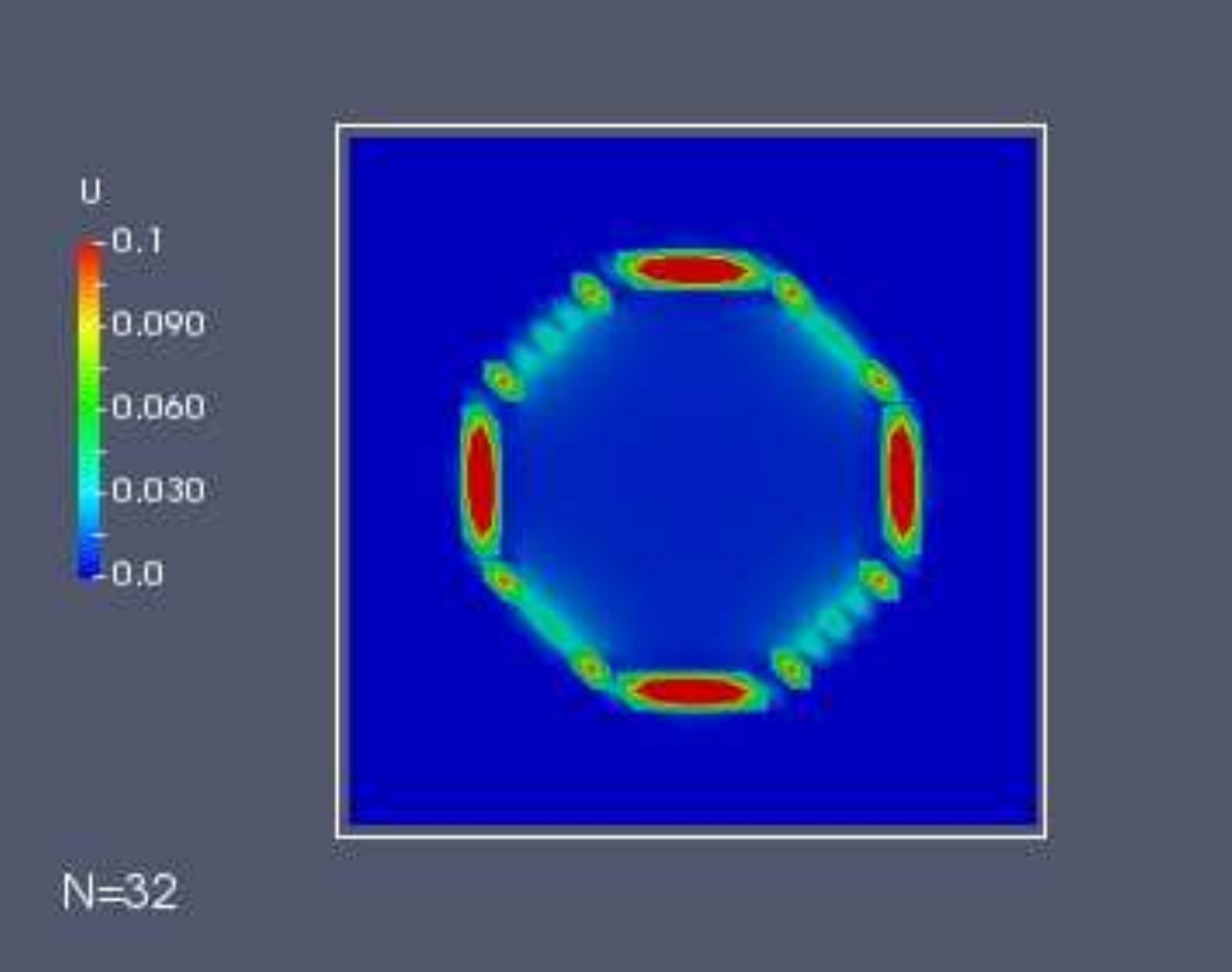} & \includegraphics[width=0.46\textwidth]{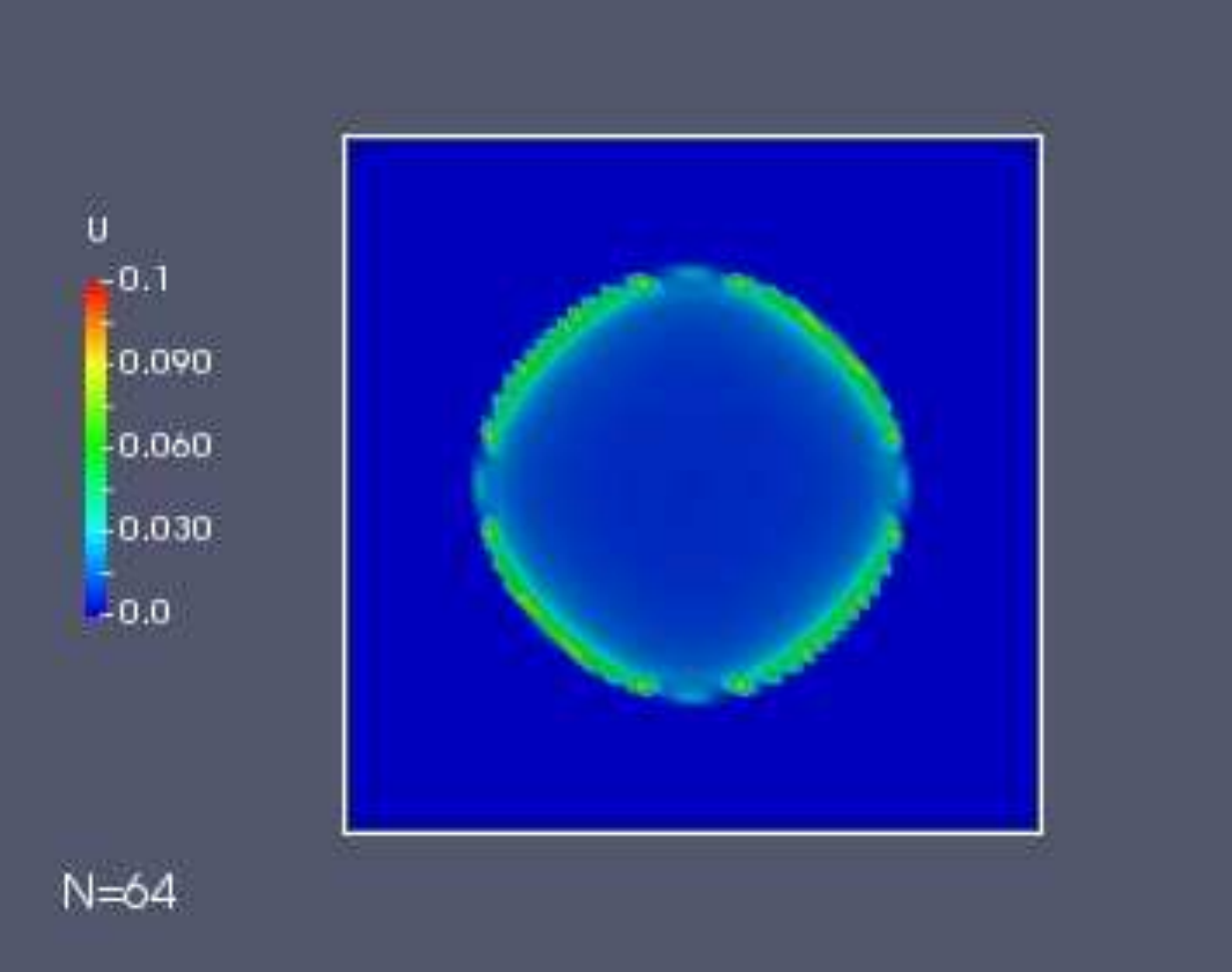}\\
  \includegraphics[width=0.46\textwidth]{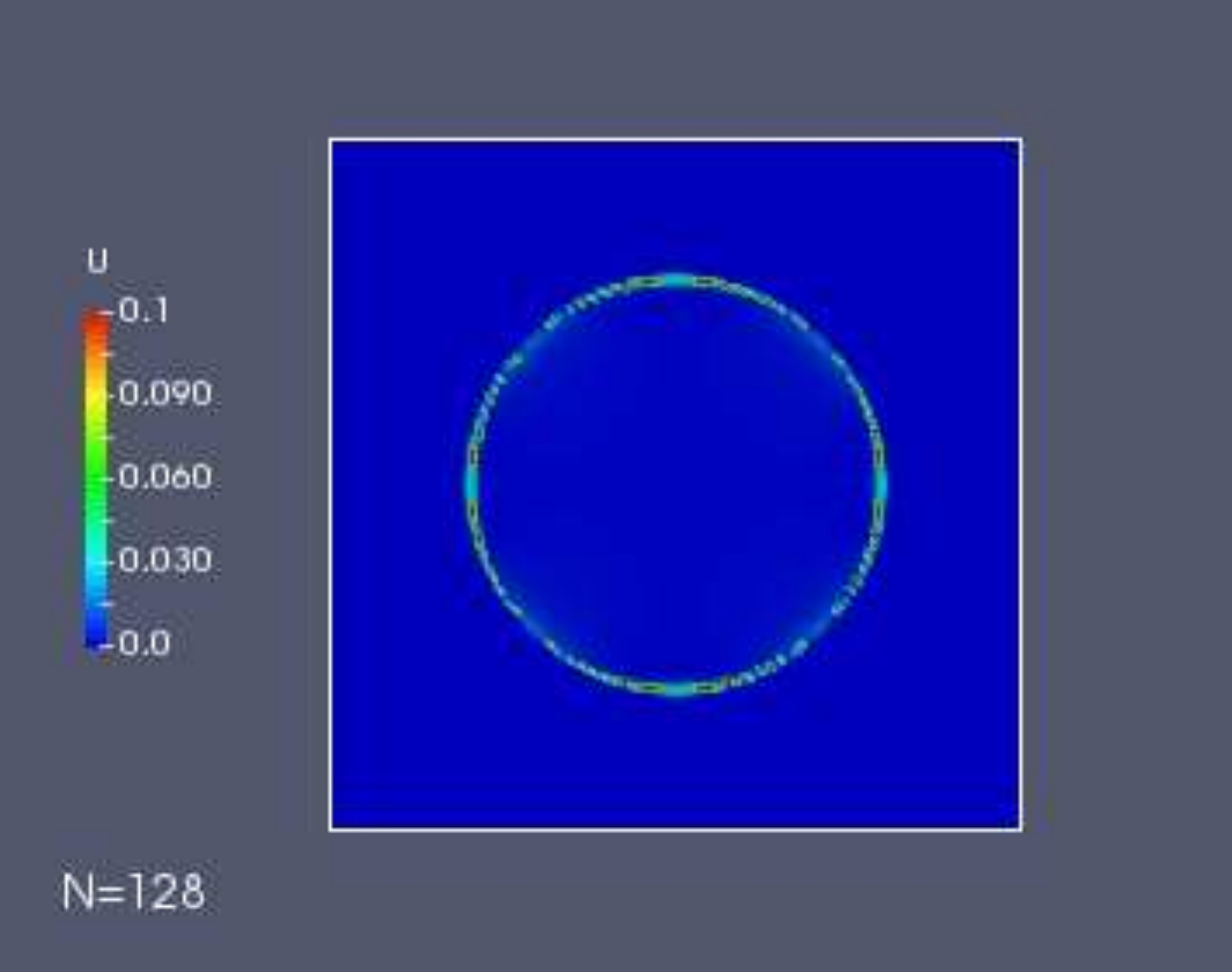} & \includegraphics[width=0.46\textwidth]{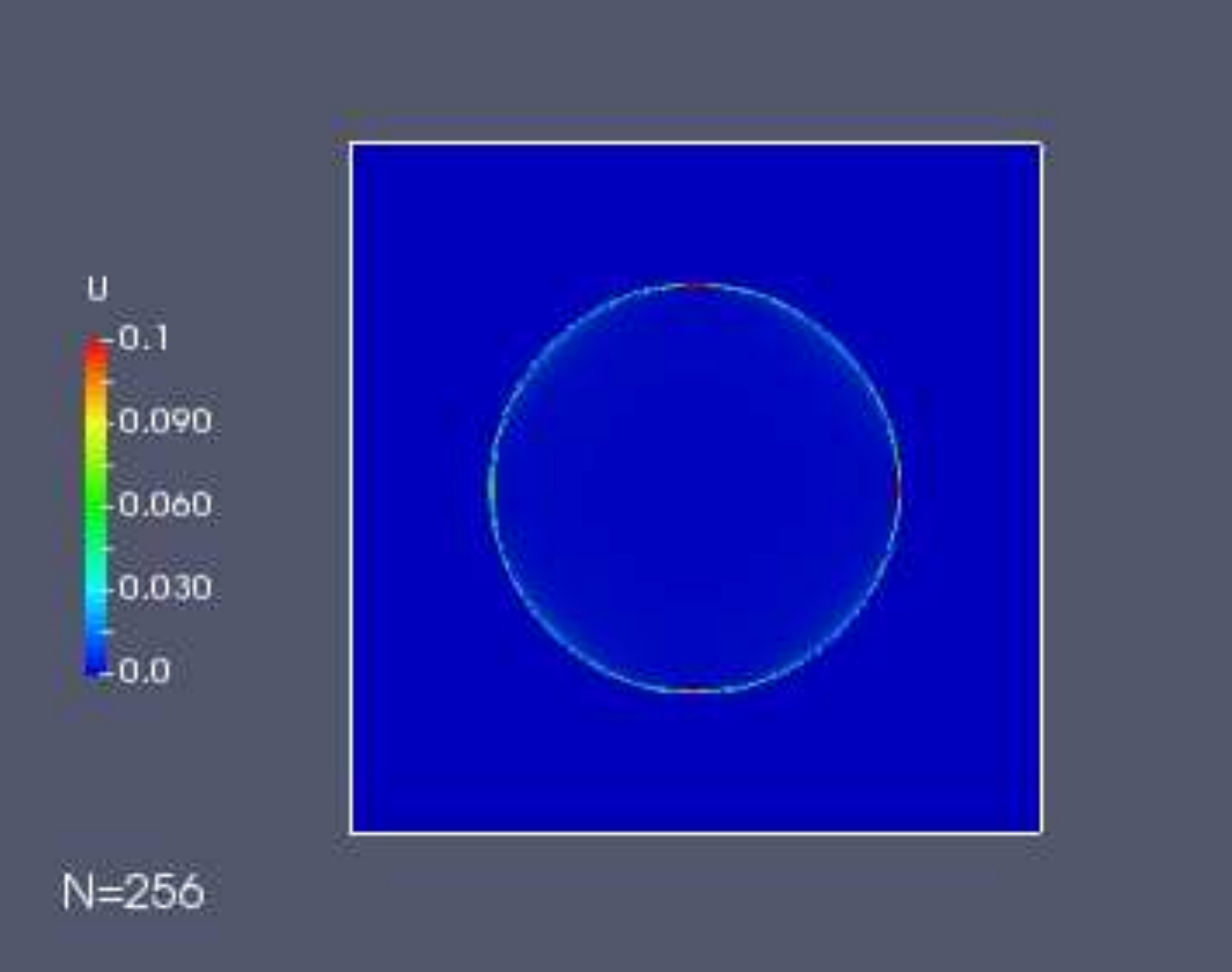}\\
  \includegraphics[width=0.47\textwidth]{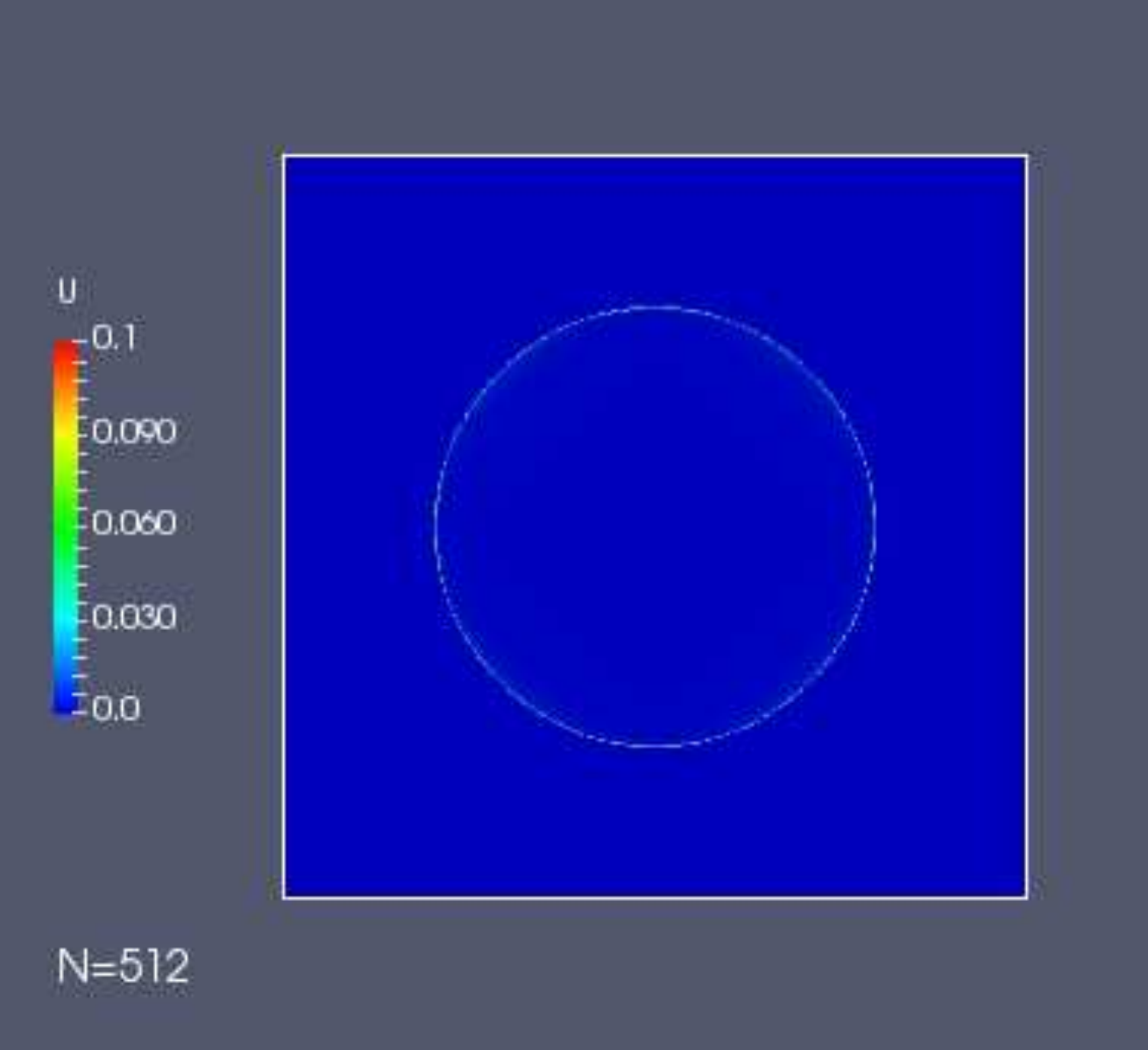} & \includegraphics[width=0.46\textwidth]{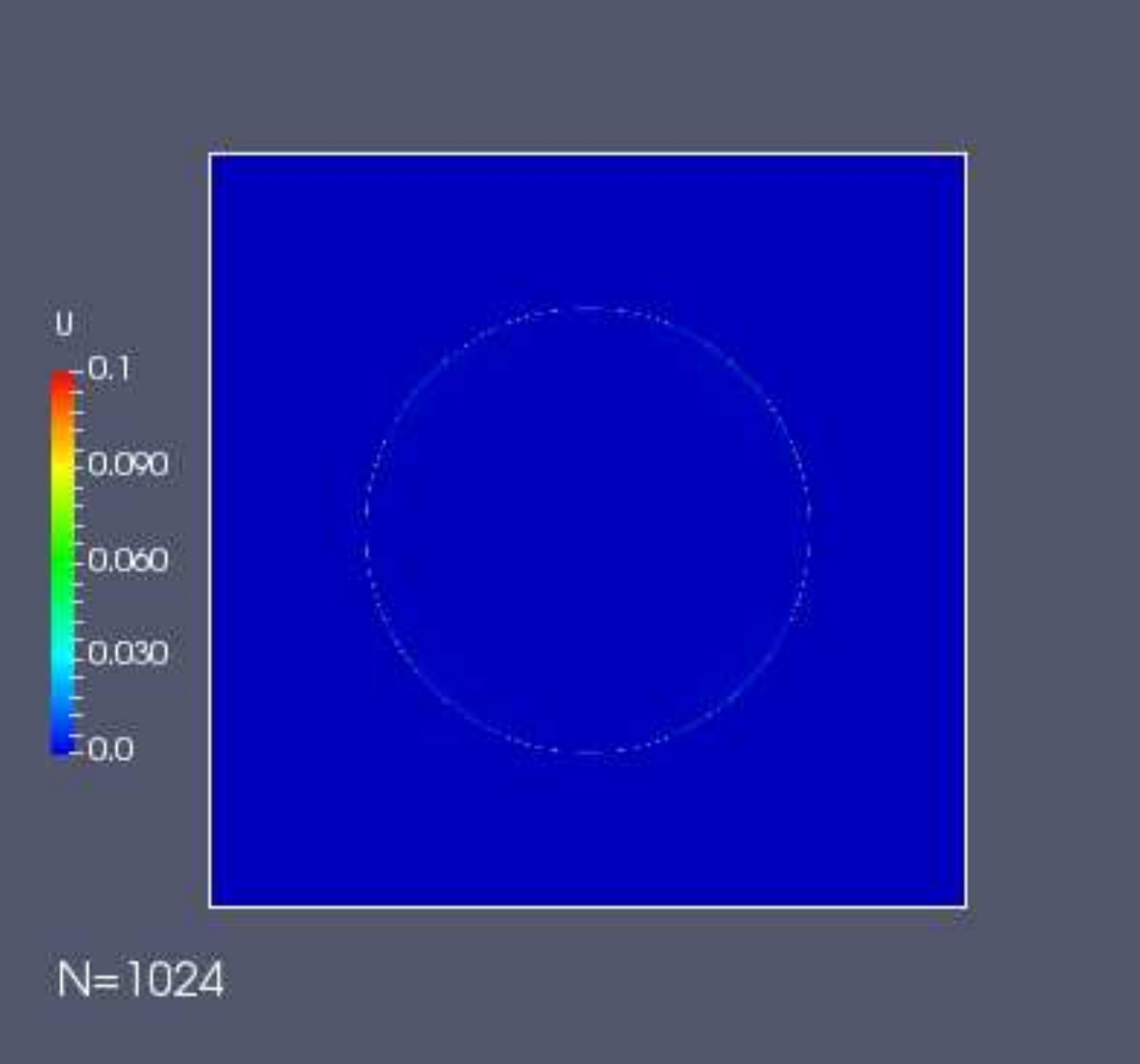}\\
   \end{tabular}
   \caption{Graphical result of difference between the numerical and the exact solutions of PME (\ref{41}) for different numbers of grid cells.}
   \label{fig3}
\end{figure*}

\begin{figure}
\centering
 \includegraphics[width=0.5\textwidth]{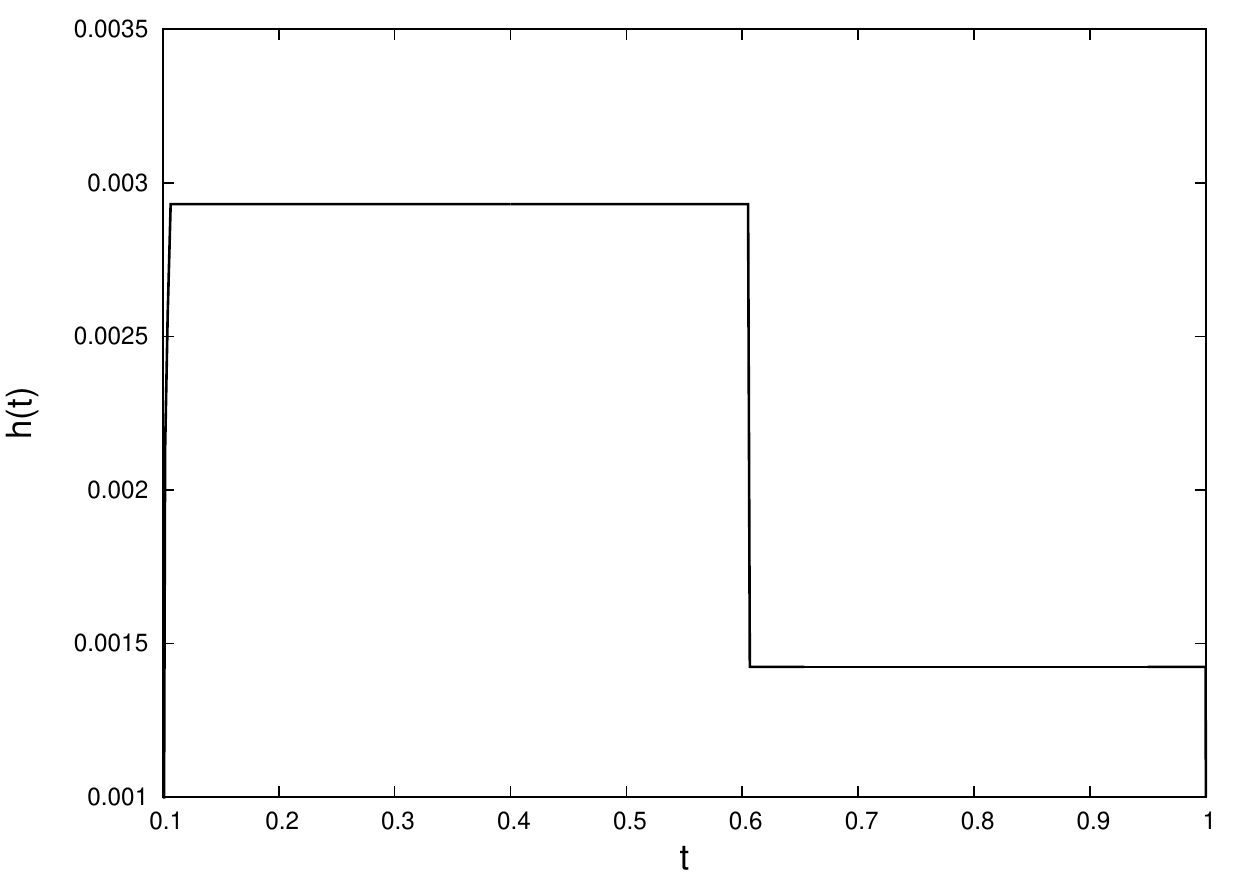}
\caption{Time-step size vs $t$ for PME (\ref{41}).} \label{fig2}
\end{figure}

Results of a computational convergence test that computes the least square error between the numerical and exact solutions of equation (\ref{41}), $E_0^{N}=\frac{1}{N^2}\parallel \mathbf{U}-\mathbf{U}^{0} \parallel_{2}$, for different numbers of grid cells are given in Table \ref{table2}. The error is steadily decreasing as the grid is refined, suggesting convergence of the method.

To illustrate spatio-temporal behaviour, the results of the numerical simulations with $\epsilon=0$ at different times are shown in Figure \ref{fig1}. The support of the solution expands radially, in agreement with the property of the porous medium equation, that initial data with compact support lead to solution with compact support (finite speed of interface propagation). The solution attains its maximum in the center of the domain and decreases toward the interface, where it indicates blow-up of the gradient, as expected from (\ref{41}) for $m>1$. Interface smearing effects that are introduced by the spatial discretisation are rather mild, and the numerical solution is free of spurious oscillations.
The numerical errors occur primarily at the interface, whereas agreement between the exact and the numerical solution is very good in the interior of the region where $U>0$ and the solutions are smooth, cf. Figure \ref{fig3}. But also here we see that the errors at the interface decrease as the grid resolution is increased.

Figure \ref{fig2} shows the time step chosen by the ROW method. Initially it remains constant at 0.003, at approximately $t=0.56$, it drops to 0.0015 where it remains for the remainder of the simulation, up to $t=1$.

These tests show that the spatial discretisation employed here produces results of acceptable accuracy for highly degenerate diffusion reaction equations.

%%%%%%%%%%%%%%%%%%%%%%%%%%%%%%%%%%%%%%%%%%%%%%%%%%%%%%
\section{Numerical examples and validation}\label{sec5}

In section \ref{illust:sec} we show an illustrative simulation of biofilm model (\ref{21}) with boundary conditions (\ref{23}) in growth limited and transport limited regimes with the numerical method discussed in section \ref{sec4}. Subsequently, in section \ref{valid:sec} we investigate the behaviour of the numerical solution with respect to grid step size $\Delta x$, regularization parameter $\epsilon$ and the tolerance required of the Rosenbrock-Wanner method.
The model parameters that we use in these simulations are given in Table \ref{tabledata}.

\begin{table}[h!]
\centering
\caption{Dimensionless model parameters of system (\ref{21}), taken from \cite{RE:2014}.} \label{tabledata}
\begin{tabular}{l*{2}{c}r}
  %\hline
  \hline
  Parameter&Symbol&Value\\
  \hline
  Decay rate &~$k$&~$0.67$\\
  Monod half saturation&~$K_U$&~$0.13$\\
  substrate uptake rate&~$\nu_U$&~$530$\\
  Nutrient diffusion coefficient&~$d_c$&~$33$\\
  biomass motility coefficient&~$\delta$&~$10^{-8}$\\
  \hline
\end{tabular}
\end{table}

\begin{figure*}[h!]
\centering
\begin{tabular}{ c c c } 
\includegraphics[width=0.47\textwidth]{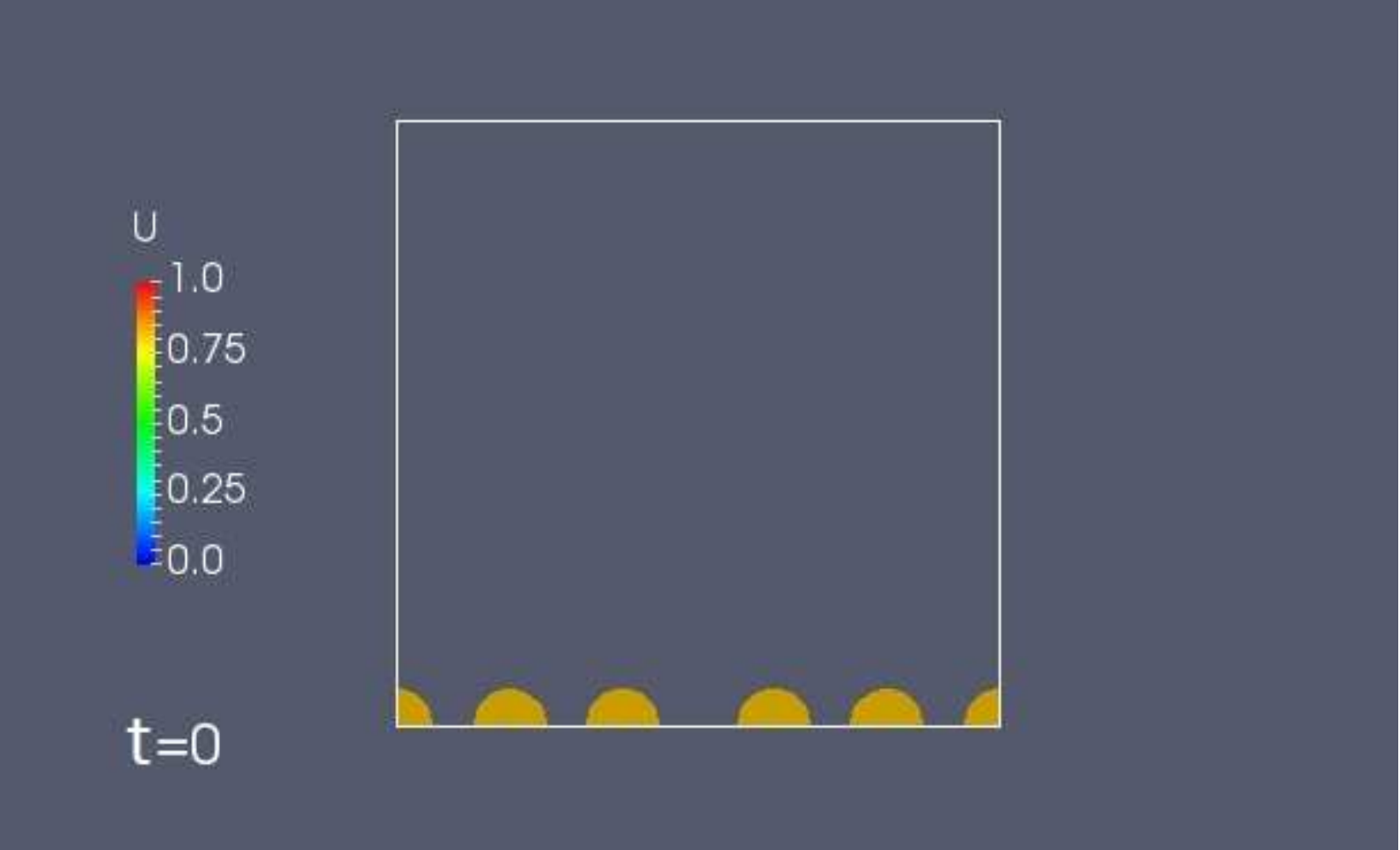}&\includegraphics[width=0.47\textwidth]{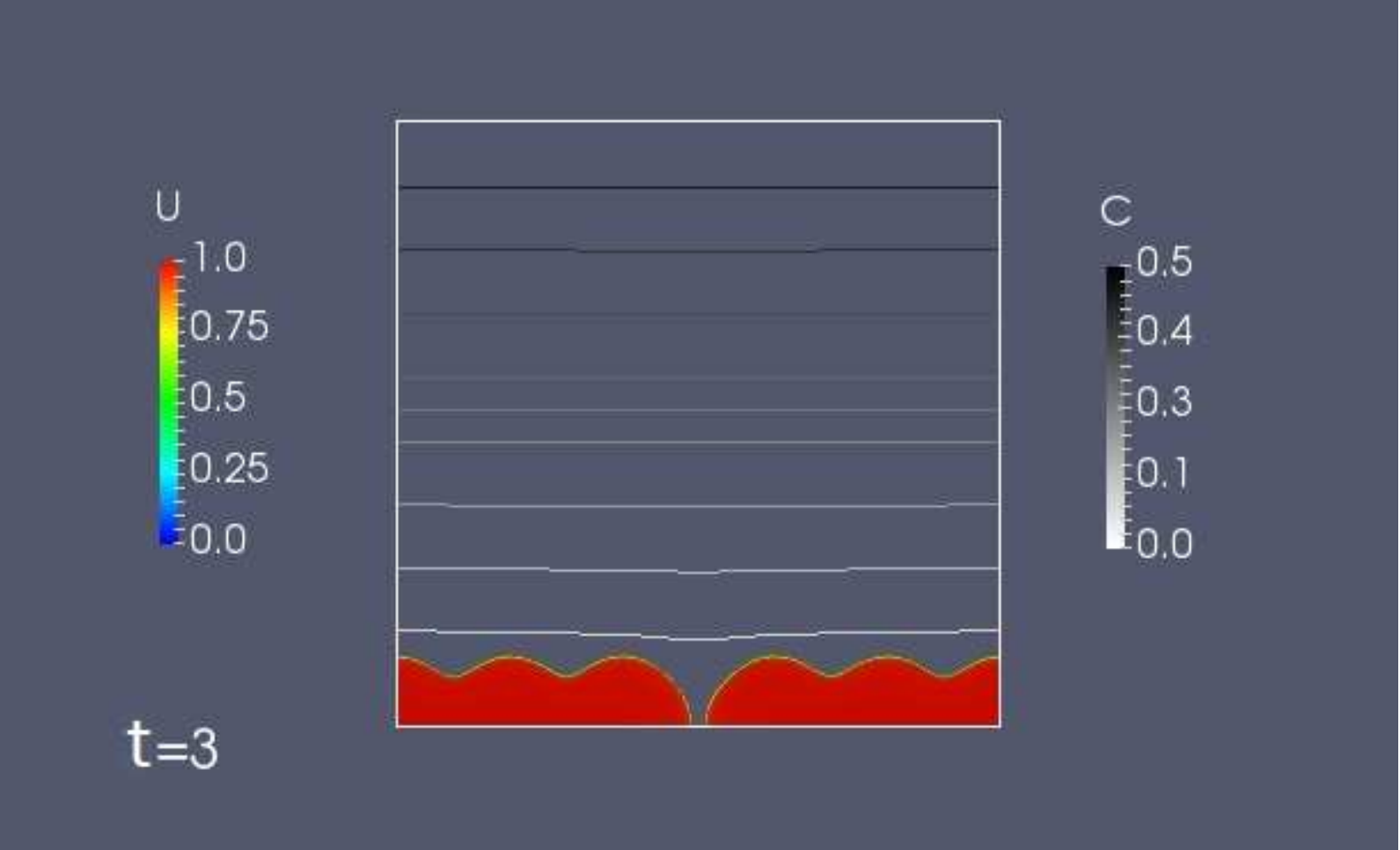}&\\
\includegraphics[width=0.47\textwidth]{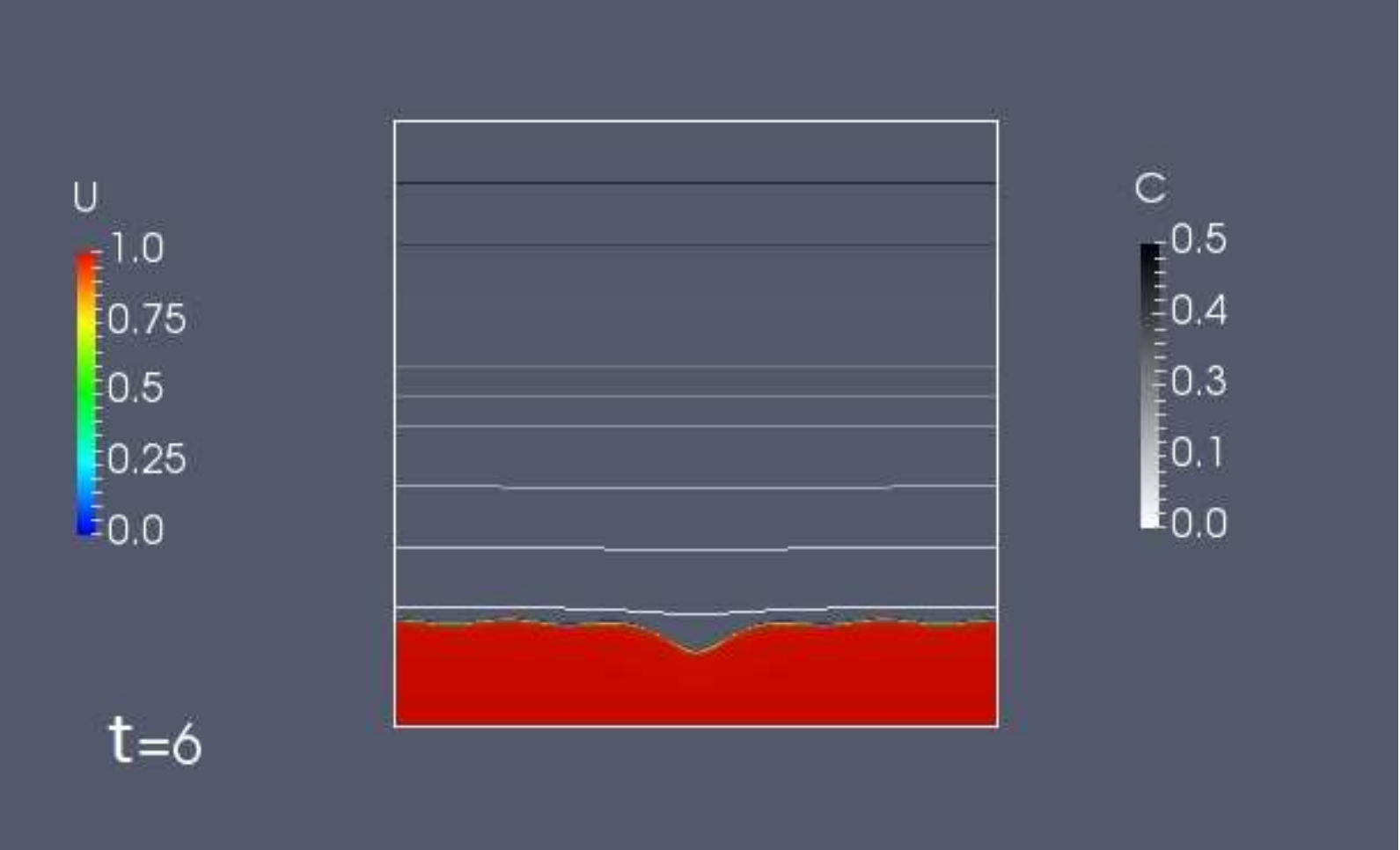}&\includegraphics[width=0.47\textwidth]{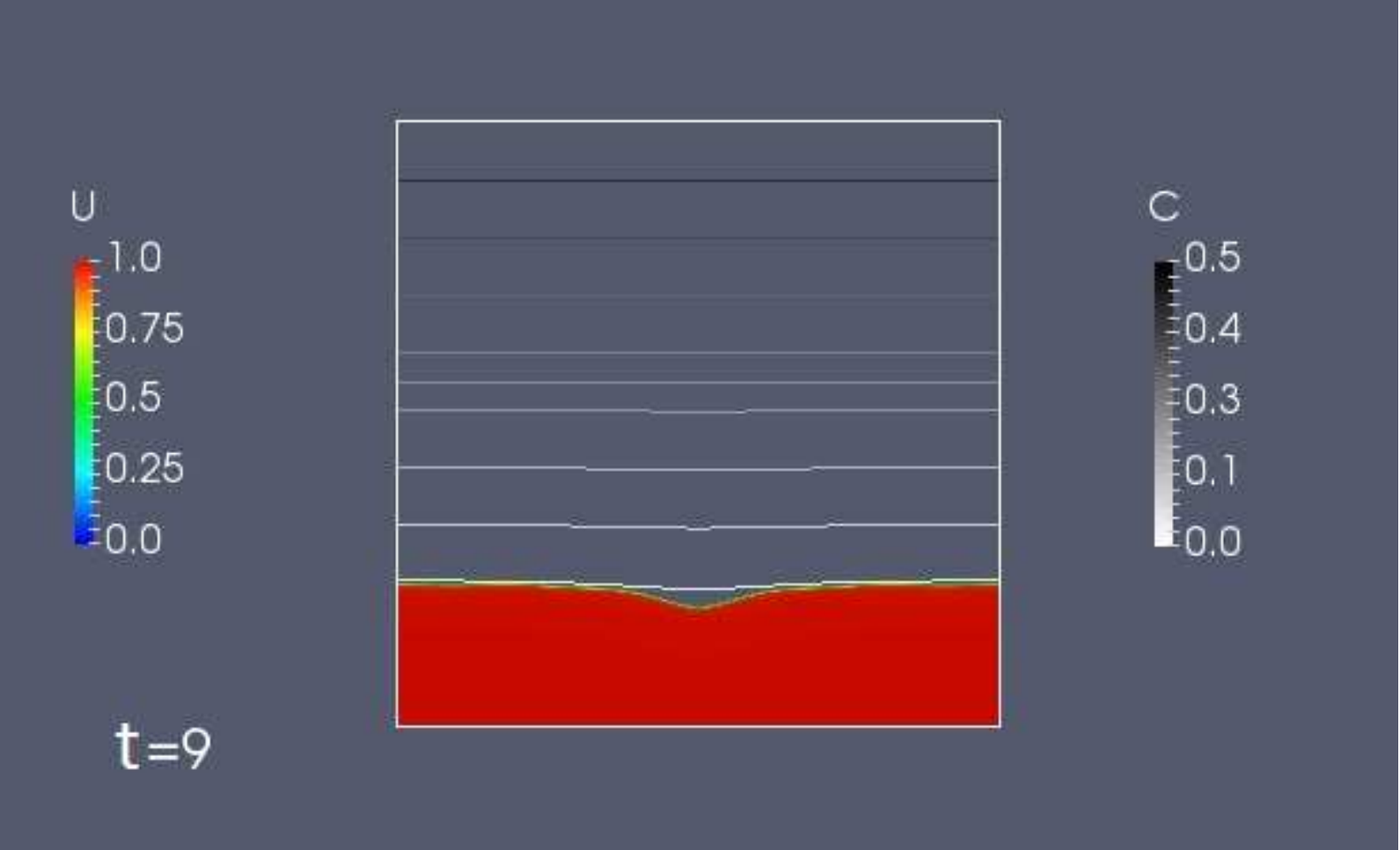}&\\
\end{tabular}
\centering 
\includegraphics[width=0.5\textwidth]{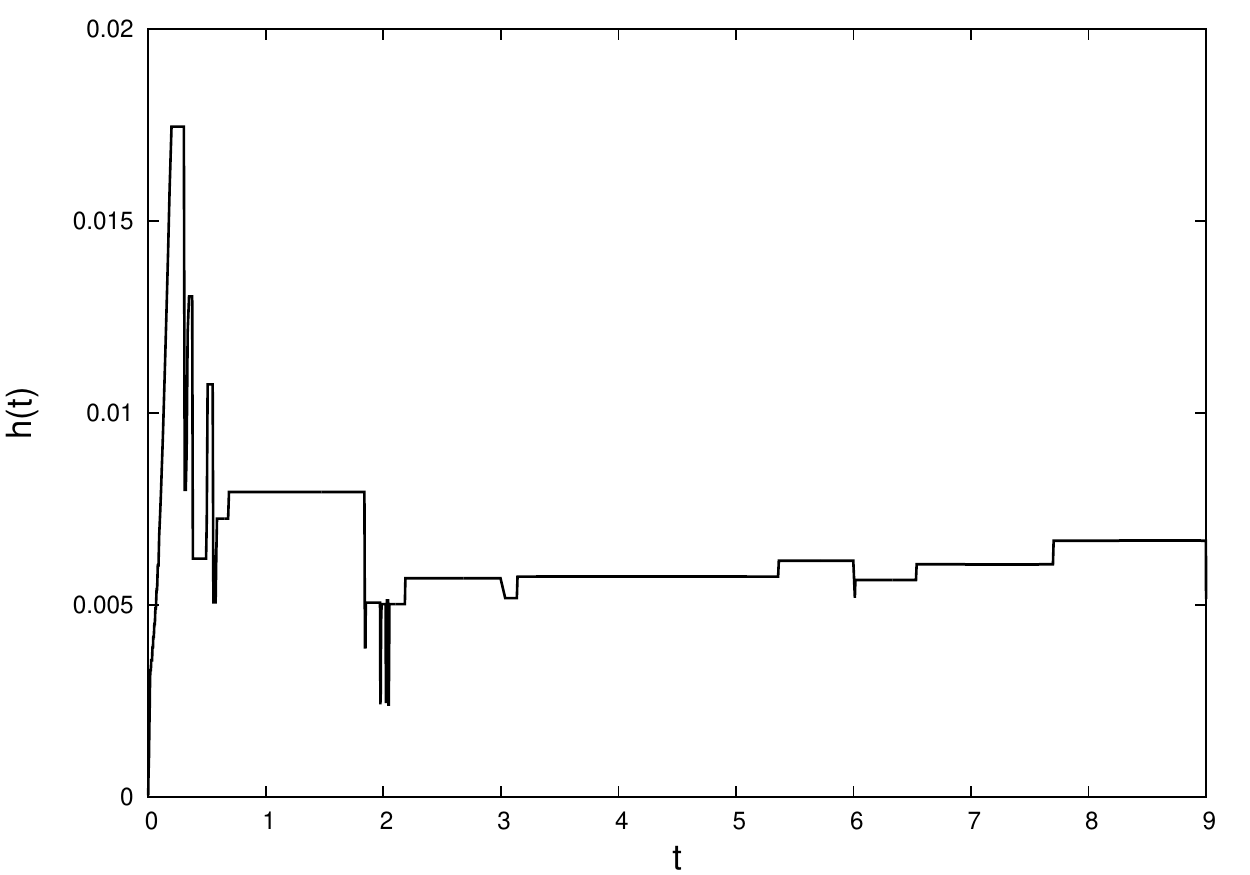}
\caption{\textbf{First two rows:} Snapshot of biofilm formation and nutrient distribution at $t= 0,3,6,9$
on a square domain with $N=M=256$ and $TOL=1e-7$. \textbf{Third row:} The value of time-step size for degenerate problem (\ref{220}) with $N=M=256$ and $TOL=1e-7$ }
\label{fig4}
\end{figure*}

\subsection{Illustrative Simulation}\label{illust:sec}

In a first illustration of the method we consider a case where initially the substratum is inoculated by six semi spherical colonies, which are symmetrically arranged around the center of the substratum on the bottom boundary. The spacing between the three colonies to the left and to the right is equidistant, but the spacing between the two inner most colonies is slightly larger. Substrate is  added from the top.
The domain $[0,1]\times[0,1]$ is square, uniformly discretised by a mesh with $N \times M=256 \times 256$ grid cells. The simulations shown here are carried out with $\epsilon=0$, i.e, the non-regularised, degenerate problem is solved.
For the time integration of the spatially discretised system (\ref{220}), we set the tolerance of the ROW method at $TOL=1e-7$.

Snapshots of the simulation are shown in Figure \ref{fig4}. Biofilm colonies are color-coded by biomass density, the grey-scale isolines are drawn to depict the substrate concentration $C$.

As the graphical results show, the numerical solution is positive and bounded by a value less than one as we have shown mathematically in section \ref{sec3}. Furthermore, the numerical simulation preserves the symmetry of the solution that was imposed by symmetric initial conditions.

Initially the individual colonies grow and expand but are still separated. At $t=3$ the colonies have merged into one larger colony left of center, and one right of center, but with a gap still in the center where the colonies initially were spaced further apart. The colonies continue to grow and stratify. At $t=6$ the gap is closed and the biofilm stratifies. At $t=12$ finally, one almost homogeneous layer of biofilm is formed. Throughout the simulation, the substrate concentration appears stratified, parallel to the substratum. While the substrate concentration decreases toward the substratum due to consumption in the biofilm, it does not become severely  limiting.

In Figure \ref{fig4} we also show how the time step $h(t)$ changes over the progress of the simulation. Initially it increases, indicating that the initial guess for the time step could have been chosen larger. After a short plateau phase it decreases in an oscillatory fashion and attains a plateau phase at time $t=3$ which is occasionally disrupted by intermittent drops in time-step from which the method recovers quickly. This shows not only the adaptability of the numerical method but also indicates that the adaptivity of the method allows a simulation with time steps that usually are much larger than the smallest occurring time step.

It was shown in \cite{PVH:1998} that the biofilm structure depends on the ratio of biomass growth and nutrient supply.  For low nutrient supply, biofilms form patchy structures while for high nutrient availability more compact structures are obtained. We have shown in Figure \ref{fig4} that for the values of parameters which we used for numerical simulations, biofilm colonies expand and form a homogeneous thick layer. Nutrient supply can be controled by the bulk concentration and by the diffusion parameters, namely the diffusion coefficient $d_c$ and the diffusion length $H$. In our non-dimensional formulation, decreasing the bulk concentration is equivalent to increasing the half saturation concentration $K_U$. In order to test our method for biofilms under nutrient limitations, we change $K_U$ from $0.13$ to $0.65$ and decrease $d_c$ from $33$ to $3.3$.

Snapshots of the simulation are shown in Figure \ref{figg}. Biofilm colonies are color-coded by biomass density. It is observed in Figure \ref{figg} that for lower nutrient availability, the biofilm has rough structure and bacterial colonies do not merge together even at $t=350$. They grow upwards, where nutrients are coming from, rather than spreading horizontally. The simulation in Figure \ref{figg} also shows the hollow inside the biofilm which emerges at $t=150$ as consequence of nutrient limitation.

\begin{figure*}[h!]
\centering
\begin{tabular}{ c c c } 
\includegraphics[width=0.45\textwidth]{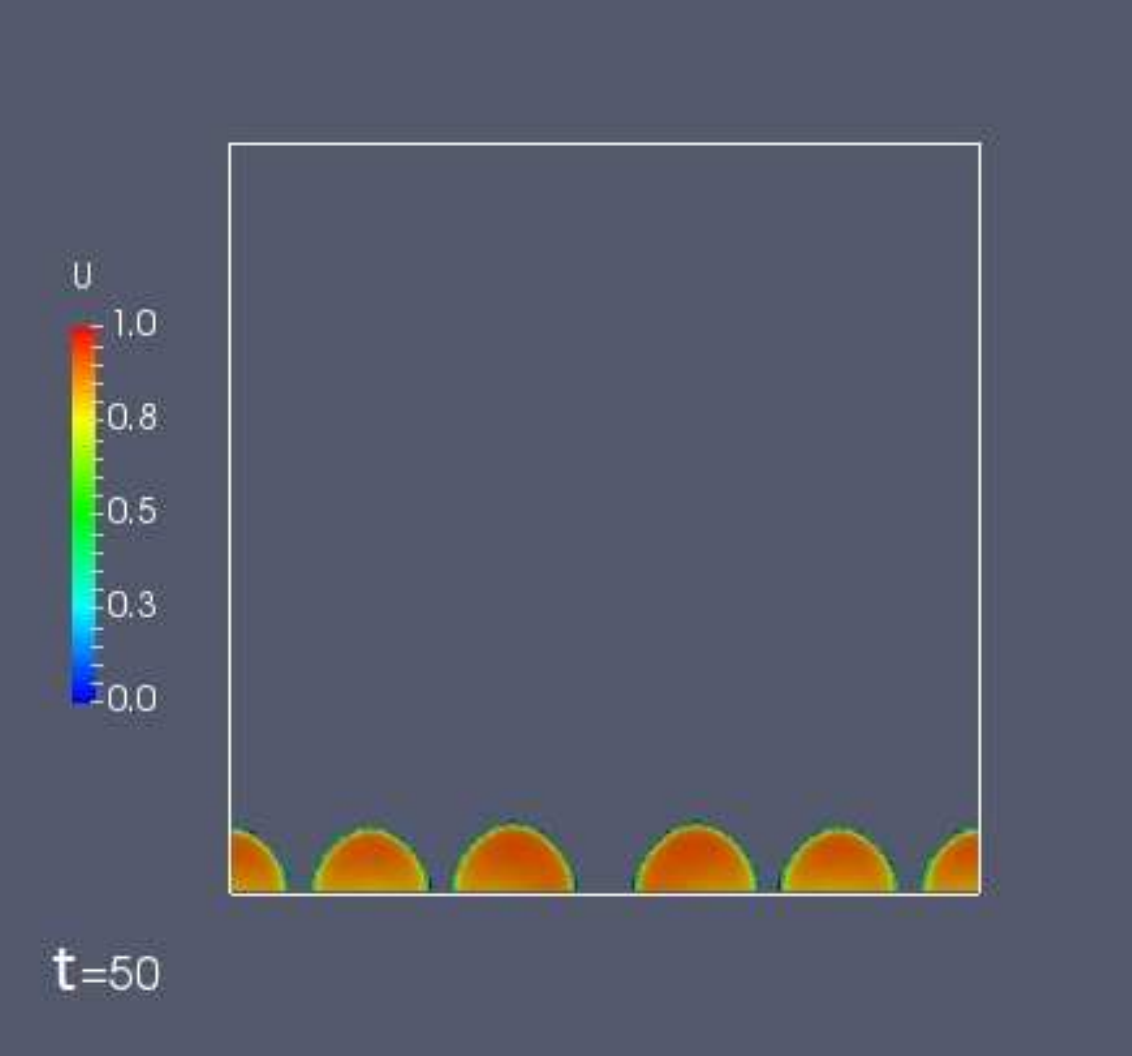}&\includegraphics[width=0.45\textwidth]{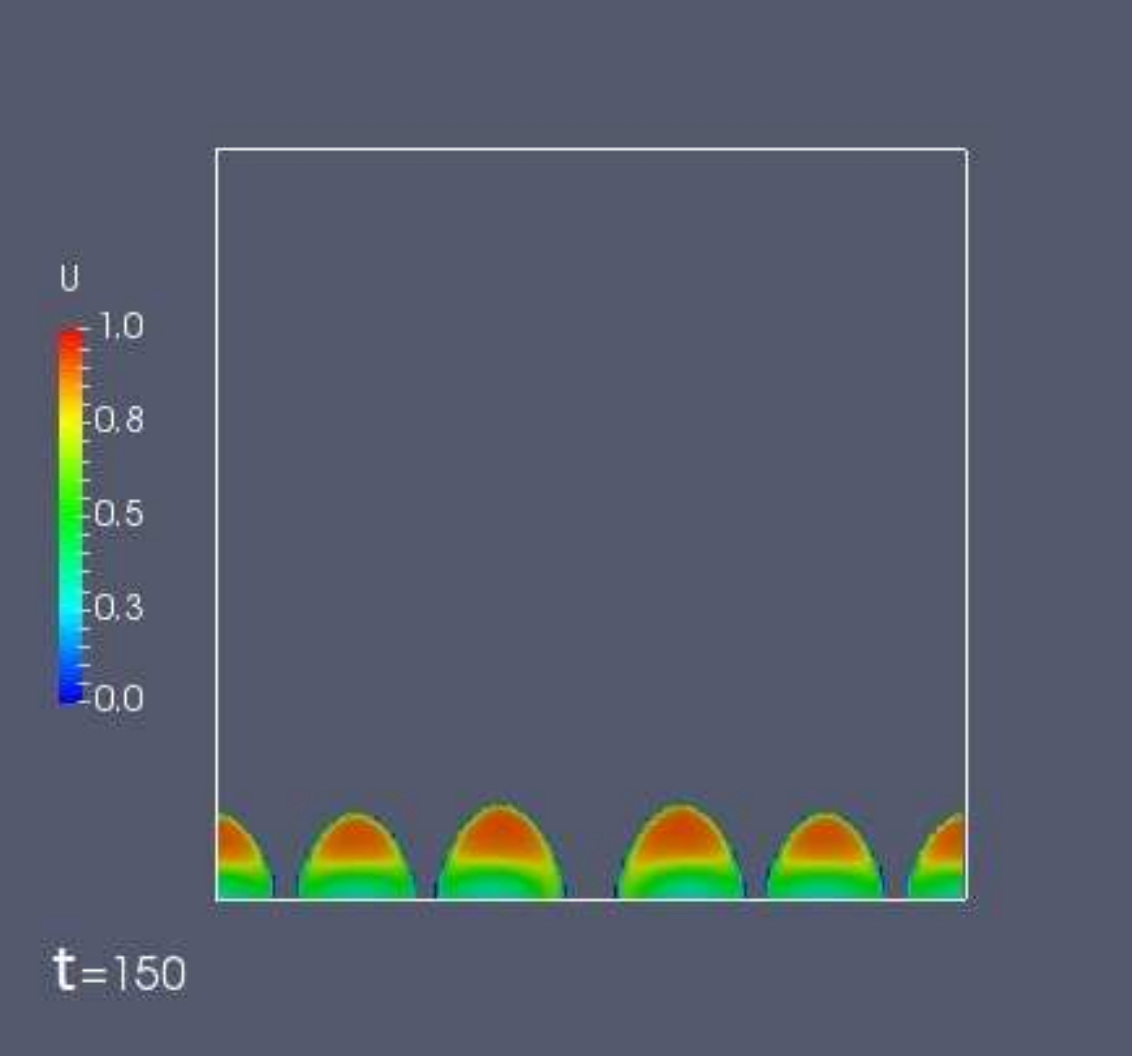}&\\
\includegraphics[width=0.45\textwidth]{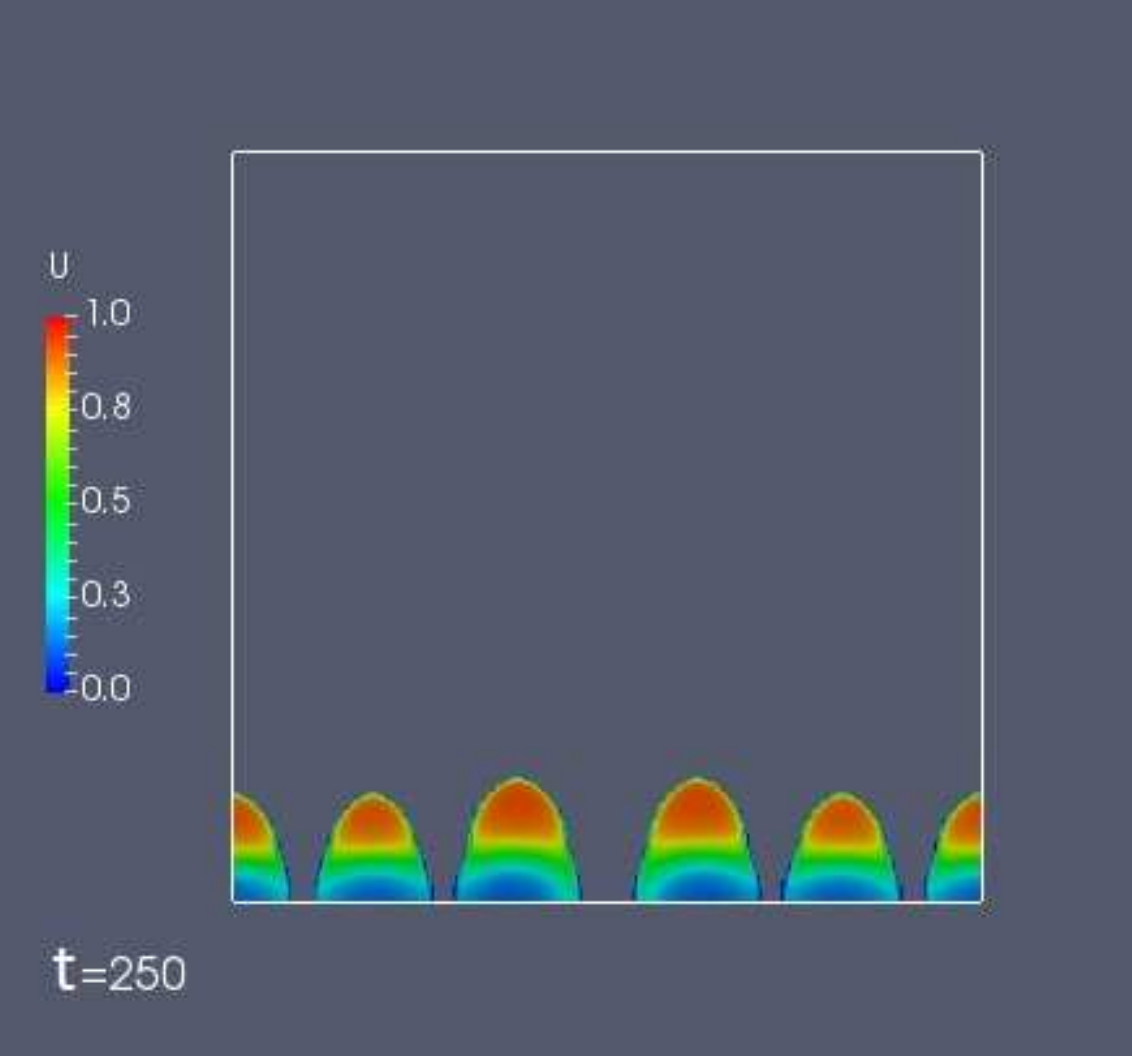}&\includegraphics[width=0.45\textwidth]{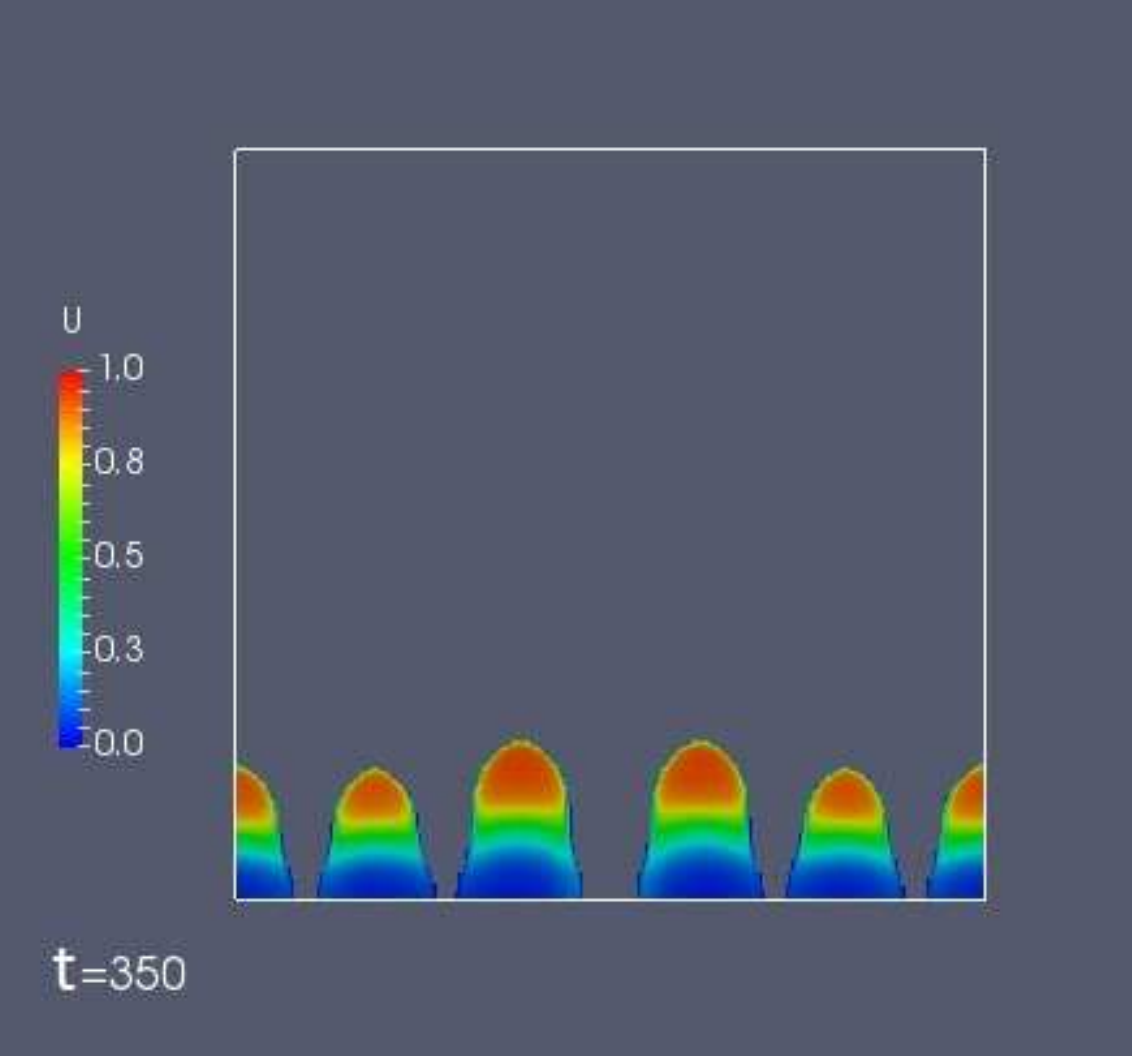}&\\
\end{tabular}
\caption{Snapshot of biofilm formation at $t=50,150,250,350$ with higher nutrient availability on a square domain with $N=M=256$ and $TOL=1e-7$.}
\label{figg}
\end{figure*}

\subsection{Validation of the Numerical Method}\label{valid:sec}

\subsubsection{Convergence of the regularization}\label{eps}

We have seen in Section \ref{sec4} for the porous medium equation that solving the regularised problem and passing $\epsilon$ to 0 gives a solution of the degenerate problem. We investigate this here for the full prototype biofilm growth model (\ref{21}) by solving (\ref{31}), on a uniform grid of size $256 \times 256$.
In order to determine the rate of convergence approximately we consider degenerate and regular semi-discrete problems (\ref{220}) and (\ref{31}) with a similar initial setting as in Section \ref{illust:sec}.

We define the error between solutions of degenerate and regular semi-discrete problems (\ref{220}) and (\ref{31}) as $E^{\epsilon}_{U}=\frac{1}{256\times 256}\parallel \mathbf{U}^0-\mathbf{U}^{\epsilon}\parallel_{2}$ and $E^{\epsilon}_{C}=\frac{1}{256\times 256}\parallel \mathbf{C}^0-\mathbf{C}^{\epsilon}\parallel_{2}$ and compute them at $t=6$. This is the time at which colonies are merged into a larger colony and there is not any gap between them. 

The values of $E^{\epsilon}_{U}$ and $E^{\epsilon}_{C}$ are shown in Figure \ref{fig7} for different values of $\epsilon$. The logarithmic scale is chosen for both axes to show the results more clear. 

As the graphical results in Figure \ref{fig7} show the error between solutions of regular and degenerate semi-discrete problems is decreasing by decreasing the value of $\epsilon$. For $\mathbf{U}$ the error is initially decreasing, at approximately $\epsilon=10^{-5}$ it becomes constant and does not show any considerable variation for $\epsilon<10^{-5}$ indicating the convergence of $\mathbf{U}^{\epsilon}$ to $\mathbf{U}^0$. For $\mathbf{C}$ the error is initially decreasing, at approximately $\epsilon=10^{-7}$ it reaches a plateau and remains constant for $\epsilon<10^{-7}$ confirming the convergence of $\mathbf{C}^{\epsilon}$ to $\mathbf{C}^0$. From the results shown in Figure \ref{fig7} we conclude that $\mathbf{U}^{\epsilon}$ converges faster than $\mathbf{C}^{\epsilon}$. 
 
 \begin{figure*}[h!]
\centering
\begin{tabular}{ c c }
\includegraphics[scale=0.45]{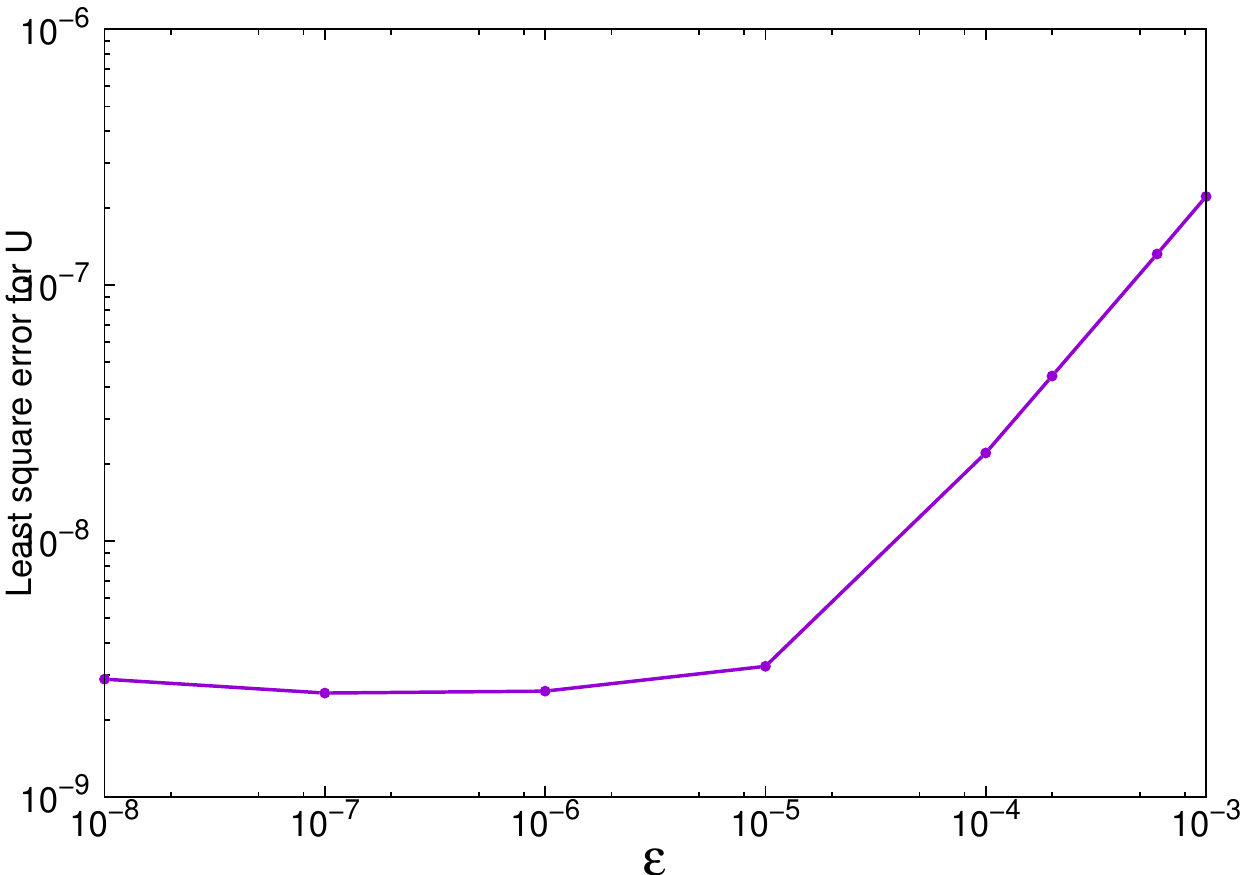} & \includegraphics[scale=0.45]{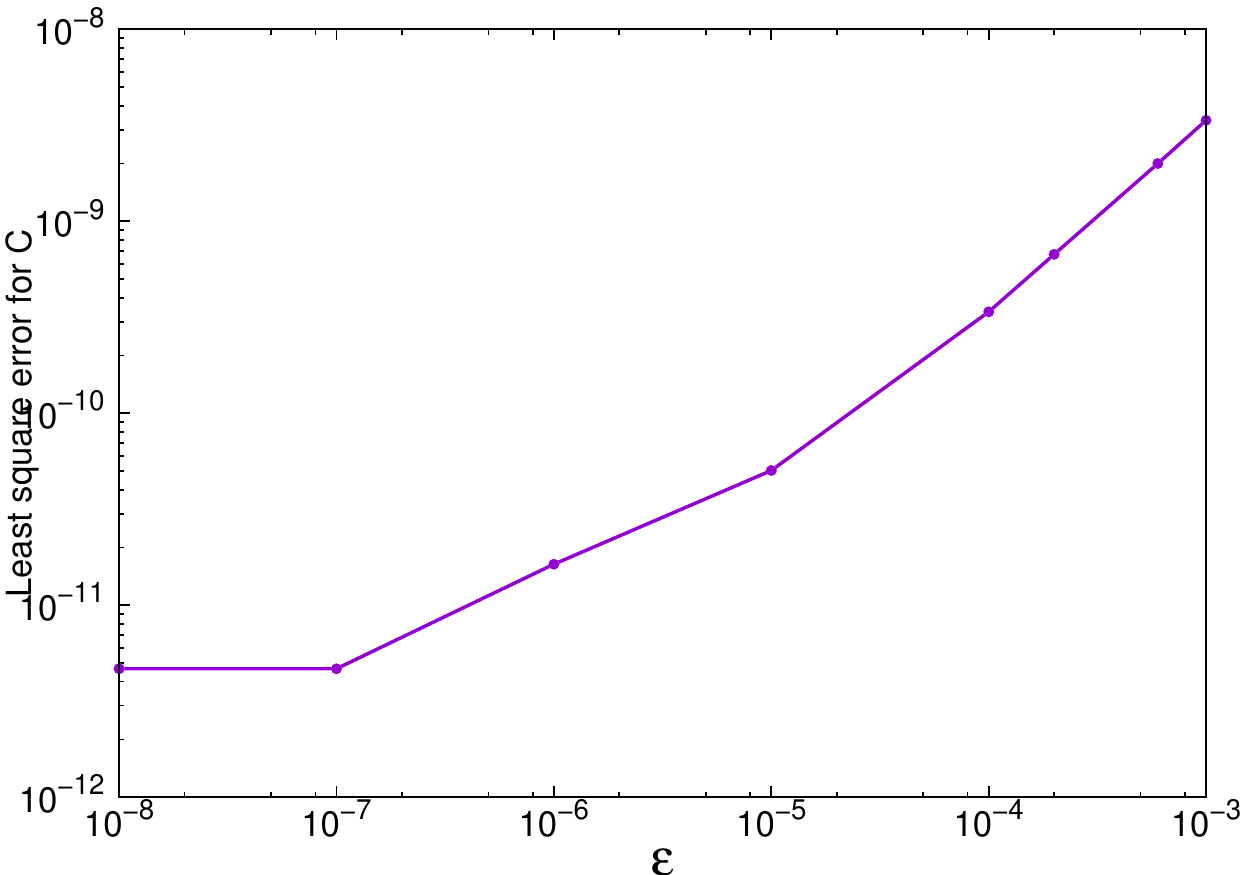} 
\end{tabular}
\caption{Least square errors $E^{\epsilon}_{U}=\frac{1}{256\times 256}\parallel \mathbf{U}^0-\mathbf{U}^{\epsilon}\parallel_{2}$ and $E^{\epsilon}_{C}=\frac{1}{256\times 256}\parallel \mathbf{C}^0-\mathbf{C}^{\epsilon}\parallel_{2}$ at $t=6$ between the solutions of problems (\ref{220}) and (\ref{31}) with $TOL=1e-7$.}
\label{fig7}
\end{figure*}

\subsubsection{Grid refinement}\label{grid}
For a grid refinement study of the full degenerate biofilm model,  we consider the degenerate semi-discrete problem (\ref{220}) with the same initial setting as Section \ref{illust:sec} to account for spatial accuracy explicitly. We refine the grids with $2^{\kappa}$, where $\kappa$ is an integer and compute the least square norm of the difference between two subsequent solutions. The least square errors are defined as $E^{\kappa}_{U}=\frac{\parallel \mathbf{U}_{\kappa}-\mathbf{U}_{\kappa-1} \parallel_{2}}{2^{2(\kappa-1)}}$ and $E^{\kappa}_{C}=\frac{\parallel \mathbf{C}_{\kappa}-\mathbf{C}_{\kappa-1} \parallel_{2}}{2^{2(\kappa-1)}}$.

We compute the errors at two different times $t=2$ and $t=6$. At $t=2$ individual colonies have not merged into one larger colony yet and at $t=6$ the gap between colonies is closed and the biofilm stratifies. We also compute the Elapsed CPU Time (ECPUT) for each cell resolution. 

Tables \ref{table3a}-\ref{table3b} show the least square error between two subsequent solutions at times $t=2$ and $t=6$ and also the elapsed CPU time. We observe a steady decrease in errors $E^{\kappa}_{U}$ and $E^{\kappa}_{C}$ at both times as the grid is refined, indicating the convergence of the method and its reliability to give tolerable smearing around the interface.

Since we have the density value at the center grids, to compute these data interpolation must be used which is a source of error so these data cannot be used to determine the convergence rate accurately.

\begin{table}
\centering
\caption{The error between two consecutive solutions for
grids with $2^{\kappa}\times2^{\kappa}$ and $2^{\kappa-1}\times2^{\kappa-1}$ cell resolution at $t=2$ (before merging) and Elapsed CPU Time (ECPUT).}\label{table3a}
\begin{tabular}{llllllll}
  %\hline
 \hline
  $\kappa$ &~$E^{\kappa}_{U}\vert_{\small{t=2}}$&~$E^{\kappa}_{C}\vert_{\small{t=2}}$&~$ ECPUT\vert_{\small{t=2}}$\\
 \hline
  %5&~\\
  5&~$0.2818501\times10^{-2}$&~$0.1102250\times10^{-2}$&~$0.2403\times 10^2$\\
  6&~$0.5379241\times10^{-3}$&~$0.1528744\times10^{-3}$&~$0.1524\times 10^3$\\
  7&~$0.1908778\times10^{-3}$&~$0.1745382\times10^{-4}$&~$0.2111\times 10^4$\\
  8&~$0.1352093\times10^{-3}$&~$0.1209376\times10^{-4}$&~$0.3424\times 10^5$\\
\hline
  \end{tabular}

  \end{table} 

\begin{table}
\centering
\caption{The error between two consecutive solutions for
grids with $2^{\kappa}\times2^{\kappa}$ and $2^{\kappa-1}\times2^{\kappa-1}$ cell resolution at $t=6$ (after merging) and Elapsed CPU Time (ECPUT).}\label{table3b}
\begin{tabular}{llllllll}
  %\hline
 \hline
  $\kappa$ &~$E^{\kappa}_{U}\vert_{\small{t=6}}$&~$E^{\kappa}_{C}\vert_{\small{t=6}}$&~$ECPUT\vert_{\small{t=6}}$\\
 \hline
  %5&~\\
  5&~$0.2238254\times10^{-2}$&~$0.2687217\times10^{-3}$&~$0.5431\times 10^2$\\
 %6 &~\\
  6&~$0.5709526\times10^{-3}$&~$0.2166271\times10^{-4}$&~$0.4148\times 10^4$\\
  7&~$0.1972021\times10^{-3}$&~$0.2606688\times10^{-5}$&~$0.6121\times 10^4$\\
  8&~$0.1095095\times10^{-3}$&~$0.1136270\times10^{-5}$&~$0.1161\times 10^6$\\
\hline
  \end{tabular}
  \end{table} 
  
\subsubsection{Influence of grid refinement on interface location}\label{interface}

In order to study in more detail the movement of the interface and its location for different choices of $\Delta x$  we change the initial setting. One colony is placed at the middle of the substratum. 
The domain $[0,1]\times[0,1]$ is square and the tolerance of the ROW method is set at $TOL=1e-7$.

To show the location of biofilm/liquid interface we pick the interface point located on the symmetry line of the initial colony,  i.e. the line parallel to the $y$-axis with $x=0.5$. The location of the interface between biofilm and liquid in the $x$-$t$ plane is plotted in Figure \ref{fig6} for various $\Delta x$. Its
movement is described by a piecewise constant increasing function, which represents the  discrete nature of the grid and that the interface location can be determined at most with a $\Delta x$ accuracy. As the grid is refined, the discrete jumps diminish and convergence of the interface location is observed as $\Delta x \rightarrow 0$.

\begin{figure}[h!]
\centering
\includegraphics[width=0.55\textwidth]{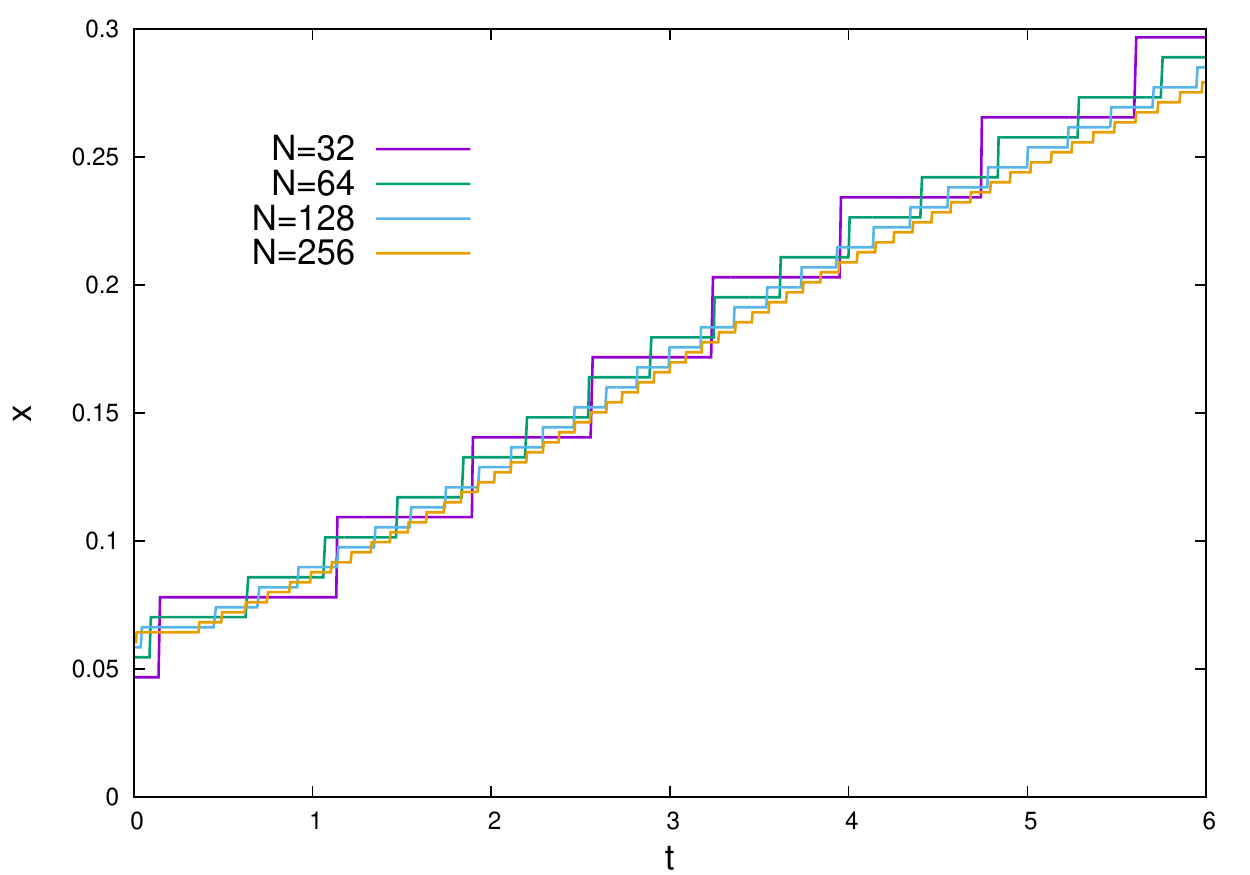}
\caption{Location of the top interface between biofilm and liquid phase for different choices of $\Delta x.$}
\label{fig6}
\end{figure}

\begin{figure*}[h!]
\centering
\begin{tabular}{ c  }
\includegraphics[width=0.55\textwidth]{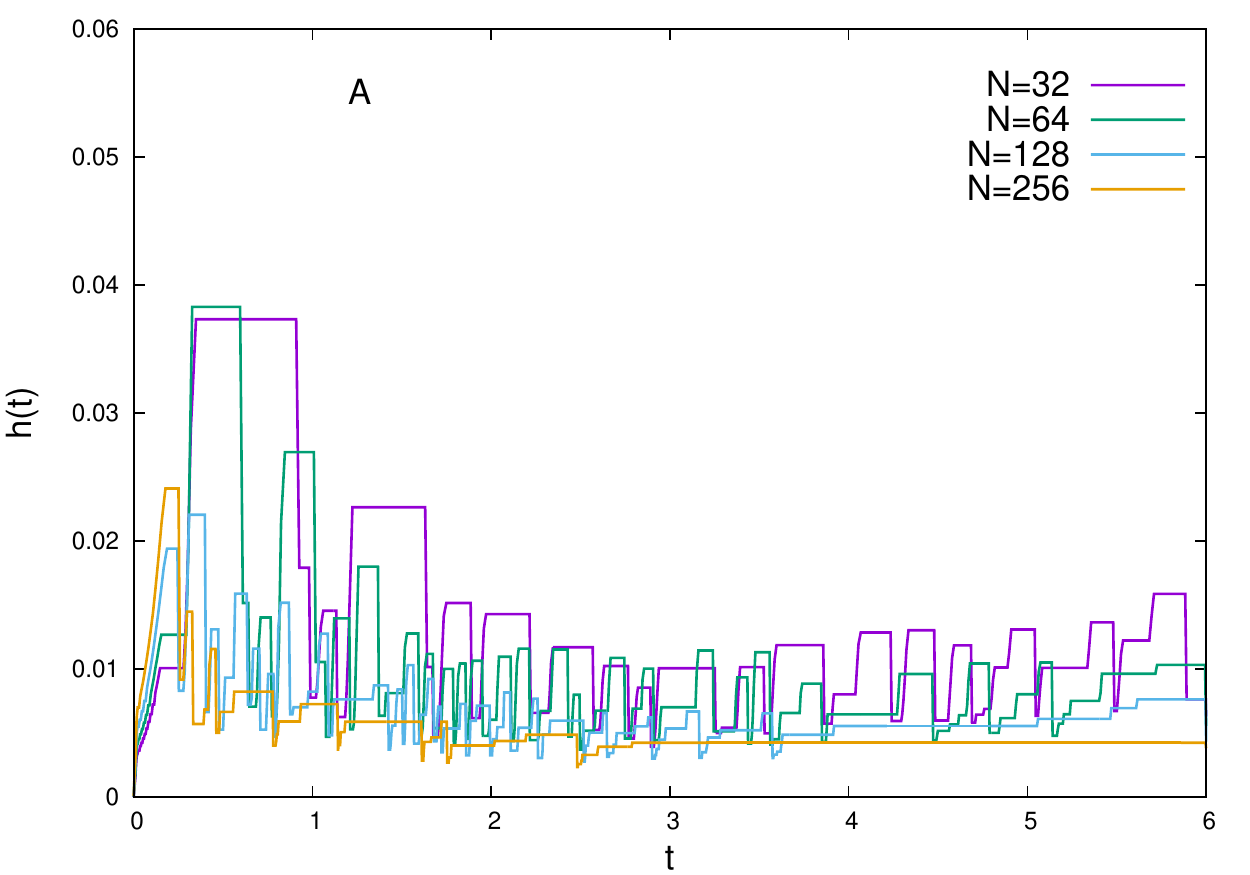}\\
\includegraphics[width=0.55\textwidth]{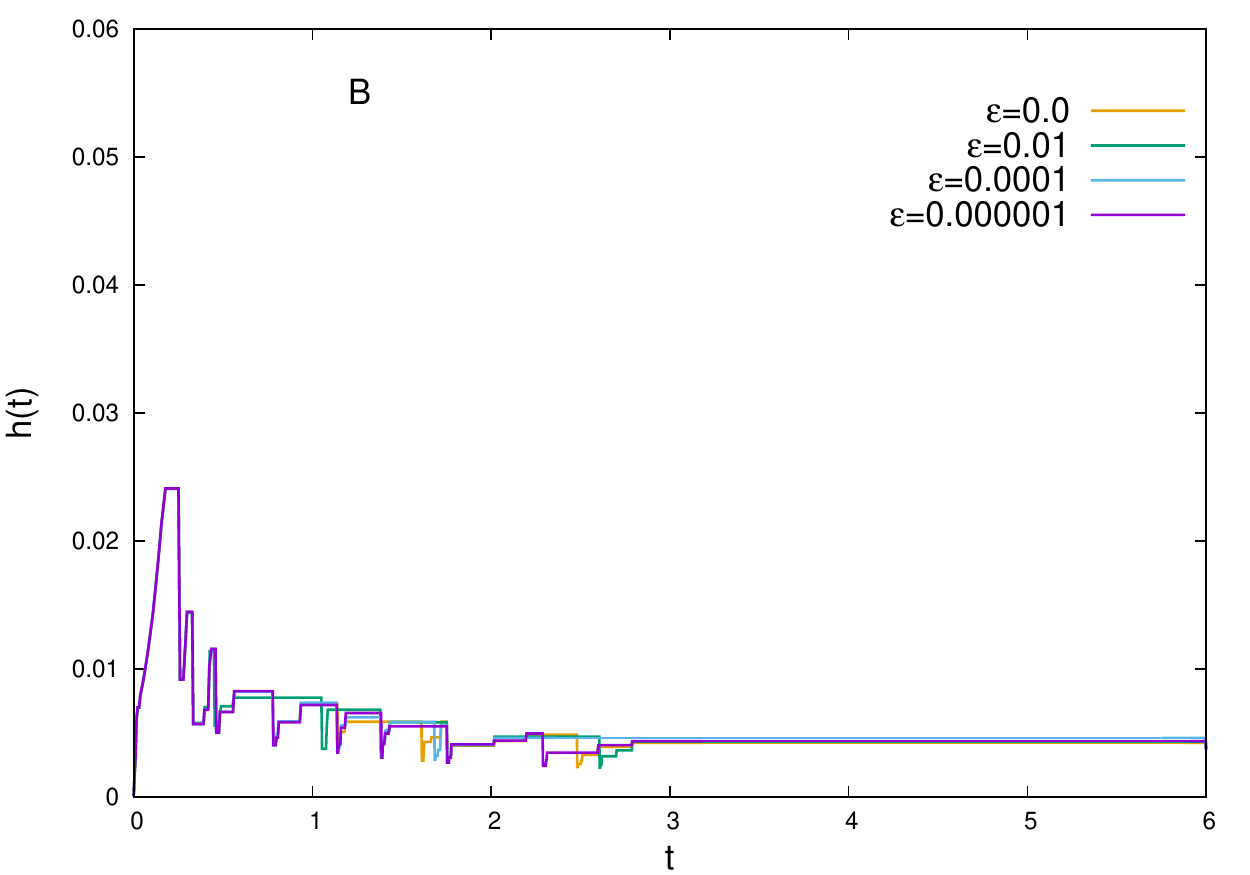}\\
\includegraphics[width=0.55\textwidth]{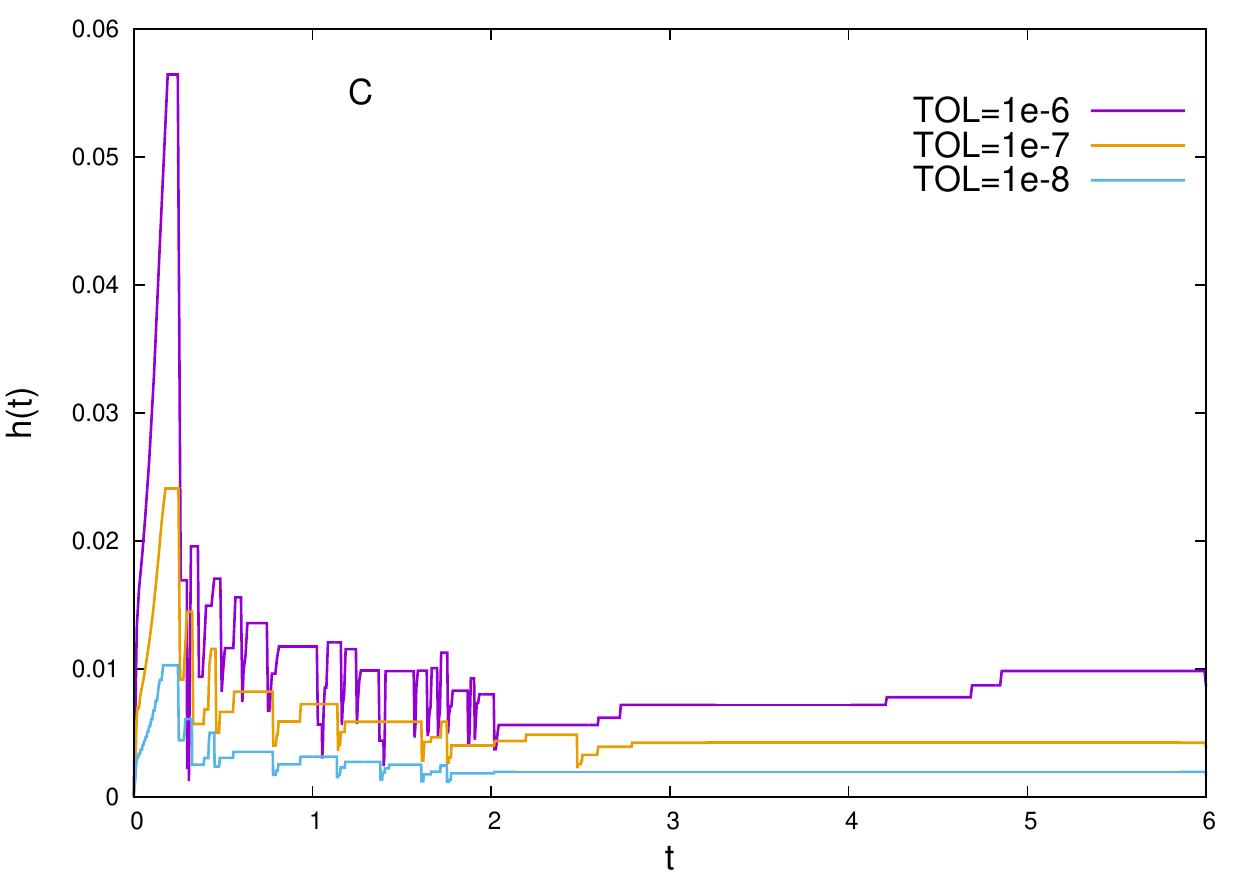}
\end{tabular}
\caption{\textbf{A}. The effect of cell resolution on the value of time-step size $h(t)$ for degenerate problem (\ref{220}) with $TOL=1e-7$, \textbf{B}. The effect of $\epsilon$ on the value of time-step size for regular and degenerate problem (\ref{31}), (\ref{220}) with $N=M=256$ and $TOL=1e-7$, \textbf{C}. The effect of tolerance on time-step size for degenerate problem (\ref{220}) with $N=M=256$.}
\label{fig5}
\end{figure*}

\subsubsection{Effect of method parameters on the time step}

We have shown in Sections \ref{grid} that the accuracy of spatial discretisation increases if grids are refined. We have also shown in section \ref{eps}  how fast the solution of the regularised semi-discrete problem is convergent to the solution of the degenerate semi-discrete problem. 
A third parameter, besides $\Delta x$ and $\epsilon$, that determines the accuracy of ROW method and controls the time-step size $h(t)$ is the  tolerance of the ROW method. Here we study the effect of these three parameters on the time-step size. We consider the same initial setting as in Section \ref{interface}. 

The effect of cell resolution on the value of time-step $h(t)$ for the degenerate semi-discrete problem (\ref{220}) is shown in Figure \ref{fig5}.A. For each cell resolution, initially the time-step size increases, indicating that initial guess could have been chosen larger. After a short plateau phase, $h(t)$ decreases in an oscillatory fashion. The time-step oscillates during the simulation for all spatial-step size except for $\Delta x=\frac{1}{256}$ in which the oscillation stops at approximately $t=2.96$ and $h(t)$ remains constant for the remainder of the simulation up to $t=6$. Figure \ref{fig5}.A also shows that finer grids require finer time steps.

The size of the time-step for the regular semi-discrete problem (\ref{31}) with different values of $\epsilon$ including $\epsilon=0$ is plotted in Figure \ref{fig5}.B, for $M\times N=256\times256$ and $TOL=1e-7$.
The time step $h(t)$ is not strongly sensitive with respect to the value of $\epsilon$. 

The time-step size $h(t)$ for the degenerate problem (\ref{220}) for different values of tolerance of the ROW method is shown in figure \ref{fig5}.C. The time step is sensitive with respect to this parameter, and decreasing the tolerance, i.e. increasing the desired accuracy, in the ROW method decreases the value of time-step size. For all three values of tolerance, initially $h(t)$ increases but the deviation from the initial guess  is considerable for $TOL=1e-6$. After the first short plateau phase, $h(t)$ decreases in oscillatory fashion for all case studies. For $TOL=1e-6$ this oscillation can be seen during the simulation up to $t=6$. For $TOL=1e-7$ oscillatory behaviour stops at approximately $t=2.96$ while this time for $TOL=1e-8$ is approximately $t=1.96$. We conclude that for smaller values of tolerance, $h(t)$ shows less oscillation and becomes stable faster.

\subsection{Discussion of the Numerical Method {\it vis-a-vis} the Semi-Implicit Method with Fixed Time Steps}

The spatial discretisation described in Section \ref{sec2.2} is the same that underlies the semi-implict method of \cite{ED:2007,SH:2012} and related schemes, which  have been predominantly used in applications. 
The resulting semi-discrete equation is a large, highly nonlinear and singular ordinary differential equation system. The difference between the methods lies in the approach to solving this initial value problem.
 We showed in Section \ref{sec3}, via a regularisation analysis, that the solutions of this system indeed possess sufficient smoothness to employ higher order methods, such as embedded Rosenbrock-Wanner methods (which we used; other choices would have been possible), despite the low regularity of the biologically relevant solutions of the underlying partial differential equation, which display gradient blow-up along an interface moving with finite speed.
Compared to the semi-implicit method, which requires in each time step the solution of 2 sparse linear systems of size $(NM)^2$, the method that we used here is more expensive per time step, as it it requires the solution of 4 sparse linear systems of size $4(NM)^2$, and the evaluation of the Jacobian.  On the other hand, it has better  convergence properties, allowing larger time steps to achieve the same accuracy.  In applications, the semi-implicit method and is relatives that use fix time-steps  require sometimes extensive {\it a priori} explorations by the user to find for a particular simulation experiment a suitable time step that gives an acceptable effort/accuracy trade-off. In the method presented here, on the other hand, accuracy is guaranteed by an automatic, adaptive time-step choice.

\bigskip
%%%%%%%%%%%%%%%%%%%%%%%%%%%%%%%%%%%%%%%%%%%%%%%%%%%%%%%%%%%%%%%%%%%%%%%%%%%%%%%%%%%%%%%%%%%%%%%%

\section{Quorum sensing induction in isolated colonies}\label{DS}

\subsection{Mathematical model}\label{model} 
 
The mathematical model that we propose to study quorum sensing induction in the growing biofilm is an extension of (\ref{21}), that accounts for the autoinducer, or signal molecule concentration $s$.  Following  \cite{EHK:2015} for autoinducer production, the modified model reads
\begin{equation}\left\{%
\begin{array}{lll}
\frac{\partial u}{\partial t}&=&\nabla(D(u)\nabla u)+\frac{c}{K_{U}+c}u-ku\\\\
\frac{\partial c}{\partial t}&=&d_c\Delta c-\frac{\nu_U c}{K_{U}+c}u\\\\
\frac{\partial s}{\partial t}&=&d_s\Delta s+(\alpha+\beta\frac{s^m}{1+s^m})u-\psi s, \label{531}
\end{array}%
\right.
\end{equation}
where $D(u)$ is defined in (\ref{22}) and $d_c$ and $d_s$ are the constant diffusion coefficients for nutrient and signal molecules respectively. In equation (\ref{531}) autoinducers are produced at a base rate $\alpha$ if the local autoinducer concentrations is small relative to induction threshold, which is scaled here to be 1. If the autoinducer concentration exceeds the induction threshold, the production rate will increase to $\alpha+\beta$. The transition is described by a Hill function with exponent $m$.
Quorum sensing parameters $\alpha,~\beta,~\psi$ and $m$ are chosen in the range of data in \cite{FKH:2011}.

For our simulations we consider a rectangular computational domain $[0,L]\times[0,H]$. The boundary condition for biomass $u$ is the homogeneous Neumann condition everywhere. Note that, as long as the biofilm interface does not reach the domain, the biomass density satisfies both the homogeneous Dirichlet and Neumann conditions.  For the concentrations $c$ and $s$ we pose the homogeneous Neumann condition at the lateral sides of the domain and at the substratum, and Robin condition at the top, where nutrients are added and autoinducer signals diffuse out of the domain.

Thus, the specific boundary condition for problem (\ref{531}) is defined as:
\begin{equation} \label{532}
\left\{%
\begin{array}{lll}
 \partial_nu=0 &\mbox{for} & x\in \partial\Omega\\
 \partial_nc=\partial_ns=0 & \mbox{for} & x_1=0, x_1=L, x_2=0\\
 c+\lambda\partial_n c=1 &\mbox{for} & x_2=H.\\
 s+\lambda\partial_n s=0 &\mbox{for} & x_2=H.
\end{array}%
\right.
\end{equation}
where $\lambda$ is the concentration boundary layer thickness.

\begin{table}
\centering
\caption{Dimensionless model parameters of system (\ref{531}).} \label{table531}
\begin{tabular}{l*{2}{c}r}
  %\hline
  \hline
  Parameter&Symbol&Value&Source\\
  \hline
  Biomass decay rate &~$k$&~$0.67$&~$RE$\\
  Monod half saturation&~$K_U$&~$0.13$&~$RE$\\
  substrate uptake rate&~$\nu_U$&~$530$&~$RE$\\
  Nutrient diffusion coefficient&~$d_c$&~$33$&~$RE$\\
  biomass motility coefficient&~$\delta$&~$10^{-8}$&~$RE$\\
  AHL diffusion coefficient&~$d_s$&~$16.5$&~$-$\\
  dimerization~exponent&~$m$&~$2.5$&~$-$\\
  AHL production rate&~$\alpha$&~$4500$&~$-$\\
  AHL production rate&~$\beta$&~$45000$&~$-$\\
  AHL decay rate&~$\psi$&~$0.02$&~$-$\\
  \hline
  \multicolumn{3}{l}{\textsuperscript{}\footnotesize{Reference: RE=\cite{RE:2014}}}
\end{tabular}
\end{table}

Initially the system is inoculated by one semi-spherical colony placed at the center of the substratum on the bottom boundary.  The simulation is terminated when the signal concentration everywhere in the domain exceed switching threshold, i.e. the first time at which $s(t,x,y)\geq 1$ for all $(x,y)\in \Omega$.
 
We investigate how the domain size affects the time to induction. We keep the height of the domain constant at $H=1$ and change the length of the computational domain in the range $ 1 \leq  L \leq 2.5$. For each $L$ we measure the following lumped output parameters:

\begin{itemize}
\item{$T_1$:} The first time at which $s(t,x,y)\geq 1$ for some  $x\in\Omega$
\item{$T_2$:} The first time at which the average signal concentration in the colony exceeds switching threshold, i.e. when $\int_{\Omega_2}s(t,x,y)dxdy \geq \int_{\Omega_2}dxdy$
\item{$T_3$:} The first time at which $s(t,x,y)\geq 1$ for some $(x,y)\in \Omega_2(t)$
\item{$T_4$:} The first time at which $s(t,x,y)\geq 1$ for all $(x,y)\in\Omega$
\end{itemize}

We also compute the total biomass in the system and its value at each time $T_i,~i=1,...,4$ to describe the relation between biofilm density and quorum sensing induction. The value of total biomass and its value at each time $T_i,~i=1,...,4$ are defined as
\begin{eqnarray}
M(t)&=&\int_{\Omega}u(t,x,y)dxdy\nonumber\\
M_{\small{total},i} &=& M(T_i),\quad i=1,...,4,\label{533}
\end{eqnarray}

The parameter values used in the simulations and their definition are shown in Table \ref{table531}. The internal tolerance is set in all simulations to $TOL=1e-7$, spatial discretization is such that $\Delta x =1/256$.

\subsection{Results}\label{illustration}

\begin{figure*}
\centering
\begin{tabular}{ c c  }
\includegraphics[width=0.48\textwidth]{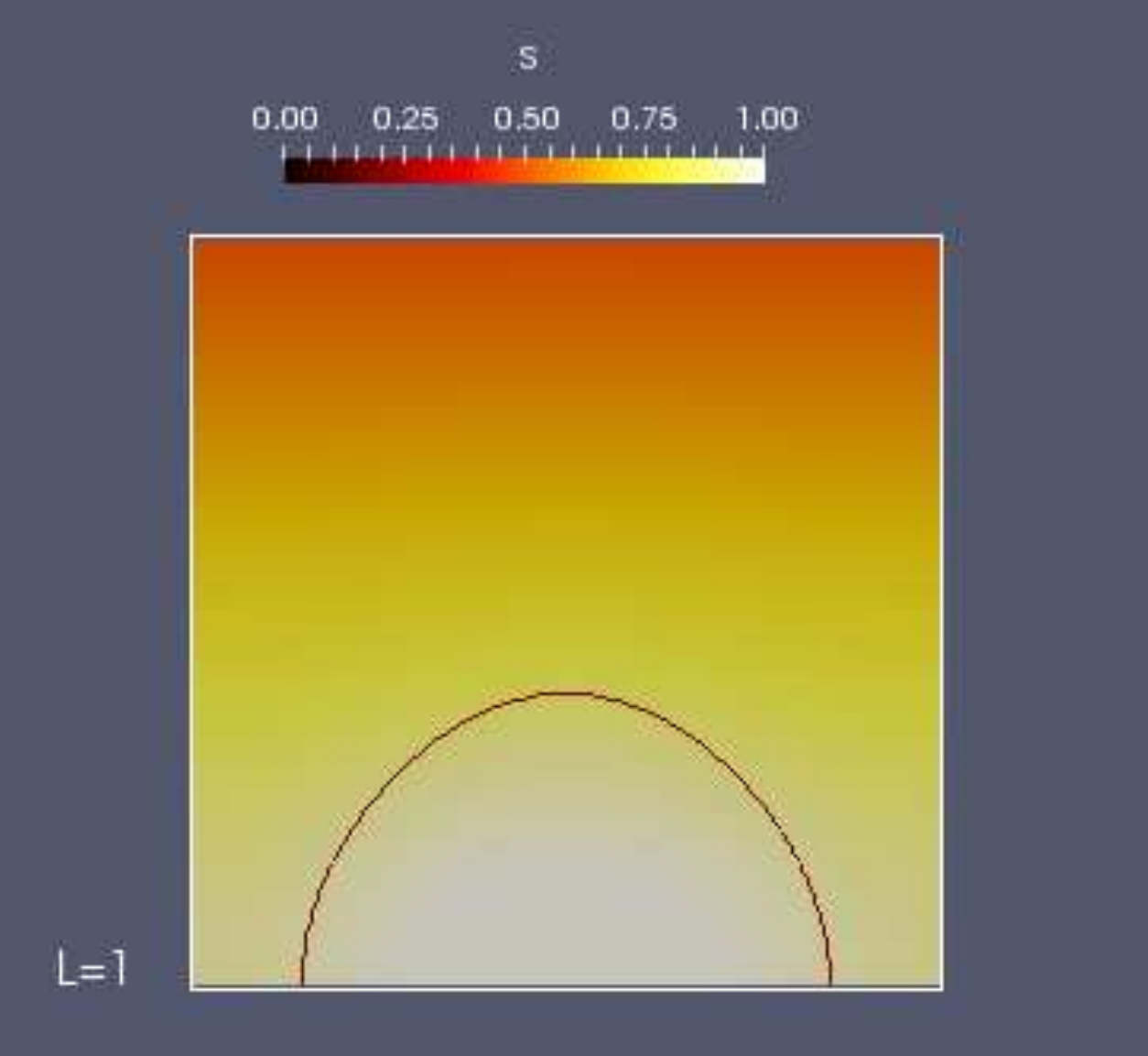}&\includegraphics[width=0.48\textwidth]{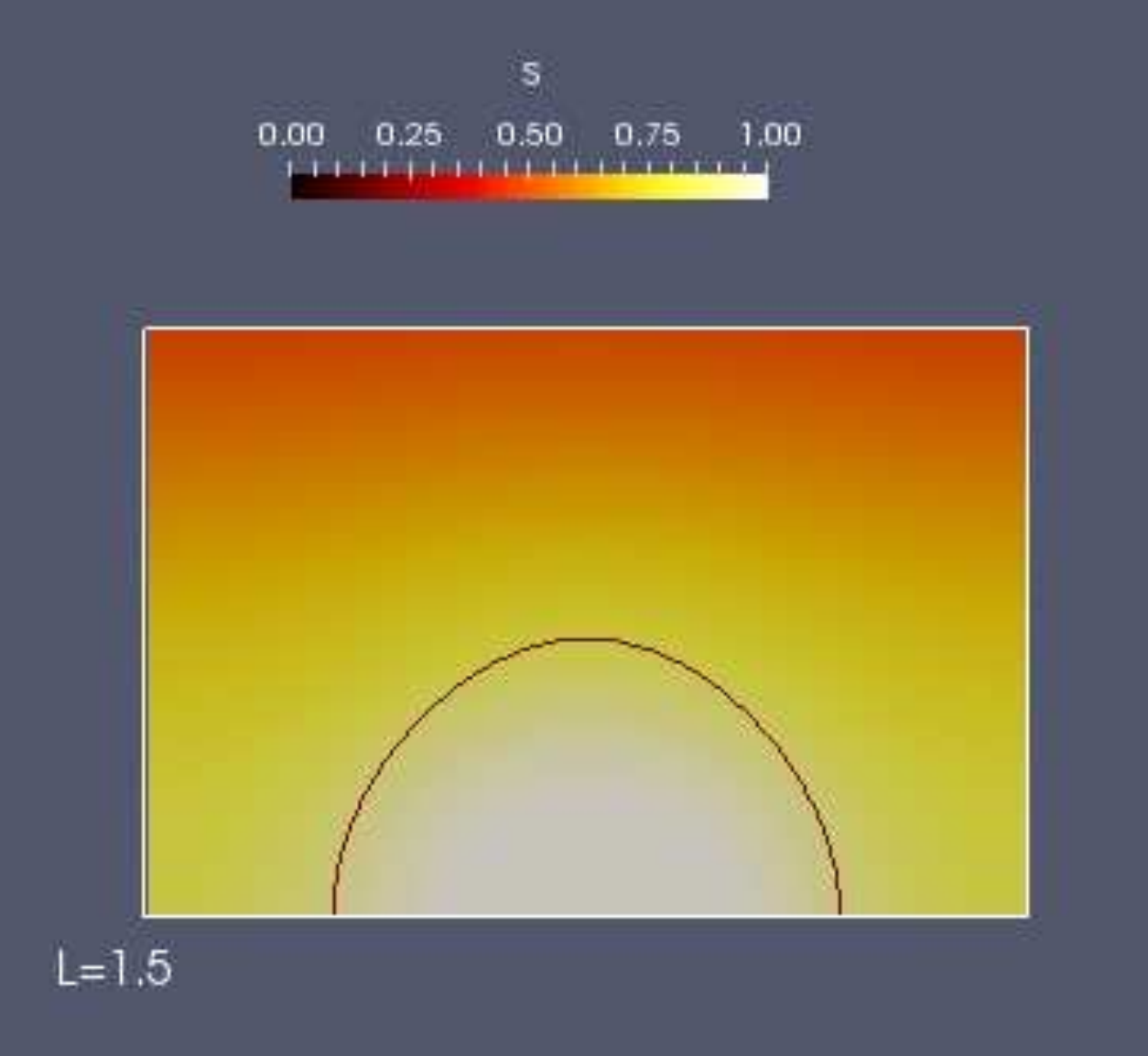}\\
\includegraphics[width=0.48\textwidth]{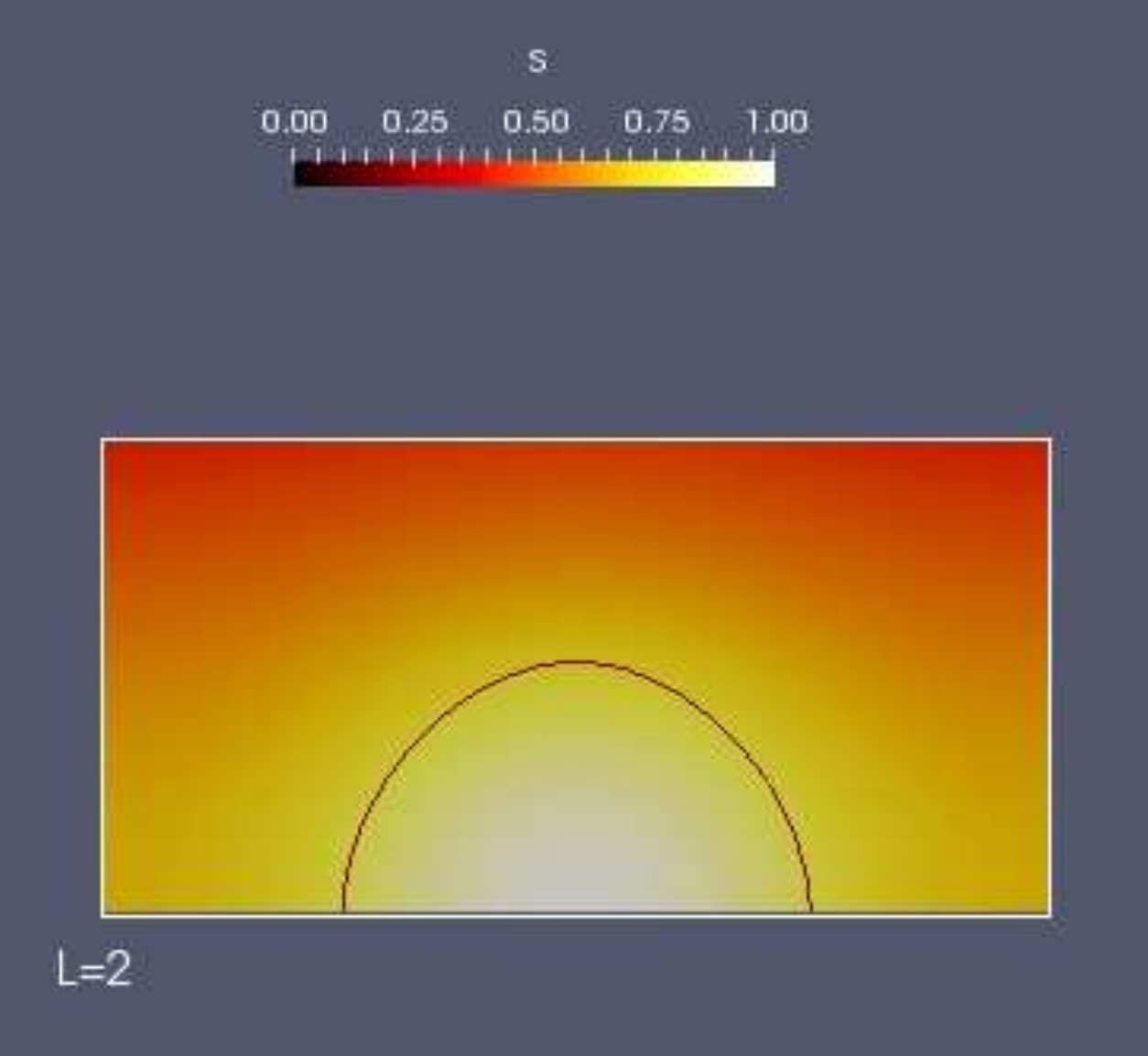}&\includegraphics[width=0.48\textwidth]{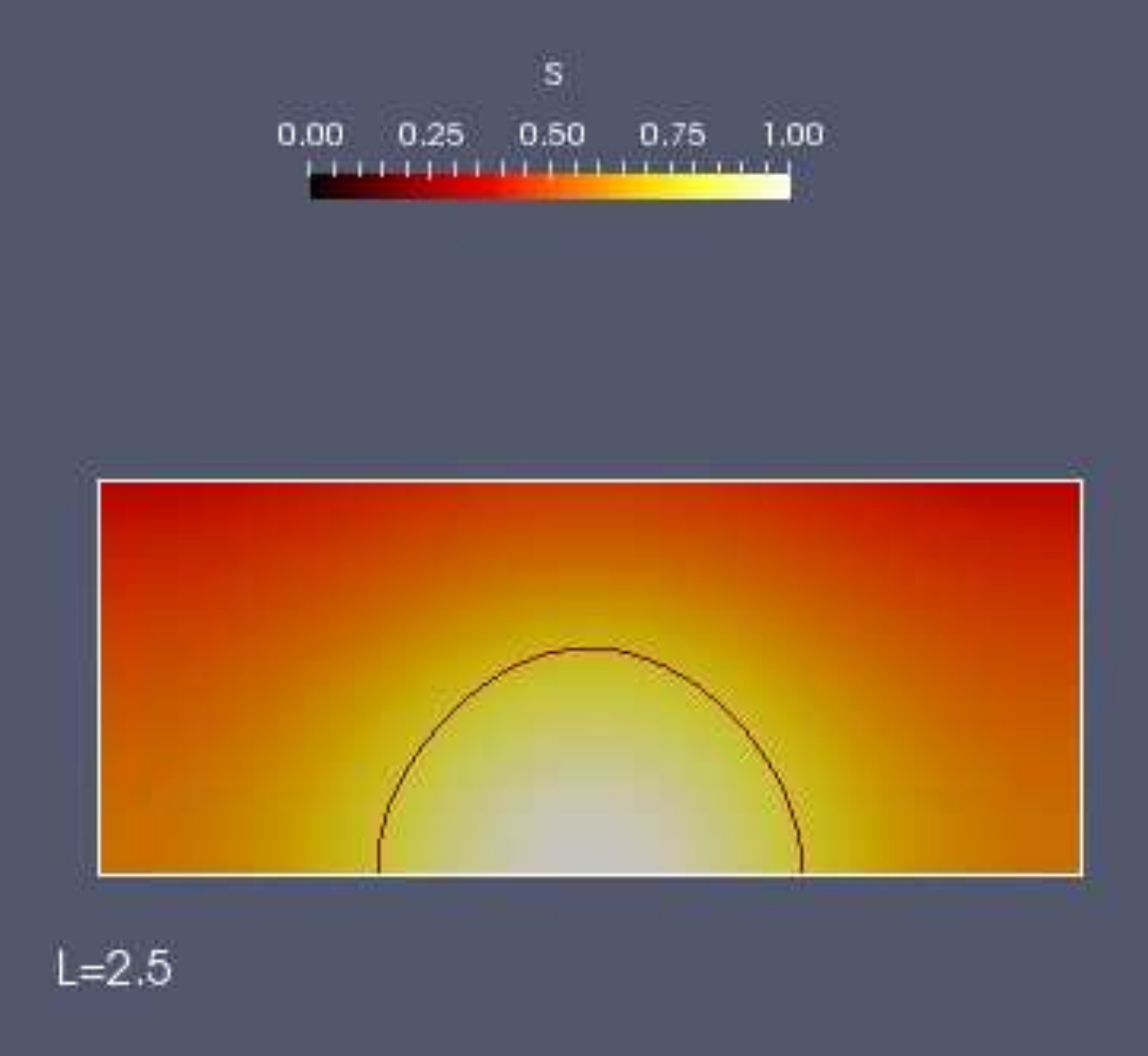}
\end{tabular}
\caption{AHL distribution in the system for different system sizes with various length at $t=T_1$. The biofilm interface is indicated in white.}
\label{figdiff2}
\end{figure*}

Figure \ref{figdiff2} shows the autoinducer concentration $s$ and the interface between biofilm and surrounding liquid at $T_1$ for four different system sizes with $L=1,1.5,2,2.5$. 
In all cases the biofilm colony remains approximately semi-spherical, but we observe that the larger $L$, the larger is the biofilm colony at the time that the
signal concentration reaches 1 first (see also Figure \ref{figdiff}.C). This is a consequence of increased substrate supply as the domain length is increased, and hence the length of the boundary through which nutrient is added.  The maximum signal concentration is attained in the center of the colony; it decreases from there toward the lateral boundaries of the domain and toward the top boundary. The concentration at the top boundary is lower, as this is where autoinducers are removed by diffusion due to the boundary conditions (\ref{532}). The larger the system length $L$, the steeper are the autoinducer gradients and the lower is the autoinducer concentration at the boundaries. This is a consequence of signals only being produced in the colony. For the smallest system length $L=1$ the autoinducer concentration appears stratified across the domain, whereas in the the highest systems length $L=2.5$ it appears to be much larger in the colony than in the aqueous phase.

\begin{figure*}[h!]
\centering
\begin{tabular}{ c  }
\includegraphics[width=0.55\textwidth]{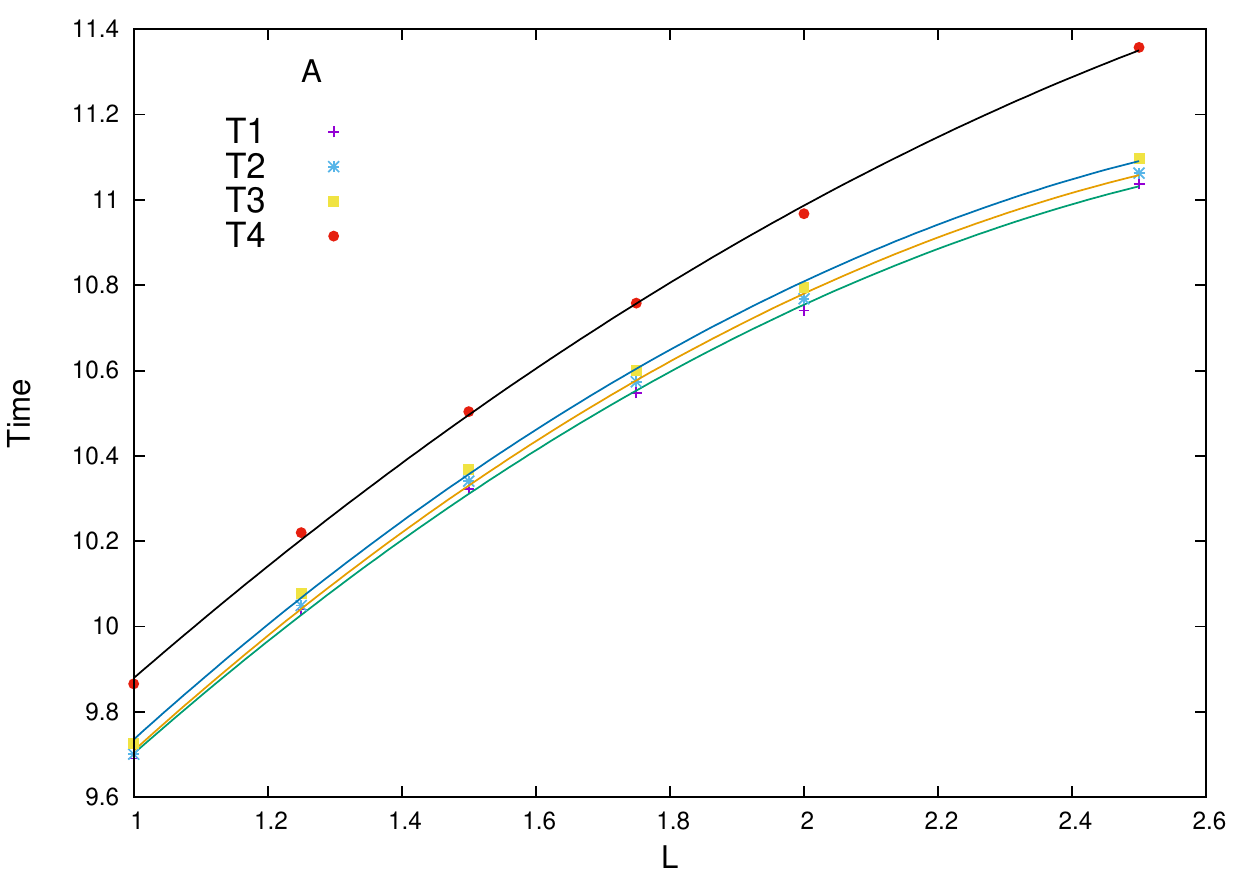}\\
\includegraphics[width=0.55\textwidth]{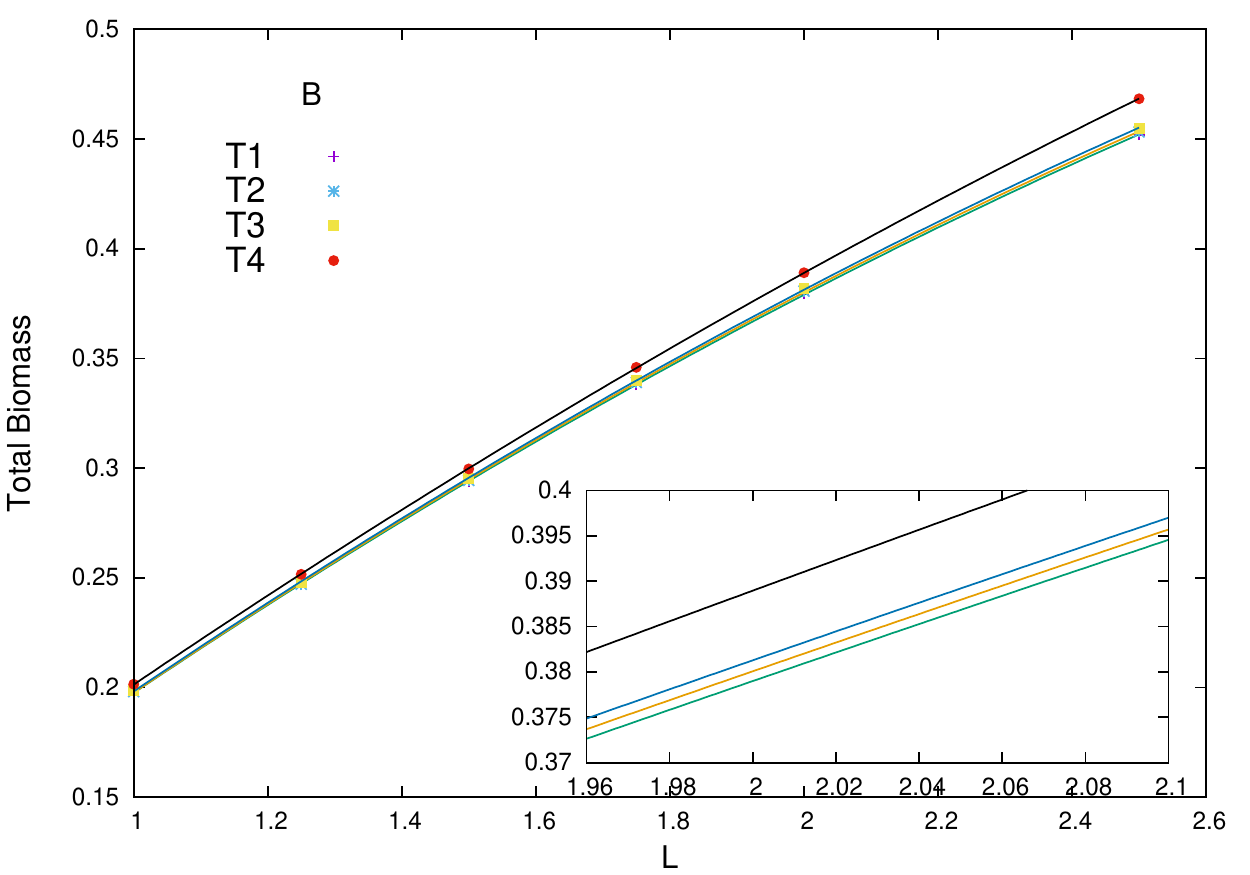}\\
\includegraphics[width=0.55\textwidth]{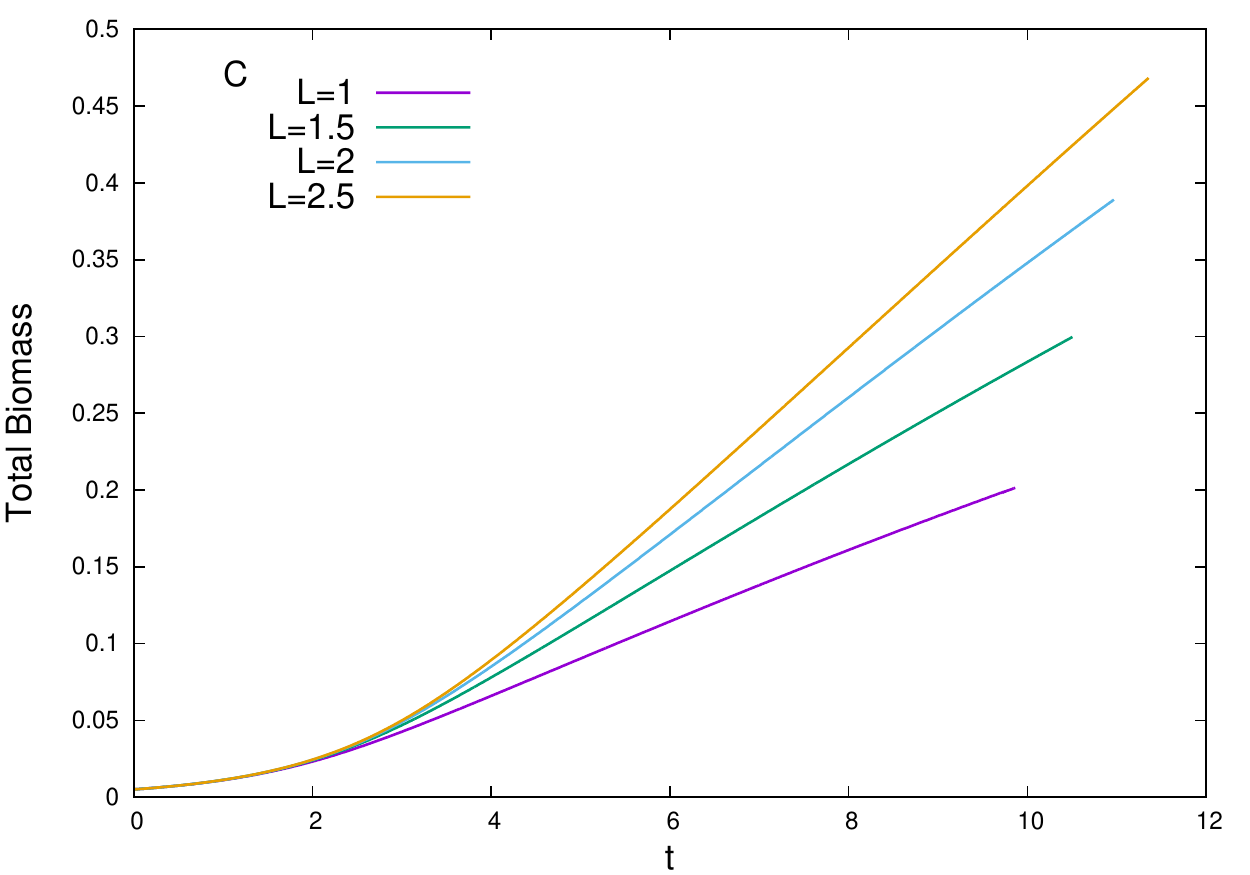}
\end{tabular}
\caption{\textbf{A}. The value of computed \textit{lumped} parameters $T_i,~i=1,...,4,$ vs lengths. \textbf{B}. The value of total biomass vs lengths at each defined time $T_i,~i=1,...,4$. \textbf{C}. Time course of total biomass for various lengths. Computational domain is discretised uniformly with $\triangle x=\frac{1}{256}$ and tolerance of ROW method is set at $TOL=1e-7$}
\label{figdiff}
\end{figure*}

\begin{table}
\centering
\caption{The values of the coefficients of fitted quadratic function for $T_i$ and $M_{total,i}$ and the error between the exact and approximated data.} \label{tablefitT}

\begin{tabular}{llllll}
  %\hline
 \hline
  $T_i(L)$ &\qquad$\alpha_1$&\qquad$\beta_1$&\qquad$\gamma_1$&\qquad$RMSError_T$\\
\hline\\
  $T_1(L)$&~$-0.331332$&~$2.04645$&~$7.98685$&~\quad$0.014949$\\
  $T_2(L)$&~$-0.343759$&~$2.10267$&~$7.95014$&~\quad$0.011887$\\
  $T_3(L)$&~$-0.339661$&~$2.09279$&~$7.98219$&~\quad$0.0138615$\\
  $T_4(L)$&~$-0.253235$&~$1.86745 $&~$8.26481$&~\quad$0.0175622$\\
  \hline
  \end{tabular}
\vskip1truecm
\begin{tabular}{llllll}
  %\hline
 \hline
  $M_{\small{total},i}(L)$ &~\qquad$\alpha_2$&~\qquad$\beta_2$&~\qquad$\gamma_2$&~$RMSError_M$\\
\hline\\
  $M_{\small{total},1}(L)$&~$-0.0234576$&~$0.251837$&~$-0.0308511$&~$0.000486687$\\
  $M_{\small{total},2}(L)$&~$-0.0237583$&~$0.253714$&~$-0.0323236$&~$0.000612262$\\
  $M_{\small{total},3}(L)$&~$-0.0235238$&~$0.253647$&~$-0.0319177$&~$0.000541766$\\
  $M_{\small{total},4}(L)$&~$-0.0191422 $&~$0.24508$&~$-0.0246319$&~$0.000274989$\\
  \hline
  \end{tabular}
\end{table} 

The value of parameters $T_i,~i=1,...,4$ and also the value  $M_{total,i}$ of total biomass at each $T_i$  for computational domains with length $L = 1, 1.25, 1.5, 1.75, 2, 2.5$ are shown in Figure \ref{figdiff}.   We observe in Figure \ref{figdiff}.A that increasing the length of the computational domain increases the value of $T_i$'s at which QS induction occurs. In fact, by increasing $L$ more autoinducers transport away from the producer. Hence, more biomass is needed to have induction, therefore the time at which QS induction occurs increases. We also observe that times $T_i,~i=1,2,3$ are very similar but there is a considerable difference between them and $T_4 $ especially in the larger domains. 

The values of total biomass at each time $T_i, i=1,...,4$ are shown in Figure \ref{figdiff}.B. For each time $T_i,~i=1,...,4$, increasing the length of the domain increases the total amount of biomass. Increasing the domain length increases mass transfer thus more biomass is needed for differentiation. Mirroring the differences between $T_{1,2,3}$ and $T_4$, Figure \ref{figdiff}.B shows considerable difference between total biomass at $T_i,~i=1,2,3$ and its value at $T_4$.  
The time course of total biomass $M(t)$ for $L=1,1.5,2,2.5$ is plotted in Figure \ref{figdiff}.C to show the effect of the length of the domain on the required biomass for QS induction vs time. We observe that initially the values of total biomass in all four systems are the same and deviation emerges at $T\simeq 2.2$. We also observe that $M(t)$ in the system with $L=2.5$ has the maximum value indicating the maximum mass transfer occurs in the system with largest domain. Furthermore, it can be seen that the value of $T_4$ which is in fact the time at which simulation stops is increasing by increasing the length of computational domain as we have shown in Figure \ref{figdiff}.A. 

The simulation results in Figure \ref{figdiff}.A,B, suggest that  $T_i$'s and $M_{total,i}$ change with respect to the system length quadratically. This is confirmed by fitting a quadratic function to our data points,  $T_{i}(L)=\alpha_1 L^2+\beta_1 L+\gamma_1$ and $M_{\small{total},i}(L)=\alpha_2 L^2+\beta_2 L+\gamma_2$, which are also included in the plots. 
The coefficients $\alpha_1,~\beta_1,~\gamma_1$ and $\alpha_2,~\beta_2,~\gamma_2$  are given in Table \ref{tablefitT}. Also given is the root mean squared error defined as $RMSError_T=\sqrt{\frac{\sum_{i=1}^{n}(\hat{T}_i-T_i)^{2}}{n}}$ for $T_i$ and $RMSError_M=\sqrt{\frac{\sum_{i=1}^{n}(\hat{M}_i-M_i)^{2}}{n}}$ for $M_{total,i}$ in which $n$ is the number of collected data and $\hat{T}_i$ and $\hat{M}_i$ are the predicted values of $T_i$ and $M_{total,i}$. We observe that $RMSError_M < RMSError_T$ indicating the error between the predicted value of $M_{\small{total},i}$ and obtained data is less than the error between the predicted and real values of $T_i$.

\subsection{Discussion of the results of QS induction in isolated colonies}

In agreement with earlier studies \cite{CH:2003,FKH:2011,Ward:2001} our results show that upregulation is rapid. In our simulations, the time between the signal concentration first reaching the switching threshold somewhere in the colony and until the switching threshold is exceeded in the colony everywhere is very short, about 1\% of the induction time.

Our simulations suggest that the environment and the physical conditions there affect the time to induction for a single isolated colony. This is to  a large extent determined by mass transfer of signals from the colony into the surrounding aqueous phase, and transport of signal molecules there. In our simulations we have neglected abiotic decay of signal molecules in the aqueous phase, but it is to conjecture that including such an effect would exacerbate this phenomenon. This suggests that upregulation of an individual colony is not a mere function of colony size, in accordance with \cite{HKM:2007}.  A related, but not directly comparable study in this regard is \cite{CH:2003}, where a one-dimensional model was used to determine the minimum thickness required for a homogeneous biofilm layer to induce. In that study the aqueous phase was not explicitly considered, but the diffusive flux of signal molecules from the colony in direction of the bulk phase is specified as a boundary condition. This, in our setting would correspond to changing the depth of the cavity, which we kept constant. In contrast to this, we varied the lateral extension of the domain, showing that also horizontal mass transfer, orthogonal to the main diffusion direction matters.  The quadratic dependence of induction time and biomass on the system size also suggest that diffusion is a dominant factor in the process.

On the other hand, if we understand the lateral homogeneous Neumann conditions as symmetry conditions instead of as no-flux conditions, then we can interpret our system also as an infinite array of equidistantly spaced identical colonies. Our system then models a biofilm community consisting of many individual colonies. System length $L$ is then the distance between the centers of 2 consecutive colonies, i.e. the smaller $L$, the more biomass in the biofilm. From the viewpoint of the entire biofilm community our results confirm the traditional interpretation of quorum sensing, i.e. induction takes place when the population reaches a certain size. We can reconcile this with the findings discussed above by the observation that the horizontal signal fluxes between two neighbouring colonies cancel each other, i.e. the net flux between neighbouring colonies is zero. Again this confirms the important role of signal diffusion.

%%%%%%%%%%%%%%%%%%%%%%%%%%%%%%%%%%%%%%%%%%%%%%%%%%%%%%%%%

\bigskip

\section{Conclusion}\label{conclusion}

Our primary objective was to study a time-adaptive, error-controlled  numerical solution strategy for a non-linear degenerate diffusion system that arises in biofilm modelling. This system has two non-linear effects (i) porous medium degeneracy for $u=0$ and (ii) super-diffusion singularity for $u=1$, which  make this problem difficult to treat numerically and analytically. 

Biologically relevant solutions of this system are characterized by an interface that propagates with finite speed, along which the biomass gradient blows up. 
Due to the low reglarity of these solutions, low order time intergation methods have been used previously for the time integration of the semi-discrete system that is obtained after spatial discretisation. These methods mostly have been explicit or semi-implicit, with small enough fixed time step to avoid numerical difficulties (e.g. overshooting of the singularity), and without error control capabilites.
The important question that we answered is whether we can also use a higher order time integration method to solve the semi-discrete approximation of this highly non-linear PDE.
By regularisation we showed that the spatially discretised problem indeed satisfies a Lipschitz condition, i.e. has classical smooth solutions, which never reach the singular point.  This allows the application of time-adaptive, error controlled time integration techniques, such as embedded Rosenbrock-Wanner methods, to simulate these highly nonlinear problems with singular and degenrate diffusion.

To demonstrate the utility of the numerical method we carried out a simulation experiment of quorum sensing induction of a biofilm colony in a protected niche. We found that induction time does not depend only on colony size, but also on the properties of the environment that affect transport of signal molecules, e.g niche size in our case.

\textbf{Acknowledgements}
This study has been financially supported by the Natural Science and Engineering Research Council of Canada (NSERC): MG holds a PGS-D graduate scholarship, HJE a Discovery Grant, the
equipment was purchased with a Research Tools and Infrastructure Grant (HJE).

%%%%%%%%%%%%%%%%%%%%%%%%%%%%%%%%%%%%%%%%%%%%%%%%%%%%%%%%%%%%%%%%%%%%

\end{document}